\documentclass[reqno,10pt]{amsart}
\usepackage[utf8]{inputenc}
\usepackage[T1]{fontenc}
\usepackage{indentfirst}
\usepackage{textcomp,lscape,rotating}
\usepackage[english]{babel}
\usepackage{enumerate}
\usepackage{amsmath,amsfonts,amssymb,amsopn,amscd,amsthm}
\usepackage{mathrsfs,marvosym}
\usepackage{stmaryrd}
\usepackage{graphicx}
\usepackage[enableskew]{youngtab}
\usepackage[dvipsnames]{xcolor}
\usepackage[colorlinks=true,citecolor=DarkOrchid,linkcolor=NavyBlue]{hyperref}

\newcommand{\E}{\mathrm{e}}
\newcommand{\I}{\mathrm{i}}
\newcommand{\N}{\mathbb{N}}
\newcommand{\Z}{\mathbb{Z}}

\newcommand{\R}{\mathbb{R}}
\newcommand{\C}{\mathbb{C}}
\newcommand{\Hq}{\mathbb{H}}
\newcommand{\Tor}{\mathbb{T}}
\newcommand{\esper}{\mathbb{E}}
\newcommand{\proba}{\mathbb{P}}
\newcommand{\Var}{\mathrm{Var}}
\newcommand{\card}{\mathrm{card}\,} 
\newcommand{\GL}{\mathrm{GL}}
\newcommand{\SL}{\mathrm{SL}}
\newcommand{\SU}{\mathrm{SU}}
\newcommand{\SO}{\mathrm{SO}}
\newcommand{\unit}{\mathrm{U}}
\newcommand{\SP}{\mathrm{Sp}}
\newcommand{\Gra}{\mathrm{Gr}}
\newcommand{\glie}{\mathfrak{g}}
\newcommand{\klie}{\mathfrak{k}}

\newcommand{\tlie}{\mathfrak{t}}
\newcommand{\sulie}{\mathfrak{su}}
\newcommand{\solie}{\mathfrak{so}}
\newcommand{\splie}{\mathfrak{sp}}
\newcommand{\sym}{\mathfrak{S}}
\newcommand{\ym}{\mathfrak{Y}}

\newcommand{\matr}{\mathrm{M}}
\newcommand{\vand}{\mathrm{V}}
\newcommand{\eps}{\varepsilon}
\newcommand{\lle}{\left[\!\left[} 
\newcommand{\rre}{\right]\!\right]} 
\newcommand{\id}{\mathrm{id}}
\newcommand{\tr}{\mathrm{tr}}
\newcommand{\leb}{\mathscr{L}}
\newcommand{\diag}{\mathrm{diag}}
\newcommand{\hendo}{\mathrm{End}}
\newcommand{\hatG}{\widehat{G}}
\newcommand{\hatchi}{\widehat{\chi}}
\newcommand{\dtv}{d_{\mathrm{TV}}}

\newcommand{\scal}[2]{\left\langle #1\vphantom{#2}\,\right |\left.#2 \vphantom{#1}\right\rangle}
\newcommand{\figcap}[2]{\begin{figure}[ht] \begin{center} {\footnotesize{#1}} \caption{#2} \end{center} \end{figure}}

\newcommand{\comment}[1]{}
\numberwithin{equation}{section}

\newtheorem{theorem}{Theorem}
\newtheorem{proposition}[theorem]{Proposition}
\newtheorem{lemma}[theorem]{Lemma}
\newtheorem{definition}[theorem]{Definition}
\newtheorem{corollary}[theorem]{Corollary}

\theoremstyle{remark}
\newtheorem*{remark}{Remark}
\newtheorem*{example}{Example}

\author{Pierre-Lo\"ic M\'eliot}
\date{\today}
\title[The cut-off phenomenon for Brownian motions on symmetric spaces of compact type]{The cut-off phenomenon for Brownian\\ motions on symmetric spaces of compact type}

\begin{document}
\setcounter{tocdepth}{2}
\begin{abstract} In this paper, we prove the cut-off phenomenon in total variation distance for the Brownian motions traced on the classical symmetric spaces of compact type, that is to say:\vspace{2mm}
\begin{enumerate}
\item the classical simple compact Lie groups: special orthogonal groups, special unitary groups and compact symplectic groups;\vspace{2mm}
\item the real, complex and quaternionic Grassmannian varieties (including the real spheres, and the complex or quaternionic projective spaces);\vspace{2mm}
\item the spaces of real, complex and quaternionic structures.\vspace{2mm}
\end{enumerate}
Denoting $\mu_{t}$ the law of the Brownian motion at time $t$, we give explicit lower bounds for $\dtv(\mu_{t},\mathrm{Haar})$ if $t < t_{\text{cut-off}}=\alpha \log n$, and explicit upper bounds if $t > t_{\text{cut-off}}$. This provides in particular an answer to some questions raised in recent papers by Chen and Saloff-Coste. Our proofs are inspired by those given by Rosenthal and Porod for products of random rotations in $\SO(n)$, and by Diaconis and Shahshahani for products of random transpositions in $\mathfrak{S}_{n}$.
\end{abstract}

\maketitle

\hrule

\tableofcontents
\vspace{-0.7cm}

\clearpage

\section{Introduction}

\subsection{The cut-off phenomenon for random permutations}
This paper is concerned with the analogue for Brownian motions on compact Lie groups and symmetric spaces of the famous \emph{cut-off phenomenon} observed in random shuffles of cards (\emph{cf.} \cite{AD86,BD92}). Let us recall this result in the case of ``natural'' shuffles of cards, also known as \emph{riffle shuffles}. Consider a deck of $n$ ordered cards $1,2,\ldots,n$, originally in this order. At each time $k \geq 1$, one performs the following procedure:\vspace{2mm}
\begin{enumerate}
\item One cuts the deck in two parts of sizes $m$ and $n-m$, the integer $m$ being chosen randomly according to a binomial law of parameter $\frac{1}{2}$:
$$
\proba[m=M]=\frac{1}{2^{n}}\binom{n}{M}.
$$
So for instance, if $n=10$ and the deck was initially $123456789\mathrm{X}$, then one obtains the two blocks $A=123456$ and $B=789\mathrm{X}$ with probability $\frac{1}{2^{10}}\binom{10}{6}=\frac{105}{512}\simeq 0.21$.\vspace{2mm}
\item The first card of the new deck comes from $A$ with probability $(\card A)/n$, and from $B$ with probability $(\card B)/n$. Then, if $A'$ and $B'$ are the remaining blocks after removal of the first card, the second card of the new deck will come from $A'$ with probability $(\card A')/(n-1)$, and from $B'$ with probability $(\card B')/(n-1)$; and similarly for the other cards. So for instance, by shuffling $A=123456$ and $B=789\mathrm{X}$, one can obtain with probability $1/\binom{10}{6}\simeq 0.0048$ the deck $17283459\mathrm{X}6$. \vspace{2mm}
\end{enumerate}
Denote $\sym_{n}$ the symmetric group of order $n$, and $\sigma^{(k)}$ the random permutation in $\sym_{n}$ obtained after $k$ independent shuffles. One can guess that as $k$ goes to infinity, the law $\proba^{(k)}$ of $\sigma^{(k)}$ converges to the uniform law $\mathbb{U}$ on $\sym_{n}$.\bigskip

There is a natural distance on the set $\mathscr{P}(\sym_{n})$ of probability measures on $\sym_{n}$ that allows to measure this convergence: the so-called \emph{total variation distance} $\dtv$. Consider more generally a measurable space $X$ with $\sigma$-field $\mathcal{B}(X)$. The total variation distance is the metric on the set of probability measures $\mathscr{P}(X)$ defined by
$$
\dtv(\mu,\nu)=\sup\left\{|\mu(A)-\nu(A)|,\,\,\,A \in \mathcal{B}(X)\right\}\,\, \in [0,1].
$$
The convergence in total variation distance is in general a stronger notion than the weak convergence of probability measures. On the other hand, if $\mu$ and $\nu$ are absolutely continuous with respect to a third measure $dx$ on $X$, then their total variation distance can be written as a $\leb^{1}$-norm:
$$
\dtv(\mu,\nu)=\frac{1}{2}\int_{X} \left|\frac{d\mu}{dx}(x)-\frac{d\nu}{dx}(x)\right|\,dx.
$$\vspace{2mm}

It turns out that with respect to total variation distance, the convergence of random shuffles occurs at a specific time $k_{\text{cut-off}}$, that is to say that $\dtv(\proba^{(k)},\mathbb{U})$ stays close to $1$ for $k < k_{\text{cut-off}}$, and that $\dtv(\proba^{(k)},\mathbb{U})$ is then extremely close to $0$ for $k > k_{\text{cut-off}}$. More precisely, in \cite{BD92} (see also \cite[Chapter 10]{CSST}), it is shown that:
\begin{theorem}[Bayer-Diaconis]\label{diaconishuffle}
Assume $k=\frac{3}{2\log 2}\,\log n +\theta$. Then, 
$$
\dtv(\proba^{(k)},\mathbb{U}) = 1-2\,\phi\left(\frac{-2^{-\theta}}{4\sqrt{3}}\right) + O\left(n^{-1/4}\right), \quad\text{with }\phi(x)=\frac{1}{\sqrt{2\pi}}\int_{-\infty}^{x} \E^{-\frac{s^{2}}{2}}\,ds.
$$
So for $\theta$ negative, the total variation distance is extremely close to $1$, whereas it is extremely close to $0$ for $\theta$ positive.
\end{theorem}
\noindent The cut-off phenomenon has been proved for other shuffling algorithms (\emph{e.g.} random transpositions of cards), and more generally for large classes of finite Markov chains, see for instance \cite{DSC96,Dia96}. It has also been investigated by 
Chen and Saloff-Coste for Markov processes on continuous spaces, \emph{e.g.} spheres and Lie groups; see in particular \cite{SC94,SC04,CSC08} and the discussion of \S\ref{sofar}. However, in this case, cut-offs are easier to prove for the $\leb^{p>1}$-norm of $p_{t}(x)-1$, where $p_{t}(x)$ is the density of the process at time $t$ and point $x$ with respect to the equilibrium measure. The case of the $\leb^{1}$-norm, which is (up to a factor $2$) the total variation distance, is somewhat different. In particular, a proof of the cut-off phenomenon for the total variation distance between the Haar measure and the marginal law $\mu_{t}$ of the Brownian motion on a classical compact Lie group was apparently not known --- see the remark just after \cite[Theorem 1.2]{CSC08}. The purpose of this paper is precisely to give a proof of this $\leb^{1}$-cut-off for all classical compact Lie groups, and more generally for all classical symmetric spaces of compact type. In the two next paragraphs, we describe the spaces in which we will be interested (\S\ref{symmetric}), and we precise what is meant by ``Brownian motion'' on a space of this type (\emph{cf.} \S\ref{brown}). This will then enable us to explain the results of Chen and Saloff-Coste in \S\ref{sofar}, and finally to state in \S\ref{statement} which improvements we were able to prove.\bigskip

\subsection{Classical compact Lie groups and symmetric spaces}\label{symmetric} To begin with, let us fix some notations regarding the three classical families of simple compact Lie groups, and their quotients corresponding to irreducible simply connected compact symmetric spaces. We use here most of the conventions of \cite{Hel78,Hel84}. For every $n \geq 1$, we denote $\unit(n)=\unit(n,\C)$ the \emph{unitary group} of order $n$; $\mathrm{O}(n)=\mathrm{O}(n,\R)$ the \emph{orthogonal group} of order $n$; and $\unit\SP(n)=\unit\SP(n,\Hq)$ the  \emph{compact symplectic group} of order $n$. They are defined by the same equations:
$$
UU^{\dagger}=U^{\dagger}U=I_{n} \quad;\quad OO^{t}=O^{t}O=I_{n} \quad;\quad SS^{\star}=S^{\star}S=I_{n}
$$
with complex, real or quaternionic coefficients, the conjugate of a quaternion $w+\I x + \mathrm{j} y + \mathrm{k} z$ being $w-\I x - \mathrm{j} y - \mathrm{k} z$.
The orthogonal groups are not connected, so we shall rather work with the \emph{special orthogonal groups}
$$
\SO(n)=\SO(n,\R)=\left\{O \in \mathrm{O}(n,\R)\,\,|\,\, \det O=1\right\}.
$$
On the other hand, the unitary groups are not simple Lie groups (their center is one-dimensional), so it is convenient to introduce the \emph{special unitary groups} 
$$\SU(n)=\SU(n,\C)=\left\{U \in \unit(n,\C)\,\,|\,\, \det U=1\right\}.$$
Then, for every $n \geq 1$, $\SU(n,\C)$, $\SO(n,\R)$ and $\unit\SP(n,\Hq)$ are connected simple compact real Lie groups, of respective dimensions 
$$
\dim_{\R} \SU(n,\C)=n^{2}-1\quad;\quad\dim_{\R} \SO(n,\R)=\frac{n(n-1)}{2}\quad;\quad \dim_{\R} \unit\SP(n,\Hq)= 2n^{2}+n.
$$
The special unitary groups and compact symplectic groups are simply connected; on the other hand, for $n \geq 3$, the fundamental group of $\SO(n,\R)$ is $\Z/2\Z$, and its universal cover is the \emph{spin group} $\mathrm{Spin}(n)$.
\bigskip

Many computations on these simple compact Lie groups can be performed by using their \emph{representation theory}, which is covered by the highest weight theorem; see \S\ref{weyltheory}. We shall recall all this briefly in Section \ref{fourier}, and give in each case the list of all irreducible representations, and the corresponding dimensions and characters. It is well known that every simply connected compact simple Lie group is: \vspace{2mm}
\begin{itemize}
\item either one group in the infinite families $\SU(n)$, $\mathrm{Spin}(n)$, $\unit\SP(n)$;\vspace{2mm}
\item or, an exceptional simple compact Lie group of type $\mathrm{E}_{6}$, $\mathrm{E}_{7}$, $\mathrm{E}_{8}$, $\mathrm{F}_{4}$ or $\mathrm{G}_{2}$.\vspace{2mm}
\end{itemize}
We shall refer to the first case as the \emph{classical simple compact Lie groups}, and as mentioned before, our goal is to study Brownian motions on these groups.
\bigskip

We shall more generally be interested in compact symmetric spaces; see \emph{e.g} \cite[Chapter 4]{Hel78}. These spaces can be defined by a local condition on geodesics, and by Cartan-Ambrose-Hicks theorem, a symmetric space $X$ is isomorphic as a Riemannian manifold to $G/K$, where $G$ is the connected component of the identity in the isometry group of $X$; $K$ is the stabilizer of a point $x\in X$ and a compact subgroup of $G$; and $(G,K)$ is a symmetric pair, which means that $K$ is included in the group of fixed points $G^{\theta}$ of an involutive automorphism of $G$, and contains the connected component $(G^{\theta})^{0}$ of the identity in this group. Moreover, $X$ is compact if and only if $G$ is compact. This result reduces the classification of symmetric spaces to the classification of real Lie groups and their involutive automorphisms. So, consider an irreducible simply connected symmetric space, of compact type. Two cases arise:\vspace{2mm}
\begin{enumerate}
\item The isometry group $G=K\times K$ is the product of a compact simple Lie group with itself, and $K$ is embedded into $G$ via the diagonal map $k \mapsto (k,k)$. The symmetric space $X$ is then the group $K$ itself, the quotient map from $G$ to $X\simeq K$ being
\begin{align*}
G &\to K\\
g=(k_{1},k_{2}) &\mapsto k_{1}k_{2}^{-1}.
\end{align*}
In particular, the isometries of $K$ are the multiplication on the left and the right by elements of $K\times K$, and this action restricted to $K \subset G$ is the action by conjugacy.\vspace{2mm}
\item The isometry group $G$ is a compact simple Lie group, and $K$ is a closed subgroup of it. In this case, there exists in fact a non-compact simple Lie group $L$ with maximal compact subgroup $K$, such that $G$ is a compact subgroup of the complexified Lie group $L^{\C}$, and maximal among those containing $K$. The involutive automorphism $\theta$ extends to $L^{\C}$, with $K=G^{\theta}=L^{\theta}$ and the two orthogonal symmetric Lie algebras $(\glie,d_{e}\theta)$ and $(\mathfrak{l},d_{e}\theta)$ dual of each other.\vspace{2mm}
\end{enumerate}
The classification of irreducible simply connected compact symmetric spaces is therefore the following: in addition to the compact simple Lie groups themselves, there are the seven infinite families
\begin{align*}
&\Gra(p+q,q,\R)=\SO(p+q)/(\SO(p)\times \SO(q))\,\,\,\text{with }p,q \geq 1\,\,\,\text{(real Grassmannians)};\\
&\Gra(p+q,q,\C)=\SU(p+q)/\mathrm{S}(\unit(p)\times \unit(q))\,\,\,\text{with }p,q \geq 1\,\,\,\text{(complex Grassmannians)};\\
&\Gra(p+q,q,\Hq)=\unit\SP(p+q)/(\unit\SP(p)\times \unit\SP(q))\,\,\,\text{with }p,q \geq 1\,\,\,\text{(quaternionic Grassmannians)};\\
&\SU(n)/\SO(n)\,\,\,\text{with }n\geq 2\,\,\,\text{(real structures on $\C^{n}$)};\\
&\unit\SP(n)/\unit(n)\,\,\,\text{with }n\geq 1\,\,\,\text{(complex structures on $\Hq^{n}$)};\\
&\SO(2n)/\unit(n)\,\,\,\text{with }n\geq 2\,\,\,\text{(complex structures on $\R^{2n}$)};\\
&\SU(2n)/\unit\SP(n)\,\,\,\text{with }n\geq 2\,\,\,\text{(quaternionic structures on $\C^{2n}$)};
\end{align*}
and quotients involving exceptional Lie groups, \emph{e.g.} $\mathbb{P}^{2}(\mathbb{O})=\mathrm{F}_{4}/\mathrm{Spin}(9)$; see \cite[Chapter 10]{Hel78}. For the two last families, one sees $\unit(n)$ as a subgroup of $\SO(2n)$ by replacing each complex number $x+\I y$ by the $2\times2$ real matrix 
\begin{equation} \begin{pmatrix} x & y \\
-y & x \end{pmatrix};\label{doublecomplex}
\end{equation}
and one sees $\unit\SP(n)$ as a subgroup of $\SU(2n)$ by replacing each quaternion number $w+\I x + \mathrm{j} y + \mathrm{k} z$ by the $2 \times 2$ complex matrix
\begin{equation}\begin{pmatrix}
w+\I x & y + \I z\\
-y + \I z & w-\I x
\end{pmatrix}; \label{doublequaternion}\end{equation}
$\unit\SP(n,\Hq)$ is then the intersection of $\SU(2n,\C)$ and of the complex symplectic group $\SP(2n,\C)$. We shall refer to the seven aforementioned families as \emph{classical simple compact symmetric spaces} (of type non-group); again, we aim to study in detail the Brownian motions on these spaces.\bigskip

\subsection{Laplace-Beltrami operators and Brownian motions on symmetric spaces}\label{brown} We denote $d\eta_{K}(k)$ or $dk$ the \emph{Haar measure} of a (simple) compact Lie group $K$, and $d\eta_{X}(x)$ or $dx$ the Haar measure of a compact symmetric space $X=G/K$, which is the image measure of $d\eta_{G}$ by the projection map $\pi : G \to G/K$. We refer to \cite[Chapter 1]{Hel84} for precisions on the integration theory over (compact) Lie groups and their homogeneous spaces. There are several complementary ways to define a Brownian motion on a compact Lie group $K$ or a on compact symmetric space $G/K$, see in particular \cite{Liao04book}. Hence, one can view them: \vspace{2mm}
\begin{enumerate}
\item as Markov processes with infinitesimal generator the Laplace-Beltrami differential operator of the underlying Riemannian manifold; \vspace{2mm}
\item as conjugacy-invariant continuous L\'evy processes on $K$, or as projections of such a process on $G/K$;\vspace{2mm}
\item at least in the group case, as solutions of stochastic differential equations driven by standard (multidimensional) Brownian motions on the Lie algebra.\vspace{2mm}
\end{enumerate}
The first and the third point of view will be specially useful for our computations. For the sake of completeness, let us recall briefly each point of view --- the reader already acquainted with these notions can thus go directly to \S\ref{sofar}.\bigskip

\subsubsection{Choice of normalization and Laplace-Beltrami operators} To begin with, let us precise the Riemannian structures chosen in each case. In the case of a simple compact Lie group $K$, the opposite of the \emph{Killing form} $B(X,Y)=\tr(\mathrm{ad}\,X\,\circ\,\mathrm{ad}\,Y)$ is negative-definite and gives by transport on each tangent space the unique bi-$K$-invariant Riemannian structure on $K$, up to a positive scalar. We choose this normalization constant as follows. When $K=\SU(n)$ or $\SO(n)$ or $\unit\SP(n)$, the Killing form on $\klie$ is a scalar multiple of the bilinear form $X\otimes Y \mapsto \Re(\tr(XY))$ --- the real part is only needed for the quaternionic case. Then, we shall always consider the following invariant scalar products on $\klie$:
\begin{equation}
\scal{X}{Y}=-\frac{\beta n}{2} \,\Re(\tr(XY)),\label{normalization}
\end{equation}
with $\beta=1$ for special orthogonal groups, $\beta=2$ for special unitary groups and unitary groups, and $\beta=4$ for compact symplectic groups (these are the conventions of \emph{e.g.} \cite{Lev11}). Similarly, on a simple compact symmetric space $X=G/K$ of type non-group, we take  the previously chosen $\mathrm{Ad}(G)$-invariant scalar product (the one given by Equation \eqref{normalization}), and we restrict it to the orthogonal complement $\mathfrak{x}$ of $\klie$ in $\glie$. This $\mathfrak{x}$ can be identified with the tangent space of $X=G/K$ at $eK$, and by transport one gets the unique (up to a scalar) $G$-invariant Riemannian structure on $X$, called the Riemannian structure \emph{induced} by the Riemannian structure of $G$. From now on, each classical simple compact symmetric space $X=G/K$ will be endowed with this induced  Riemannian structure. 
\medskip

\begin{remark}
This is not necessarily the ``usual'' normalization for these quotients: in particular, when $G=\SO(n+1)$ and $K=\SO(n)\times\SO(1)=\SO(n)$, the Riemannian structure defined by the previous conventions on the $n$-dimensional sphere $X=\mathbb{S}^{n}(\R)$ differs from the restriction of the standard euclidian metric of $\R^{n+1}$ by a factor $\sqrt{n+1}$. However this normalization does not change the nature of the cut-off phenomenon that we are going to prove.
\end{remark} 

\begin{remark}
The bilinear form in \eqref{normalization} is only proportional to minus the Killing form, and not equal to it; for instance, the Killing form of $\SO(n,\R)$ is $$(n-2)\,\tr(XY)=-\frac{2n-4}{n}\scal{X}{Y},$$ and not $-\scal{X}{Y}$. However, the normalization of Formula \eqref{normalization}  enables one to relate the Brownian motions on the compact Lie groups to the ``standard'' Brownian motions on their Lie algebras, and to the classical ensembles of random matrix theory (see the SDEs at the end of this paragraph).
\end{remark}\bigskip

The \emph{Laplace-Beltrami operator} on a Riemannian manifold $M$ is the differential operator of degree $2$ defined by
$$\Delta f(m)=\sum_{1\leq i,j\leq d} g^{ij}( \nabla_{X_{i}}\!\nabla_{X_{j}}f(m)-\nabla_{\nabla_{X_{i}}X_{j}}f(m)),$$
where $(X_{1},\ldots,X_{d})$ is a basis of $T_{m}M$, $(g^{ij})_{i,j}$ is the inverse of the metric tensor $(g_{ij}=\scal{X_{i}}{X_{j}}_{T_{m}M})_{i,j}$, and $\nabla_{X}Y$ denotes the covariant derivative of a vector $Y$ along a vector $X$ and with respect to the Levi-Civita connection. In the case of a compact Lie group $K$, this expression can be greatly simplified as follows (see for instance \cite[\S2.3]{Liao04book}). Fix once and for all an orthonormal basis $(X_{1},X_{2},\ldots,X_{d})$ of $\klie$. On another tangent space $T_{k}K$, one transports each $X_{i}$ by setting
$$
X_{i}^{l}(k)=\{d_{e}R_{k}\}(X_{i}) \in T_{k}K,
$$
where $R_{k}$ is the multiplication on the right by $k$. One thus obtains a vector field $X_{i}^{l}=\frac{\partial}{\partial x_{i}}$ which is left-invariant by construction and right-invariant because of the $\mathrm{Ad}(K)$-invariance of the scalar product on $\klie$. Then,
\begin{equation}
\Delta=\sum_{i=1}^{d}\frac{\partial^{2}}{\partial x_{i}^{2}}.\label{laplacebeltrami}
\end{equation}
\begin{definition}\label{defbrown}
A (standard) Brownian motion on a compact Riemannian manifold $M$ is a continuous Feller process $(m_{t})_{t \in \R_{+}}$ whose infinitesimal generator restricted to $\mathscr{C}^{2}(M)$ is $\frac{1}{2}\,\Delta$.
\end{definition}

\noindent In the following, on a compact Lie group $K$ or a compact symmetric space $G/K$, we shall also assume that $m_{0}=e$ or $m_{0}=eK$ almost surely. We shall then denote $\mu_{t}$ the marginal law of the process at time $t$, and $p_{t}^{K}(k)=\frac{d\mu_{t}}{d\eta_{K}}(k)$ or $p_{t}^{X}(x)=\frac{d\mu_{t}}{d\eta_{X}}(x)$ the density of $\mu_{t}$ with respect to the Haar measure. General results about hypoelliptic diffusions on manifolds ensure that these densities exist for $t >0$ and are continuous in time and space; we shall later give explicit formulas for them (\emph{cf.} Section \ref{fourier}). 
\bigskip

\subsubsection{Brownian motions as continuous L\'evy processes} By using the geometry of the spaces considered and the language of L\'evy processes, one can give another equivalent definition of Brownian motions. The \emph{right increments} of a random process $(g_{t})_{t \in \R_{+}}$ with values in a (compact) Lie group $G$ are the random variables $r_{t}^{s}=g_{s}^{-1}g_{t}$, so $g_{t}=g_{s}\,r_{t}^{s}$ for any times $s \leq t$. Then, a \emph{left L\'evy process} on $G$ is a c\`adl\`ag random process such that:\vspace{2mm}
\begin{enumerate}
\item For any times $0=t_{0}\leq t_{1}\leq \cdots \leq t_{n}$, the right increments $r_{t_{1}}^{t_{0}}, r_{t_{2}}^{t_{1}},\ldots, r_{t_{n}}^{t_{n-1}}$ are independent.\vspace{2mm}
\item For any times $s \leq t$, the law of $r_{t}^{s}$ only depends on the difference $t-s$: $r_{t-s}^{0}\stackrel{\text{\tiny law}}{=} r_{t}^{s}.$\vspace{2mm}
\end{enumerate}
Denote $P_{t}$ the operator on the space $\mathscr{C}(G)$ of continuous functions on $G$ defined by $(P_{t}f)(g)=\esper[f(gg_{t})];$
and $\mu_{t}$ the law of $g_{t}$ assuming that $g_{0}=e_{G}$ almost surely. For $h \in G$, we also denote by $L_{h}$ the operator on $\mathscr{C}(G)$ defined by $L_{h}f(g)=f(hg)$. If $(g_{t})_{t \in \R_{+}}$ is a left L\'evy process on $G$ starting at $g_{0}=e_{G}$, then:\vspace{2mm}
\begin{enumerate}
\item The family of operators $(P_{t})_{t \in \R_{+}}$ is a Feller semigroup that is left $G$-invariant, meaning that $P_{t}\circ L_{h}=L_{h}\circ P_{t}$ for all $h \in G$ and for all time $t$. Conversely, any such Feller semigroup is the group of transitions of a left L\'evy process which is unique in law.
\vspace{2mm}
\item The family of laws $(\mu_{t})_{t \in \R_{+}}$ is a semigroup of probability measures for the convolution product of measures
$$
(\mu*\nu)(f)=\int_{G^{2}}f(gh)\,d\mu(g)\,d\nu(h).
$$
Hence, $\mu_{s}*\mu_{t}=\mu_{s+t}$ for any $s$ and $t$. Moreover, this semigroup is continuous, \emph{i.e.}, the limit in law $\lim_{t \to 0}\mu_{t}$ exists and is the Dirac measure $\delta_{e}$. Conversely, given such a semigroup of measures, there is always a corresponding left L\'evy process, and it is unique in law.
\end{enumerate}\bigskip

Thus, left L\'evy processes are the same as left $G$-invariant Feller semigroups of operators, and they are also the same as continuous semigroups of probability measures on $G$. In particular, on a compact Lie group, they are characterized by their infinitesimal generator
$$Lf(g)=\lim_{t \to \infty}\frac{P_{t}f(g)-f(g)}{t}$$
defined on a suitable subspace of $\mathscr{C}(G)$. Hunt's theorem (\emph{cf.} \cite{Hun56}) then characterizes the possible infinitesimal generators of (left) L\'evy processes on a Lie group; in particular, continuous left-L\'evy processes correspond to left-invariant differential operator of degree $2$. \bigskip

Assume then that $(g_{t})_{t\in \R_{+}}$ is a continuous L\'evy process on a simple compact Lie group $G$, starting from $e$ and with the additional property that $(hg_{t}h^{-1})_{t \in \R_{+}}$ and $(g_{t})_{t \in \R_{+}}$ have the same law in $\mathscr{C}(\R_{+},G)$ for every $h$.
These hypotheses imply that the infinitesimal generator $L$, which is a differential operator of degree $2$, is a scalar multiple of the Laplace-Beltrami operator $\Delta$. Thus, on a simple compact Lie group $K$, up to a linear change of time $ t\mapsto at$, a \emph{conjugacy-invariant continuous left L\'evy process} is a Brownian motion in the sense of Definition \ref{defbrown}. Similarly, on a simple compact symmetric space $G/K$, up to a linear change of time, the image $(g_{t}K)_{t \in \R_{+}}$ of a conjugacy-invariant continuous left L\'evy process on $G$ is a Brownian motion in the sense of Definition \ref{defbrown}. This second definition of Brownian motions on compact symmetric spaces has the following important consequence:
\begin{lemma}\label{nonincreasingdistance}
Let $\mu_{t}$ be the law of a Brownian motion on a compact Lie group $K$ or on a compact symmetric space $G/K$. The total variation distance $\dtv(\mu_{t},\mathrm{Haar})$ is a non-increasing function of $t$.
\end{lemma}
\begin{proof}
First, let us treat the case of compact Lie groups. If $f_{1},f_{2}$ are in $\leb^{1}(K,d\eta_{K})$, then their convolution product $f_{1}*f_{2}$ is again in $\leb^{1}(K)$, with $$\|f_{1}*f_{2}\|_{\leb^{1}(K)} \leq \|f_{1}\|_{\leb^{1}(K)}\,\|f_{2}\|_{\leb^{1}(K)}.$$ Now, since $\mu_{s+t}=\mu_{s}*\mu_{t}$, the densities of the Brownian motion also satisfy $p_{s+t}^{K}=p_{s}^{K}*p_{t}^{K}$.
Consequently,
\begin{align*}
2\,\dtv(\mu_{s+t},\eta_{K})&=\|p_{s+t}^{K}-1\|_{\leb^{1}(K)}=\|(p_{s}^{K}-1)*p_{t}^{K}\|_{\leb^{1}(K)}\\
& \leq \|p_{s}^{K}-1\|_{\leb^{1}(K)} \,\|p_{t}^{K}\|_{\leb^{1}(K)}=\|p_{s}^{K}-1\|_{\leb^{1}(G)}=2\,\dtv(\mu_{s},\eta_{K}).
\end{align*}
The proof is thus done in the group case. For a compact symmetric space $X=G/K$, denote $p_{t}^{G}$ the density of the Brownian motion on $G$, and $p_{t}^{X}$ the density of the Brownian motion on $X$. Since the Brownian motion on $X$ is the image of the Brownian motion on $G$ by $\pi :G \to G/K$, one has:
$$\forall x=gK,\,\,\,p_{t}^{X}(x)=\int_{K}p_{t}^{G}(gk)\,dk.$$
As a consequence, 
\begin{align*}
\|p_{s+t}^{X}-1\|&_{\leb^{1}(X)}=\int_{G} \left|p_{s+t}^{X}(gK)-1\right|dg =\int_{G}\left|\int_{K}(p_{s+t}^{G}(gk)-1)\,dk\right|dg\\
&=\int_{G}\left|\int_{K\times G}(p_{s}^{G}(h^{-1}gk)-1)\,p_{t}^{G}(h)\,dk\,dh\right|dg = \int_{G} \left|\int_{G} (p_{s}^{X}(h^{-1}gK)-1)\,p_{t}^{G}(h)\,dh\right|dg\\
&\leq \int_{G\times G} \left|p_{s}^{X}(h^{-1}gK)-1)\right|\,\left|p_{t}^{G}(h)\right|\,dh\,dg=\|p_{s}^{X}-1\|_{\leb^{1}(X)}\,\|p_{t}^{G}\|_{\leb^{1}(G)}=\|p_{s}^{X}-1\|_{\leb^{1}(X)},
\end{align*}
so $\dtv(\mu_{s+t},\eta_{X}) \leq \dtv(\mu_{s},\eta_{X})$ also in the case of symmetric spaces.
\end{proof}
\begin{remark}
 Later, this property will allow us to compute estimates of $\dtv(\mu_{t},\eta_{X})$ only for $t$ around the cut-off time. Indeed, if one has for instance an (exponentially small) estimate of $1-\dtv(\mu_{t_{0}},\eta_{X})$ at time $t_{0}=(1-\eps)\,t_{\text{cut-off}}$, then the same estimate will also hold for $1-\dtv(\mu_{t},\eta_{X})$ with $t < t_{0}$.
 \end{remark}
\begin{remark}
Actually, the same result holds for the $\leb^{p}$-norm of $p_{t}(x)-1$, and in the broader setting of Markov processes with a stationary measure; see \emph{e.g.} \cite[Proposition 3.1]{CSC08}. Our proof is a little more elementary.
\end{remark}
\bigskip

\subsubsection{Brownian motions as solutions of SDE} A third equivalent definition of Brownian motions on compact Lie groups is by mean of stochastic differential equations. More precisely, given a Brownian motion $(k_{t})_{t \in \R_{+}}$ traced on a compact Lie group $K$, there exists a (trajectorially unique) standard $d$-dimensional Brownian motion $(W_{t})_{t \in \R_{+}}$ on the Lie algebra $\klie$ that drives stochastic differential equations for every test function $f \in \mathscr{C}^{2}(K)$ of $(k_{t})_{t \in \R_{+}}$ (\emph{cf.} \cite{Liao04book}). So for instance, on a unitary group $\unit(n,\C)$, the Brownian motion is the solution of the SDE
$$
U_{0}=I_{n}\qquad;\qquad dU_{t}=\I\,U_{t}\cdot dH_{t}-\frac{1}{2}\,U_{t}\,dt,
$$
where $(H_{t})_{t \in \R_{+}}$ is a Brownian hermitian matrix normalized so that at time $t=1$ the diagonal coefficients are independent real gaussian variables of variance $1/n$, and the upper-diagonal coefficients are independent complex gaussian variables with real and imaginary parts independent and of variance $1/2n$. In the general case, let us introduce the \emph{Casimir operator}
\begin{equation}
C=\sum_{i=1}^{d} X_{i} \otimes X_{i}.\label{casimir}
\end{equation}
This tensor should be considered as an element of the universal enveloping algebra $U(\klie)$. Then, for every representation $\pi : K \to \GL(V)$, the image of $C$ by the infinitesimal representation $d\pi : U(\klie) \to \mathrm{End}(V)$ commutes with $d\pi(\glie)$.  In particular, for an irreducible representation $V$, $d\pi(C)$ is a scalar multiple $\kappa_{V}\mathrm{id}_{V}$ of $\mathrm{id}_{V}$. Assume that $K$ is a classical simple Lie group. Then its ``geometric'' representation is irreducible, so $\sum_{i=1}^{d} (X_{i})^{2}=\alpha_{\glie}\,I_{n}$ if one sees the $X_{i}$'s as matrices in $\matr(n,\R)$ or $\matr(n,\C)$ or $\matr(n,\Hq)$. The stochastic differential equation satisfied by a Brownian motion on $K$ is then
$$
k_{0}=e_{K}\qquad;\qquad dk_{t}=k_{t}\cdot dB_{t} + \frac{\alpha_{\klie}}{2} k_{t}\,dt,
$$
where $B_{t}=\sum_{i=1}^{d}W_{t}^{i}\,X_{i}$ is a standard Brownian motion on the Lie algebra $\klie$. The constant $\alpha_{\klie}$ is given in the classical cases by\label{casimircoefficient}
$$
\alpha_{\sulie(n)}=-\frac{n^{2}-1}{n^{2}}\quad;\quad\alpha_{\solie(n)}=-\frac{n-1}{n}\quad;\quad \alpha_{\splie(n)}=-\frac{2n+1}{2n}
$$
see \cite[Lemma 1.2]{Lev11}. These Casimir operators will play a prominent role in the computation of the densities of these Brownian motions (\emph{cf.} \S\ref{weyltheory}), and also at the end of this paper (\S\ref{zonal}), see Lemma \ref{expectationcoefficients}.\bigskip

\subsection{Chen-Saloff-Coste results on $\mathscr{L}^{p}$-cut-offs of Markov processes}\label{sofar}
Fix $p \in [1,\infty)$, and consider a Markov process $\mathfrak{X}=(x_{t})_{t \in \R_{+}}$ with values in a measurable space $(X,\mathcal{B}(X))$, and admitting an invariant probability $\eta$. One denotes $\mu_{t,x}$ the marginal law of $x_{t}$ assuming $x_{0}=x$ almost surely, and
$$d_{t}^{p}(\mathfrak{X})=\max_{x\in X} \left(\int_{X} \left|\frac{d\mu_{t,x}}{d\eta}(y)-1\right|^{p}\eta(dy)\right)^{\frac{1}{p}},$$
with by convention $$d_{t}^{p}(\mathfrak{X})=\begin{cases} 2&\text{if }p=1,\\+\infty &\text{if }p>1,\end{cases}$$ when $\mu_{t,x}$ is not absolutely continuous with respect to $\eta$. This is obviously a generalization of the total variation distance to the stationary measure. In virtue of the remark stated just after Lemma \ref{nonincreasingdistance}, $t \mapsto d_{t}^{p}(\mathfrak{X})$ is always non-increasing. A sequence of Markov processes $(\mathfrak{X}^{(n)})_{n \in \N}$ with values in measurable spaces $(X^{(n)},\mathcal{B}(X^{(n)}))_{n \in \N}$ is said to have a \emph{max-$\mathscr{L}^{p}$-cut-off} with cut-off times $(t^{(n)})_{n \in \N}$ if
$$\lim_{n \to \infty} \left(\sup_{t>(1+\eps)t^{(n)}}d_{t}^{p}(\mathfrak{X}^{(n)})\right)=0\qquad;\qquad \lim_{n \to \infty} \left(\inf_{t<(1-\eps)t^{(n)}}d_{t}^{p}(\mathfrak{X}^{(n)})\right)=\limsup_{n\to \infty}\,d_{0}^{p}(\mathfrak{X}^{(n)})=M>0$$
for every $\eps \in (0,1)$ --- usually $M$ will be equal to $2$ or $+\infty$. A generalization of Theorem \ref{diaconishuffle} ensures that these $\leb^{p}$-cut-offs occur for instance in the case of riffle shuffles of cards, with $t^{(n)}=\frac{3\log n}{2\log 2}$ for every $p \in [1,+\infty)$.\bigskip

In \cite{CSC08}, Chen and Saloff-Coste shown that a general criterion due to Peres ensures a $\leb^{p>1}$-cut-off for a sequence of Markov processes; but then one does not know necessarily the value of the cut-off time $t^{(n)}$. Call \emph{spectral gap} $\lambda(\mathfrak{X})$ of a Markov process $\mathfrak{X}$ the largest $c \geq 0$ such that for all $f \in \leb^{2}(X,\eta)$ and all time $t$,
$\|(P_{t}-\eta)f\|_{\leb^{2}(X)}\leq \E^{-tc}\,\|f\|_{\leb^{2}(X)},$ where $(P_{t})_{t \in \R_{+}}$ stands for the semigroup associated to the Markov process. \begin{theorem}[Chen-Saloff-Coste]\label{csc}
Fix $p \in (1,\infty)$. One considers a family of Markov processes $(\mathfrak{X}^{(n)})_{n \in \N}$ with normal operators $P_{t}$ and spectral gaps $\lambda^{(n)}$, and one assumes that $\lim_{t \to \infty}d_{t}^{p}(\mathfrak{X}^{(n)})=0$ for every $n$. For $\eps_{0} >0$ fixed, set $$t^{(n)}=\inf\{t : d_{t}^{p}(\mathfrak{X}^{(n)})\leq \eps_{0}\}.$$
The family of Markov processes has a max-$\leb^{p}$-cut-off if and only if Peres' criterion is satisfied:
$$\lim_{n \to \infty} \lambda^{(n)}\,t^{(n)}=+\infty.$$
\end{theorem}\medskip

In this case, the sequence $(t^{(n)})_{n \in \N}$ gives the values of the cut-off times. A lower bound on $t^{(n)}$ also ensures the cut-off phenomenon; but then, the cut-off time remains \emph{unknown}. Nevertheless, an important application of this general criterion is (see \cite[Theorem 1.2]{CSC08}, and also \cite[Theorem 1.1 and 1.2]{SC04}):
\begin{corollary}[Saloff-Coste]\label{window}
Consider the Brownian motions traced on the special orthogonal groups $\SO(n,\R)$, with the normalization of the metric detailed in the previous paragraph. They exhibit for every $p \in (1,\infty)$ a cut-off with $t^{(n)}$ asymptotically between $2\log n$ and $4\log n$ --- notice that $t^{(n)}$ depends on $p$.
\end{corollary}
\noindent Indeed, the spectral gap stays bounded and has a non-negative limit (which we shall compute later), whereas $t^{(n)}$ was shown by Saloff-Coste to be a $O(\log n)$. Similar results are presented in \cite{SC04} in the broader setting of simple compact Lie groups or compact symmetric spaces, but without a proof of the cut-off phenomenon (Saloff-Coste gave a window for $t^{(n)}$ for every $p \in [1,+\infty]$). The main result of our paper is that a cut-off indeed occurs for every $p \in [1,+\infty($, for every classical simple compact Lie group or classical simple compact symmetric space, and with a cut-off time equal to  $\log n$ or $2\log n$ depending on the type of the space considered. In particular, the main improvements in comparison to the aforementioned theorems are:\vspace{2mm}
\begin{enumerate}
\item the case $p=1$ is now included; \vspace{2mm}
\item one knows the precise value of the cut-off time.\vspace{2mm}
\end{enumerate}
\bigskip

\subsection{Statement of the main results and discriminating events}\label{statement}
\begin{theorem}\label{main}
Let $\mu_{t}$ be the marginal law of the Brownian motion traced on a classical simple compact Lie group, or on classical simple compact symmetric space. There exists positive constants $\alpha$, $\gamma_{b}$, $\gamma_{a}$, $c$, $C$ and an integer $n_{0}$ such that in each family, for all $n \geq n_{0}$,
\begin{align}
&\forall \eps \in (0,1/4),\,\,\, \dtv(\mu_{t},\mathrm{Haar}) \geq 1-\frac{c}{n^{\gamma_{b}\eps}}\,\,\,\text{ if }t=\alpha\,(1-\eps)\,\log n;\label{mainlower}\\
&\forall \eps \in (0,\infty),\,\,\, \dtv(\mu_{t},\mathrm{Haar}) \leq \frac{C}{n^{\gamma_{a}\eps/4}}\,\,\,\,\,\,\,\,\,\,\,\,\text{ if }t=\alpha\,(1+\eps)\,\log n.\label{mainupper}
\end{align}
The constants $\alpha$, $\gamma_{b}$ and $\gamma_{a}$ are determined by the type of the space considered, and then one can make the following choices for $n_{0}$, $c$ and $C$:\vspace{2mm}
$$
\begin{tabular}{|c|c|c|c|c|c|c|c|}
\hline $K$ or $G/K$ & $\beta$ & $\alpha$ & $\gamma_{b}$ & $\gamma_{a}$ & $n_{0}$& $c$&$C$\\
\hline \hline $\SO(n,\R)$ & $1$& $2$& $2$& $2$& $10$& $36$ & $6$\\
\hline $\SU(n,\C)$ & $2$ & $2$ &$2$&$4$& $2$ & $8$ & $10$\\
\hline $\unit\SP(n,\Hq)$& $4$ & $2$ &$2$&$2$& $3$ & $5$& $3$\\
\hline \hline $\Gra(n,q,\R)$& $1$ & $1$ &$1$&$1$& $10$ & $32$& $2$\\
\hline $\Gra(n,q,\C)$& $2$ & $1$ &$1$&$2$& $2$ & $32$& $2$\\
\hline $\Gra(n,q,\Hq)$& $4$ & $1$ &$1$&$1$& $3$ & $16$& $2$\\
\hline \hline $\SO(2n,\R)/\unit(n,\C)$& $1$ & $1$ &$2$&$1$& $10$ & $8$& $2$\\
\hline $\SU(n,\C)/\SO(n,\R)$& $2$ & $1$ &$2$&$2$& $2$ & $24$& $8$\\
\hline $\SU(2n,\C)/\unit\SP(n,\Hq)$& $2$ & $1$ &$2$&$2$& $2$ & $22$& $8$\\
\hline $\unit\SP(n,\Hq)/\unit(n,\C)$& $4$ & $1$ &$2$&$1$& $3$ & $17$& $2$\\
\hline
\end{tabular}
$$
\end{theorem}
\figcap{
\includegraphics{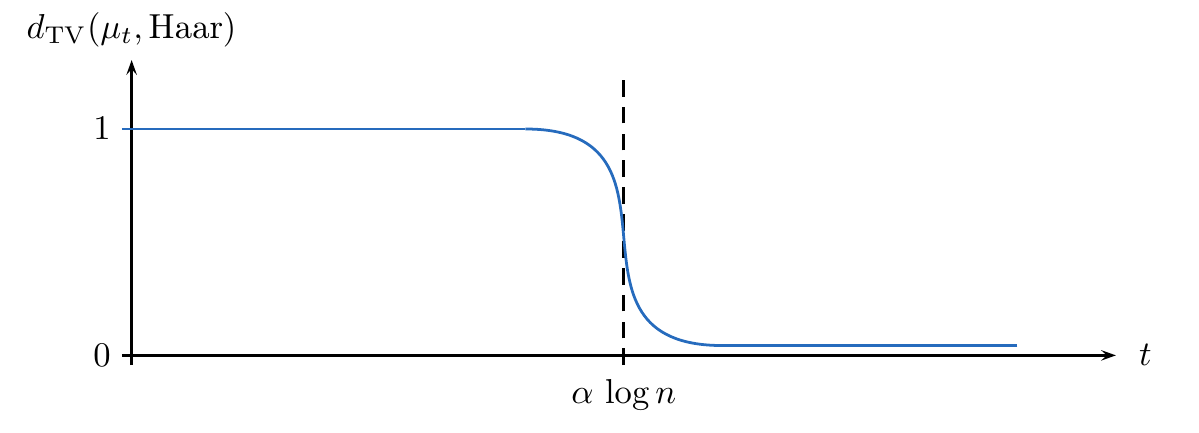}\vspace{-2mm}
}{Aspect of the function $t \mapsto \dtv(\mu_{t},\mathrm{Haar})$ for the Brownian motion on a classical simple compact Lie group or on a classical simple compact symmetric space.\label{mainfig}}
As the function $t \mapsto \dtv(\mu_{t},\mathrm{Haar})$ is non-increasing in $t$, the aspect of this function in the scale $t \propto \log n$ is then always as on Figure \ref{mainfig}. The constants $c$ and $C$ in Theorem \ref{main} can be slightly improved by raising the integer $n_{0}$; the restriction $n\geq n_{0}$ will only be used to ease certain computations and to get reasonable constants $c$ and $C$. A result similar to Theorem \ref{main} has been proved by Rosenthal and Porod in \cite{Ros94,Por96a,Por96b} for random products of (real, or complex, or quaternionic) reflections. Our proofs are really inspired by their proofs, though quite different in the details of the computations. \bigskip

For the upper bound \eqref{mainupper}, it has long been known that if $\lambda(X_{n})$ denotes the spectral gap of the heat semigroup associated to the infinitesimal generator $L=\frac{1}{2}\Delta$, then for $n$ fixed, the total variation distance $\dtv(\mu_{t},\eta_{X_{n}})$ decreases exponentially fast (see \emph{e.g.} \cite{Liao04paper}):
$$
\dtv(\mu_{t},\eta_{X_{n}}) \leq C(X_{n})\,\E^{-\lambda(X_{n})\,t}.
$$
Consider now the family of spaces $(X_{n})_{ n \in \N}$, and assume that $C(X_{n})=C\,n^{\delta}$, and that $\lambda(X_{n})$ stays almost constant to $\lambda$ --- this last condition is ensured by the normalization \eqref{normalization}. Then, one obtains for $t=(1+\eps)\,\frac{\delta}{\lambda}\,\log n$ the bound
$$
\dtv(\mu_{t},\eta_{X_{n}}) \leq \frac{C}{n^{\delta \eps}}.
$$
Thus in theory, the upper bound \eqref{mainupper} follows from the calculations of the constants $C(X_{n})$ and $\lambda(X_{n})$ in each classical family. It is very hard to find directly a constant $C(X_{n})$ that works for every time $t$. But on the other side, by using the representation theory of the classical simple compact Lie groups (\emph{cf.} Section \ref{fourier}), one can determine series of negative exponentials that dominates the total variation distance; see Proposition \ref{scaryseries}. In these series, the ``least negative'' exponentials give the correct order of decay $\lambda(X_{n})$. It remains then to prove that the other terms can be uniformly bounded. This is tedious, but doable, and these precise estimates are shown in Section \ref{upper}: we shall adapt and improve the arguments of \cite{Ros94,Por96a,Por96b,CSST}.\bigskip

As for the lower bound \eqref{mainlower}, it is obtained by looking at \emph{discriminating events}, that have a probability close to $1$ with respect to a marginal law $\mu_{t}$ with $t <t_{\text{cut-off}}$, and close to $0$ with respect to the Haar measure. For instance, in the case of riffle shuffles, the sizes of the \emph{rising sequences} of a permutation enable one to discriminate a random shuffle of order $k < k_{\text{cut-off}}$ from a uniform permutation; see \cite[\S2]{BD92}. In the case of a Brownian motion on a classical compact Lie group, this is the \emph{trace} of the matrices that allows to discriminate Haar distributed elements and random Brownian elements before cut-off time. 
\figcap{\vspace{-50mm}\includegraphics[scale=0.5]{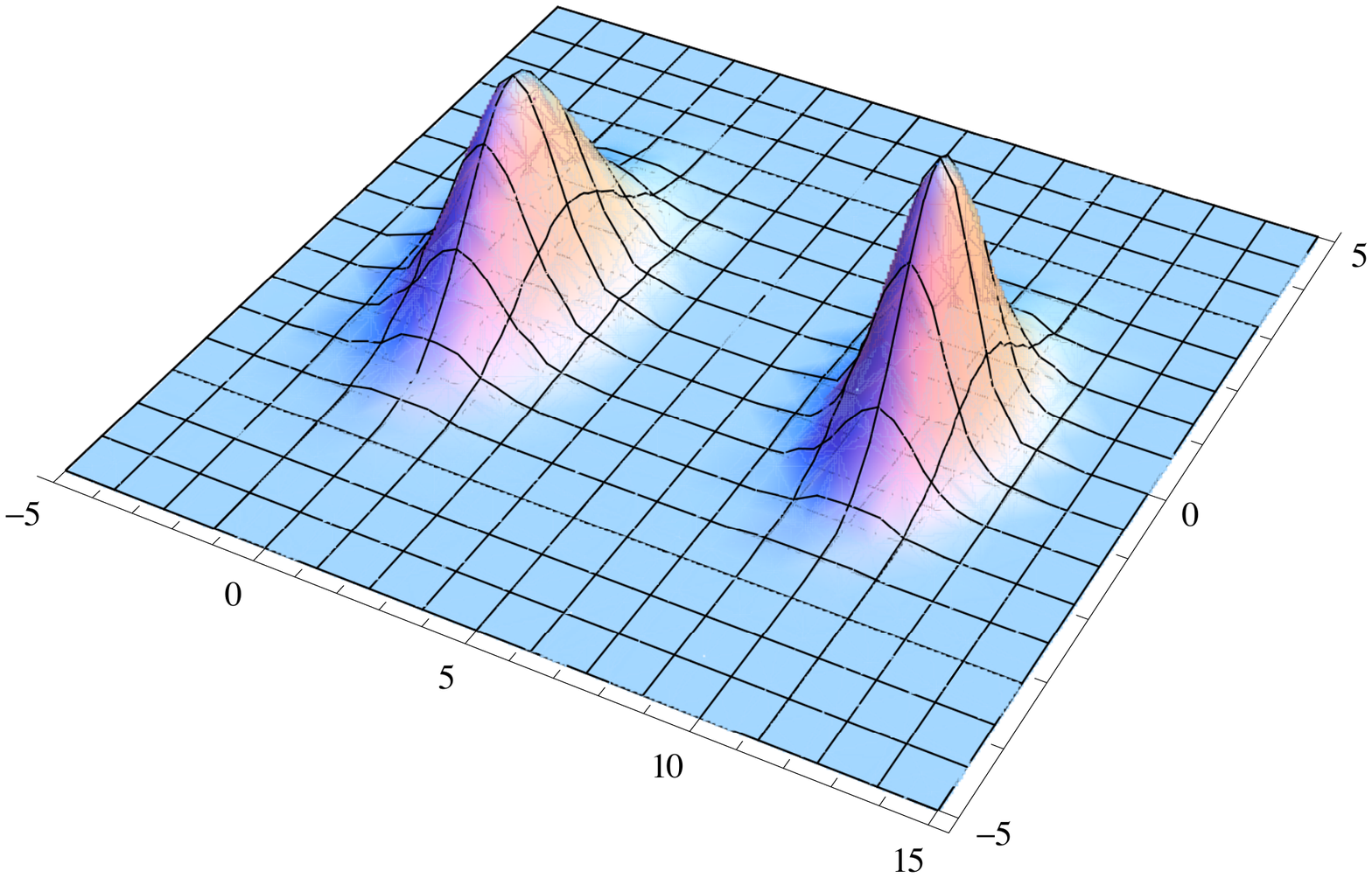}\vspace{-30mm}}{Aspect of the density of the trace $\tr\,U_{n}$ of a random unitary matrix, with $U_{n} \sim \mathrm{Haar}$ for the left peak, and $U_{n} \sim \mu_{t<t_{\text{cut-off}}}$ for the right peak (using \texttt{Mathematica}).}
Indeed, consider for instance a random unitary matrix $U_{n}$ of size $n$, taken under the Haar measure or under the marginal law $\mu_{t}$ of the Brownian motion at a given time $t$. Then, $\tr\, U_{n}$ is a complex valued random variable, and we shall see that
$$
\esper\left[|\tr\, U_{n}-m|^{2}\right]\leq 1,
$$
where $m$ is the mean of $\tr\, U_{n}$; and this, for any $n\geq 1$ and any time $t \geq 0$ if $U_{n} \sim \mu_{t}$. However, $m=0$ under the Haar measure, whereas $|m|\gg 1$ for $t < t_{\text{cut-off}}$. So, the trace of a Brownian unitary matrix before cut-off time will never ``look the same'' as the trace of an Haar distributed unitary matrix.\bigskip

Up to a minor modification, the same argument will work for special orthogonal groups and compact special orthogonal groups --- in this later case, the trace of a quaternionic matrix of size $n$ is defined as the trace of the corresponding complex matrix of size $2n$, \emph{cf.} the remark at the end of \S\ref{symmetric}.  Over the classical simple compact symmetric spaces, the trace of matrices will be replaced by a zonal spherical function ``of minimal non-zero weight''; these minimal zonal spherical functions are also those that give the order of decay of the series of negative exponentials that dominate $\dtv(\mu_{t},\mathrm{Haar})$ after the cut-off time. This argument for the lower bound was already known, since it has been used successfully in \cite{SC94} to prove the cut-off phenomenon over spheres: we have simply extended it to the case of general simple compact symmetric spaces (\emph{cf.} Section \ref{lower}).
\bigskip

An important consequence of Theorem \ref{main} and its proof is that one also has a max-$\leb^{p}$-cut-off for every $p \in [1,\infty]$. Moreover, the value of the cut-off time is known when $p \in [1,2]$. 
\begin{corollary}\label{lpcutoff}
For every $p \in [1,+\infty]$, the family of Brownian motions $(\mathfrak{X}_{n})_{n \in \N}$ traced on simple compact Lie groups $(K_{n})_{n \in \N}$ in one of the three classical families (respectively, on simple compact symmetric spaces of type non-group $(X_{n})_{n \in \N}$ in one of the seven classical families) has a max-$\leb^{p}$-cut-off. If $p \in [1,2]$, it is with respect to the sequence $t^{(n)}=2\log n$ (respectively, $t^{(n)}=\log n$).
\end{corollary}
\begin{proof}
The upper bound in Theorem \ref{main} will be shown by using Cauchy-Schwarz inequality and estimating the $\leb^{2}$-norm of $|p_{t}-1|$, which can be written as a series $S_{n}(t)$ of negative exponentials. Section \ref{upper} is devoted to the proof of the fact that $S_{n}(t)$ is small after cut-off time, and on the other hand, the same series trivially goes to infinity before cut-off time, because some of its terms go to infinity (consider for instance the term indexed by the ``minimal'' label identified in Lemma \ref{decay}). Thus, our proof of Theorem \ref{main} implies readily a $\leb^{2}$-cut-off; and since the Brownian motion is invariant by action of the isometry group, it is even a max-$\leb^{2}$-cut-off. We can then use \cite[Theorem 5.3]{CSC08} to obtain the existence of a max-$\leb^{p}$-cut-off for every $p \in (1,+\infty]$, and the comparison theorem of mixing times \cite[Proposition 5.1]{CSC08} to get the value of the cut-off time when $p$ is between $1$ and $2$.
\end{proof}
\begin{remark}
 When $p=+\infty$, \cite[Theorem 5.3]{CSC08} also gives the value of the cut-off time: it is $4 \log n$ in the group case, and $2 \log n$ in the non-group case. However, when $p \in (2,+\infty)$, one still does not know the value of the mixing time: one has only the window $\alpha\log n \leq t^{(n)} \leq 2\alpha \log n$.
\end{remark}\bigskip

\subsection{Organization of the paper} 
In Section \ref{fourier}, we recall the basics of representation theory and harmonic analysis on compact symmetric spaces, with a particular emphasis on explicit formulas since we will need them in each case. All of it is really classical and of course well-known by the experts, but it is necessary in order to fix the notations related to the harmonic analysis of the classical compact Lie groups and compact symmetric spaces. In Section \ref{upper}, we use the explicit expansion of the densities to establish precise upper bounds on $\|p_{t}-1\|_{\leb^{2}(X,\eta)}$; by Cauchy-Schwarz we obtain similar upper bounds on $\dtv(\mu_{t},\eta)$. The main idea is to control the growth of the dimension of an irreducible spherical representation involved in the expansion of $p_{t}$ when the corresponding highest weight grows in the lattice of weights (\S\ref{versus}). The crucial fact, which was apparently unknown, is that precisely at cut-off time, the quantity
$$\begin{cases}  (D^{\lambda})^{2}\,\E^{-t_{\text{cut-off}} B_{n}(\lambda)} & \text{in the group case},\\
D^{\lambda}\,\E^{-t_{\text{cut-off}} B_{n}(\lambda)} & \text{in the non-group case},
\end{cases}
$$
stays \emph{bounded} for every $n$ and every $\lambda$; $D^{\lambda}$ being the dimension of the irreducible or spherical irreducible representation of label $\lambda$, and $-B_{n}(\lambda)$ the associated eigenvalue of the Laplace-Beltrami operator. Combining this argument with a simple analysis of the generating series 
$$\sum_{\lambda \text{ partition}} x^{|\lambda|}=\prod_{i \geq 1}\frac{1}{1-x^{i}},$$
this is sufficient to get a correct upper bound after cut-off time. \bigskip

Section \ref{lower} is then devoted to the proof of the lower bounds. We use in each case a ``minimal'' zonal spherical function (the trace of matrices in the case of groups; see \S\ref{zonal}), and we compute its expectation and variance under Haar measure and Brownian measures (\S\ref{bienayme}). A simple application of Bienaym\'e-Chebyshev's inequality will then show that the chosen zonal spherical function is indeed discriminating. An algebraic difficulty occurs in the case of symmetric spaces $G/K$ of type non-group, as one has to compute the expansion in zonal functions of the square of the discriminating zonal function, and this is far less obvious than in the case of irreducible characters. The problem is solved by writing the discriminating zonal function in terms of the coefficients of the matrices in the isometry group $G$, and by computing the joint moments of these coefficients under a Brownian measure. The combinations of negative exponentials appearing in these formulas are then in correspondence with the expansions of the squares of the discriminating zonal spherical functions. 
\medskip

\subsection*{Acknowledgements}
Many thanks are due to Yacine Barhoumi, Philippe Biane, Florent Benaych-Georges, Paul Bourgade, Reda Chhaibi, Djalil Chafa\"i, Kenneth Maples, Ashkan Nikeghbali and Simon P\'epin-Lehalleur for discussions around the cut-off phenomenon and the theory of Lie groups.
\bigskip
\bigskip

\section{Fourier expansion of the densities}\label{fourier}

In this section, we explain how to compute the density $p_{t}^{K}(k)$ or $p_{t}^{X}(x)$ of the marginal law $\mu_{t}$ of the Brownian motion traced on a classical compact symmetric space. This computation is done in an abstract setting for instance in \cite{Liao04paper} or \cite{Apple11}, and we shall give at the end of this section its concrete counterpart in each classical case, see Theorem \ref{explicitdensity}. The main ingredients of the computation are:\vspace{2mm}
\begin{enumerate}
\item Peter-Weyl's theorem and its refinement due to Cartan, that ensures that the matrix coefficients of the irreducible representations of $K$ (respectively, of the irreducible spherical representations of $G$) form an orthogonal basis of $\leb^{2}(K,\eta)$ (respectively, of $\leb^{2}(G/K,\eta)$); see \S\ref{peterweyl}.\vspace{2mm}
\item the classical highest weight theory, that describes the irreducible representations of a compact simple Lie group and give formulas for their dimensions and characters; see \S\ref{weyltheory}.\vspace{2mm}
\end{enumerate}
On these subjects, we refer to the two books by Helgason \cite{Hel78,Hel84}, and also to \cite{BD85,Var89,FH91,Far08,GW09} for the representation theory of compact Lie groups. We shall only recall what is needed in order to have a good understanding of the formulas of Theorem \ref{explicitdensity}. We shall also fix all the notations related to the harmonic analysis on (classical) compact symmetric spaces.\bigskip

\subsection{Peter-Weyl's theorem and Cartan's refinement}\label{peterweyl}
Let $K$ be a compact (Lie) group, and $\widehat{K}$ be the set of isomorphism classes of irreducible complex linear representations of $K$. Each class $\lambda \in \widehat{K}$ is finite-dimensional, and we shall denote $V^{\lambda}$ the corresponding complex vector space; $\rho^{\lambda} : K \to \unit(V^{\lambda})$ the representation morphism; $D^{\lambda}=\dim_{\C}V^{\lambda}$ the dimension of the representation; $\chi^{\lambda}(\cdot)=\tr\,\rho^{\lambda}(\cdot)$ the character; and $\hatchi^{\lambda}(\cdot)=\chi^{\lambda}(\cdot)/D^{\lambda}$ the normalized character. An Hermitian scalar product on $\hendo(V^{\lambda})$ is $\scal{M}{N}_{\hendo(V^{\lambda})}=D^{\lambda}\,\tr(M^{\dagger}N)$. For every class $\lambda$ and every function $f \in \leb^{2}(K)$, we set 
$$\widehat{f}(\lambda)=\int_{K}f(k)\,\rho^{\lambda}(k)\,dk;$$
this is an element of $\hendo(V^{\lambda})$. We refer to \cite{BD85,Far08} for a proof of the following results.
\begin{theorem}[Peter-Weyl]
The (non-commutative) Fourier transform $\mathcal{F} : f \mapsto \sum_{\lambda \in \widehat{K}}\widehat{f}(\lambda)$ realizes an isometry and an isomorphism of (non-unital) algebras between $\leb^{2}(K,\eta)$ and $\bigoplus_{\lambda \in \widehat{K}}\hendo(V^{\lambda})$. So, if $f \in \leb^{2}(K)$, then
\begin{align}
f(k)&=\sum_{\lambda \in \widehat{K}} D^{\lambda}\,\tr\left(\widetilde{f}(\lambda)\,\rho^{\lambda}(k)\right)= \sum_{\lambda \in \widehat{K}} D^{\lambda}\,\tr\left(\int_{K}f(h^{-1}k)\,\rho^{\lambda}(h)\,dh\right)\label{fourierexpansion}\\
 \|f\|^{2}_{\leb^{2}(K)}&=\sum_{\lambda\in \widehat{K}} \left\|\widehat{f}(\lambda)\right\|_{\hendo(V^\lambda)}^{2}=\sum_{\lambda\in \widehat{K}} D^{\lambda}\,\tr\left(\widehat{f}(\lambda)^{\dagger}\widehat{f}(\lambda)\right)\label{parseval}
\end{align}
where $\widetilde{f}(\lambda)=\widehat{f^{-}}(\lambda)=\int_{K}f(k^{-1})\,\rho^{\lambda}(k)\,dk$.
\end{theorem}
\bigskip

Assume now that $f$ is in $\leb^{2}(K,\eta)^{K}$, the subalgebra of conjugacy-invariant functions. The Fourier expansion \eqref{fourierexpansion} and the Parseval identity \eqref{parseval} become then
$$f(k)=\sum_{\lambda \in \widehat{K}}(D^{\lambda})^{2}\,\hatchi^{\lambda}(f^{-})\,\hatchi^{\lambda}(k)\qquad;\qquad \|f\|_{\leb^{2}(K)}^{2}=\sum_{\lambda \in \widehat{K}} |\chi^{\lambda}(f)|^{2},$$
and in particular, the irreducible characters $\chi^{\lambda}$ form an orthonormal basis of $\leb^{2}(K)^{K}$. Cartan gave a statement generalizing Theorem \ref{peterweyl} for $\leb^{2}(G/K,\eta)$, where $X=G/K$ is a simply connected irreducible compact symmetric space. Call \emph{spherical} an irreducible representation $(V^{\lambda},\rho^{\lambda})$ of $G$ such that $(V^{\lambda})^{K}$, the space of vectors invariant by $\rho^{\lambda}(K)$, is non-zero. Then, it is in fact one-dimensional, so one can find a vector $e^{\lambda}$ of norm $\|e^{\lambda}\|^{2}=1$, unique up to multiplication by $z \in \Tor$, such that $(V^{\lambda})^{K}=\C e^{\lambda}$. Denote then $\mathscr{C}^{\lambda}(G/K)$ the set of functions from $G$ to $\C$ that can be written as
\begin{equation}
f(g)=f_{v}(g)=\scal{v}{\rho^{\lambda}(g)(e^{\lambda})}\quad\text{with }v\in V^{\lambda}.\label{generalizedmatrixcoeff}
\end{equation}
Such a function is right-$K$-invariant, so it can be considered as a function from $G/K$ to $\C$. 
\begin{theorem}[Cartan]\label{helgason}
Let $\widehat{G}^{K}$ be the set of spherical irreducible representations of $G$. The Hilbert space $\leb^{2}(G/K,\eta)$ is isomorphic to the orthogonal sum $\bigoplus_{\lambda \in \widehat{G}^{K}}\mathscr{C}^{\lambda}(G/K)$. This decomposition corresponds to the Fourier expansion 
\begin{equation}
f(gK)=\sum_{\lambda \in \widehat{G}^{K}}D^{\lambda}\,\tr\left(\int_{G}f(h^{-1}gK)\,\rho^{\lambda}(h)\,dh\right)\label{superfourier}
\end{equation}
for $f \in \leb^{2}(G/K)$.
\end{theorem}
\noindent In each space $\mathscr{C}^{\lambda}(G/K)$, the space of left $K$-invariant functions is one-dimensional, and it is generated by the \emph{zonal spherical function} $ \phi^{\lambda}(gK)=\scal{e^{\lambda}}{\rho^{\lambda}(g)(e^{\lambda})}.$ These spherical functions form an orthogonal basis of $\leb^{2}(X)^{K}$ when $\lambda$ runs over spherical representations. So, a $K$-invariant function writes as
$$f(gK)=\sum_{\lambda \in \widehat{G}^{K}} D^{\lambda}\,\phi^{\lambda}(f^{-})\,\phi^{\lambda}(gK),$$
where $\phi^{\lambda}(f)=\int_{G/K}f(x)\,\phi^{\lambda}(x)\,dx=\scal{e^{\lambda}}{\int_{G} f(gK)\,\rho^{\lambda}(g)(e^{\lambda})\,dg}$. \bigskip

To conclude with, notice that the decomposition of Theorem \ref{helgason} is the decomposition of $\leb^{2}(G/K,\eta)$ in common eigenspaces of the elements of $\mathscr{D}(G/K)$, the commutative algebra of $G$-invariant differential operators on $X$. Thus, there are morphisms of algebras $c^{\lambda} : \mathscr{D}(G/K) \to \C$ such that
$$L(f^{\lambda})=c^{\lambda}(L)\,f^{\lambda}$$
for every $\lambda \in \widehat{G}^{K}$, every $L \in \mathscr{D}(G/K)$ and every $f^{\lambda} \in \mathscr{C}^{\lambda}(G/K)$.\bigskip

\subsection{Highest weight theorem and Weyl's character formula}\label{weyltheory} The theory of highest weights of representations enables us to identify $\widehat{K}$ or $\widehat{G}^{K}$, and to compute the coefficients $c^{\lambda}(\Delta)$ associated to the Laplace-Beltrami operator. If $G$ is a connected compact Lie group, its maximal tori are all conjugated, and every element of $K$ is contained in a maximal torus $T$. Denote $W=\mathrm{Norm}(T)/T$ the \emph{Weyl group} of $G$, and call \emph{weight} of a representation $V$ of $G$ an element of $\tlie^{*}$, or equivalently a group morphism $\omega : T \to \Tor$ such that $V^{\omega}=\{v \in V\,\,|\,\, \forall t \in T,\,\,\,t \cdot v = \omega(t)\cdot v\}\neq 0$. Every representation $V$ of $G$ is the direct sum of its weight subspaces $V^{\omega}$, and this decomposition is always $W$-invariant. Besides, the set of all weights of all representations of $G$ is a lattice $\Z\Omega$ whose rank is also the dimension of $T$. We take a $W$-invariant scalar product on the real vector space $\R\Omega=\Z\Omega \otimes_{\Z}\R$, \emph{e.g.}, the dual of the scalar product given by Equation \eqref{normalization}, where $\R\Omega$ is identified with $\tlie^{*}$ by mean of $\omega \mapsto d_{e}\omega$ for $\omega \in \Z\Omega$. We also fix a closed fundamental set $C$ for the action of the Weyl group on $\R\Omega$. We call \emph{dominant} a weight $\omega$ that falls in the Weyl chamber $C$. A \emph{root} of $G$ is a non-zero weight of the adjoint representation. The set of roots $\Phi$ is a root system, which means that certain combinatorial relations are satisfied between its elements. There is a unique way to split $\Phi$  in a set $\Phi_{+}$ of positive roots and a set $\Phi_{-}=-\Phi_{+}$ such that $$C=\{x \in \R\Omega\,\,|\,\,\forall \alpha \in \Phi_{+},\,\, \scal{x}{\alpha} \geq 0\}.$$  Call \emph{simple} a positive root $\alpha$ that cannot be written as the sum of two positive roots; and simple coroot an element 
$\check{\alpha}=\frac{2\alpha}{\scal{\alpha}{\alpha}}$ with $\alpha$ simple root. Then, a distinguished basis of the lattice $\Z\Omega$ is given by the \emph{fundamental weights} $\varpi_{1},\varpi_{2},\ldots,\varpi_{r}$, the dual basis of the basis of coroots. Hence, the sets of weights and of dominant weights have the following equivalent descriptions: 
\begin{align*}
\Z\Omega&=\bigoplus_{i=1}^{r}\Z\varpi_{i}=\left\{x \in \R\Omega\,\,\big|\,\,\forall \alpha \in \Phi,\,\,\,\frac{\scal{x}{\alpha}}{\scal{\alpha}{\alpha}} \in \Z \right\} ;\\ \mathrm{Dom}(\Z\Omega)&=\bigoplus_{i=1}^{r}\N\varpi_{i}=\left\{x \in \R\Omega\,\,\big|\,\,\forall \alpha \in \Phi,\,\,\,\frac{\scal{x}{\alpha}}{\scal{\alpha}{\alpha}} \in \N \right\}.
\end{align*}
\bigskip

Suppose now that $G$ is a semi-simple simply connected compact Lie group, and consider the partial order induced by the convex set $C$ on $\R\Omega$. Recall that the Weyl group $W$ is a Coxeter group generated by the symmetries along the simple roots $\alpha_{1},\alpha_{2},\ldots,\alpha_{r}$; so in particular, it admits a signature morphism $\eps : W \to \{\pm 1\}$. Weyl's theorem ensures that every irreducible representation $V$ of $G$ has a unique highest weight $\omega_{0}$ for this order, which is then of multiplicity one and determines the isomorphism class of $V$. Moreover, the restriction to $T$ of the irreducible character associated to a dominant weight $\lambda$ is given by
\begin{equation}
\chi^{\lambda}(t)=\frac{\sum_{\sigma \in W} \eps(\sigma)\,\sigma(\lambda+\rho)(t)}{\sum_{\sigma \in W}\eps(\sigma)\,\sigma(\rho)(t)},\label{weylcharacter}
\end{equation}
where $\rho$ is the half-sum of all positive roots, or equivalently the sum of the fundamental weights. This formula degenerates into the dimension formula
\begin{equation}
D^{\lambda}=\dim V^{\lambda}=\frac{\prod_{\alpha \in \Phi_{+}}\scal{\lambda+\rho}{\alpha}}{\prod_{\alpha \in \Phi_{+}}\scal{\rho}{\alpha}}.
\end{equation}
\noindent These results make Equation \eqref{fourierexpansion} essentially explicit in the case of a conjugacy invariant function on a (semi-)simple compact Lie group $K$; in particular, we shall see in a moment that the highest weights are labelled by partitions or similar combinatorial objects in all the classical cases. \bigskip

The case of a compact symmetric space $X=G/K$ of type non-group is a little more involved. Denote $\theta$ an involutive automorphism of a semi-simple simply connected compact Lie group $G$, with $K=G^{\theta}$. Set $P=\{g \in G\,\,|\,\, \theta(g)=g^{-1}\}$; one has then the Cartan decomposition $G=KP$. In addition to the previous assumptions, one assumes that the maximal torus $T \subset G$ is chosen so that $T^{\theta}=T$ and $P \cap T$ is a maximal torus in $P$ (one can always do so up to conjugation of the torus). Then, Cartan-Helgason theorem (\cite[Theorem 4.1]{Hel84}) says that the spherical representations in $\widehat{G}^{K}$ are precisely the irreducible representations in $\widehat{G}$ that are trivial on $K \cap T=T^{\theta}$. This subgroup $T^{\theta}$ of $T\simeq \Tor^{r}$ is always the product of a subtorus $\Tor^{s\leq r}$ with an elementary abelian $2$-group $(\Z/2\Z)^{t}$;  this will correspond to additional conditions on the size and the parity of the parts of the partitions labeling the highest weights in $\widehat{G}^{K}$ (in comparison to $\widehat{G}$), \emph{cf.} \S\ref{explicit}. The corresponding zonal spherical functions $\phi^{\lambda}$ do not have in general an expression as simple as \eqref{weylcharacter}; see however \cite[Part 1]{HS94}. For most of our computations, this will not be a problem, since we shall only use certain properties of the spherical functions --- \emph{e.g.}, their orthogonality and the formula for the dimension $D^{\lambda}$ --- and not their explicit form; see however \S\ref{zonal}. \bigskip

The last ingredient in the computation of the densities is the value of the coefficient $c^\lambda(\Delta)$ such that $$\frac{\Delta(f^{\lambda})}{2}=c^{\lambda}(\Delta)\,f^{\lambda}$$ for every function $f^{\lambda}$ either in $\mathscr{R}^{\lambda}(K)=\mathrm{Vect}(\{k \mapsto (\rho^{\lambda}(k))_{ij},\,\,\,1\leq i,j \leq D^{\lambda}\})$ in the group case, or in $\mathscr{C}^{\lambda}(G/K)$ in the case of a symmetric space. In the group case, by comparing the definition of the Casimir operator \eqref{casimir} with the definition of the Laplace-Beltrami operator \eqref{laplacebeltrami}, one sees that $2c^{\lambda}(\Delta)$ is also $\kappa_{\lambda}$, the constant by which the Casimir operator $C$ acts \emph{via} the infinitesimal representation $d\rho^{\lambda} : U(\klie) \to \hendo(V^{\lambda})$ --- \emph{cf.} the remark at the end of \S\ref{brown}. This constant is equal to
\begin{equation}
\kappa_{\lambda}=-\scal{\lambda+2\rho}{\lambda},\label{kappa}
\end{equation}
see \cite[Equation (3.4)]{Apple11} and the references therein, or \cite{Lev11} and \cite[Chapter12]{Far08} for a case-by-case computation. These later explicit computations follow from the following expressions of the Casimir operators (see \cite[Lemma 1.2]{Lev11}):
\begin{align*}
C_{\mathfrak{so}(n)}&=\sum_{1\leq i<j\leq n} \left(\frac{E_{ij}-E_{ji}}{\sqrt{n}}\right)^{\otimes 2}\\
C_{\mathfrak{su}(n)}&=\frac{1}{n}\sum_{i=1}^{n} \I E_{ii}\otimes \I E_{ii}-\frac{1}{n^{2}}\sum_{i,j=1}^{n} \I E_{ii}\otimes \I E_{jj}+\sum_{1\leq i<j\leq n} \left(\frac{E_{ij}-E_{ji}}{\sqrt{2n}}\right)^{\otimes 2} + \left(\frac{\I E_{ij}+\I E_{ji}}{\sqrt{2n}}\right)^{\otimes 2} 
\\
C_{\mathfrak{usp}(n)}&=\frac{1}{2n}\sum_{i=1}^{n} \I E_{ii}\otimes \I E_{ii} + \mathrm{j} E_{ii}\otimes \mathrm{j} E_{ii}+\mathrm{k} E_{ii}\otimes \mathrm{k} E_{ii}\\
&+\sum_{1\leq i<j\leq n} \left(\frac{E_{ij}-E_{ji}}{\sqrt{4n}}\right)^{\otimes 2}+\left(\frac{\I E_{ij}+\I E_{ji}}{\sqrt{4n}}\right)^{\otimes 2}+\left(\frac{\mathrm{j} E_{ij}+\mathrm{j} E_{ji}}{\sqrt{4n}}\right)^{\otimes 2}+\left(\frac{\mathrm{k} E_{ij}+\mathrm{k} E_{ji}}{\sqrt{4n}}\right)^{\otimes 2}
\end{align*}
where $E_{ij}$ are the elementary matrices in $\mathrm{M}(n,k)$ with $k=\R$, $\C$ or $\Hq$ --- beware that the tensor product are over $\R$, since we deal with real Lie algebras. \bigskip

In the case of a compact symmetric space, the same Formula \eqref{kappa} gives the action of $\Delta^{G/K}$ on $\mathscr{C}^{\lambda}(G/K)$. Indeed, remember that the Riemannian structures on $G$ and $G/K$ are chosen in such a way that for any $f \in \mathscr{C}^{\infty}(G)$ that is right $K$-invariant, $\Delta^{G/K}(f)(gK)=\Delta^{G}(f)(g).$ Consider then a function in $\mathscr{C}^{\lambda}(G/K)$, viewed as a function on $G$. In Definition \eqref{generalizedmatrixcoeff}, $f$ appears clearly as a linear combination of matrix coefficients of the spherical representation $\lambda$, so the previous discussion holds.
\bigskip

\subsection{Densities of a Brownian motion with values in a compact symmetric space}\label{explicit}
Let us now see how the previous results can be used to compute the density $p_{t}^{K}(k)$ or $p_{t}^{X}(x)$ of a Brownian motion on a compact Lie group or symmetric space. These densities are in both cases $K$-invariant, so they can be written as
$$p_{t}^{K}(k)=\sum_{\lambda \in \widehat{K}} a_{\lambda}(t)\,\hatchi^{\lambda}(k) \quad \text{or} \quad p_{t}^{X}(x)=\sum_{\lambda \in \widehat{G}^{K}} a_{\lambda}(t)\,\phi^{\lambda}(x)$$
by using either Peter-Weyl's theorem in the case of conjugacy-invariant functions on $K$, or Cartan's theorem in the case of left $K$-invariant functions on $G/K$. We then apply $\frac{\Delta}{2}=\left.\frac{dP_{t}}{dt}\right|_{t=0}$ to these formulas:
$$
\frac{\Delta p_{t}^{K}(k)}{2}  =\sum_{\lambda \in \widehat{K}} \frac{\kappa_{\lambda}}{2}\,a_{\lambda}(t)\,\hatchi^{\lambda}(k)= \frac{dp_{t}^{K}(t)}{dt} = \sum_{\lambda \in \widehat{K}}\frac{da_{\lambda}(t)}{dt}\,\hatchi^{\lambda}(k),
$$
and similarly in the case of a compact symmetric space. Thus, $\frac{da_{\lambda}(t)}{dt}=\frac{\kappa_{\lambda}}{2}a_{\lambda}(t)$ and $a_{\lambda}(t)=a_{\lambda}(0)\,\E^{\frac{\kappa_{\lambda}}{2}t}$ for every class $\lambda$. The coefficient $a_{\lambda}(0)$ is given in the group case by
$$a_{\lambda}(0)=(D^{\lambda})^{2}\int_{K} \hatchi^{\lambda}(k)\,\delta_{e_{K}}(dk)=(D^{\lambda})^{2}\,\hatchi^{\lambda}(e_{K})=(D^{\lambda})^{2}$$
and in the case of a compact symmetric space of type non-group by
$$a_{\lambda}(0)=D^{\lambda}\, \scal{e^{\lambda}}{\int_{G} \rho^{\lambda}(g)(e^{\lambda})\,\delta_{e_{G}}(dg)}=D^{\lambda}\,\phi^{\lambda}(e_{G})=D^{\lambda}.$$
\medskip
\begin{proposition}\label{abstractdensity}
The density of the law $\mu_{t}$ of the Brownian motion traced on a classical simple compact Lie group $K$ is
$$p_{t}^{K}(k)=\sum_{\lambda \in \widehat{K}} \E^{-\frac{t}{2}\scal{\lambda+2\rho}{\lambda}}\,\left(\frac{\prod_{\alpha \in \Phi_{+}} \scal{\lambda+\rho}{\alpha}  }{\prod_{\alpha \in \Phi_{+}} \scal{\rho}{\alpha}}  \right)^{\!2}\hatchi^{\lambda}(k),$$
and the density of the Brownian motion traced on a classical simple compact symmetric space $G/K$ is
$$p_{t}^{X}(x)=\sum_{\lambda \in \widehat{G}^{K}} \E^{-\frac{t}{2}\scal{\lambda+2\rho}{\lambda}}\,\left(\frac{\prod_{\alpha \in \Phi_{+}} \scal{\lambda+\rho}{\alpha}  }{\prod_{\alpha \in \Phi_{+}} \scal{\rho}{\alpha}}  \right)\phi^{\lambda}(x).$$
\end{proposition}
\bigskip

Let us now apply this in each classical case. We refer to \cite{BD85}, \cite[Chapter 24]{FH91} and \cite[Chapter 10]{Hel78} for most of the computations. Unfortunately, we have not found a reference which describes explicitly the spherical representations;  this explains the following long discussion.  For convenience,  we shall assume:\vspace{2mm}
\begin{itemize}
\item $n \geq 2$ when considering $\SU(n)$, $\SU(n)/\SO(n)$, $\SU(2n)/\unit\SP(n)$ or $\SU(n)/\mathrm{S}(\unit(n-q)\times \unit(q))$;  \vspace{2mm}
\item $n \geq 3$ when considering $\unit\SP(n)$, $\unit\SP(n)/\unit(n)$ or $\unit\SP(n)/(\unit\SP(n-q) \times \unit\SP(q))$;  \vspace{2mm}
\item $n \geq 10$ when considering $\SO(n)$, $\SO(2n)/\unit(n)$ or $\SO(n)/(\SO(n-q)\times \SO(q))$.\vspace{2mm}
\end{itemize}
For $\SU(2n)/\unit\SP(n)$ and $\SO(2n)/\unit(n)$, the restriction will hold on the ``$2n$'' parameter of the group of isometries.
These assumptions shall ensure that the root systems and the Schur functions of type $\mathrm{B}$, $\mathrm{C}$ and $\mathrm{D}$ are not degenerate, and later this will ease certain computations. For Grassmanian varieties, we shall also suppose by symmetry that $q \leq \lfloor \frac{n}{2} \rfloor$.\vspace{2mm}
\bigskip

\subsubsection{Special unitary groups and their quotients}

In $\SU(n,\C)$, a maximal torus is 
$$
T=\left\{\diag(z_{1},z_{2},\ldots,z_{n})\,\,\bigg|\,\,\forall i,\,\,z_{i} \in \Tor \text{ and } \prod_{i=1}^{n}z_{i}=1\right\}=\Tor^{n}/\Tor, 
$$ and the Weyl group is the symmetric group $\sym_{n}$. The simple roots and the fundamental weights, viewed as elements of $\tlie^{*}$, are $\alpha_{i}=e^{i}-e^{i+1}$ and 
$$\varpi_{i}=\frac{n-i}{n}(e^{1}+\cdots+e^{i})-\frac{i}{n}(e^{i+1}+\cdots+e^{n})$$
for $i \in \lle 1,n-1\rre$, where $e^{i}$ is the coordinate form on $\tlie=\I\R^{n}$ defined by $e^{i}(\diag(\I t_{1},\I t_{2},\ldots,\I t_{n}))=t_{i}$.
The dominant weights are then the 
$$(\lambda_{1}-\lambda_{2})\varpi_{1}+\cdots +\lambda_{n-1}\varpi_{n-1}=\lambda_{1}e^{1}+\cdots +\lambda_{n-1}e^{n-1}-|\lambda|\,\frac{\varpi_{n}}{n},$$
 where $\lambda = (\lambda_{1}\geq \lambda_{2}\geq \cdots \geq \lambda_{n-1})$ is any partition (non-increasing sequence of non-negative integers) of length $(n-1)$;  it is then convenient to set $\lambda_{n}=0$. The half-sum of positive roots is given by $2\rho=2(\varpi_{1}+\cdots + \varpi_{n-1})=\sum_{i=1}^{n}(n+1-2i)e^{i}$, and the scalar product on $\tlie^{*}$ is $\frac{1}{n}$ times the usual euclidian scalar product $\scal{e^{i}}{e^{j}}=\delta_{ij}$. So,
$$
D^{\lambda}=\prod_{1\leq i<j \leq n}\frac{\lambda_{i}-\lambda_{j}+j-i}{j-i}\qquad;\qquad \chi^{\lambda}(k)=s_{\lambda}(z_{1},\ldots,z_{n})=\frac{\det(z_{i}^{\lambda_{j}+n-j})_{1\leq i,j \leq n}}{\det(z_{i}^{n-j})_{1\leq i,j \leq n}},
$$
where $z_{1},\ldots,z_{n}$ are the eigenvalues of $k$; thus, characters are given by Schur functions. The Casimir coefficient is $$-\kappa_\lambda=-\frac{|\lambda|^{2}}{n^{2}}+\frac{1}{n}\sum_{i=1}^{n}\lambda_{i}^{2}+(n+1-2i)\lambda_{i},$$
where $|\lambda|=\sum_{i=1}^{n}\lambda_{i}$ denotes the size of the partition.
\bigskip

Though we have chosen to examine only the Brownian motions on simple Lie groups, the same work can be performed over the unitary groups $\unit(n,\C)$, which are reducible Lie groups. Irreducible representations of $\unit(n,\C)$ are labelled by sequences $\lambda=(\lambda_{1}\geq \cdots \geq \lambda_{n})$ in $\Z^{n}$, the action of the torus $\Tor^{n}$ on a corresponding highest weight vector being given by the morphism $\lambda(z_{1},\ldots,z_{n})=z_{1}^{\lambda_{1}}\cdots z_{n}^{\lambda_{n}}$. The dimensions and characters are the same as before, and the Casimir coefficient is  $\frac{1}{n}\sum_{i=1}^{n}\lambda_{i}^{2}+(n+1-2i)\lambda_{i}$.\bigskip

For the spaces of quaternionic structures $\SU(2n,\C)/\unit\SP(n,\Hq)$, the involutive automorphism defining the symmetric pair is  $\theta(g)=J_{2n}\,\overline{g}\,J_{2n}^{-1}$, where $J_{2n}$ is the skew symmetric matrix 
$$ J_{2n}=\begin{pmatrix} 0 & 1 &         &     & \\
				-1 & 0 &         &     & \\
				    &    &\ddots&    &  \\
				    &    &         & 0  & 1 \\
				    &    &         & -1 & 0
\end{pmatrix}$$
of size $2n$. The subgroup $T^{\theta}$ is the set of matrices $\diag(z_{1},z_{1}^{-1},\ldots,z_{n},z_{n}^{-1})$, with all the $z_{i}$'s in $\Tor$. The dominant weights $\lambda$ trivial on $T^{ \theta}$ correspond then to partitions will all parts doubled: $$\forall i \in \lle 1,n\rre,\,\,\,\lambda_{2i-1}=\lambda_{2i}.$$\vspace{1mm}

In the spaces of real structures $\SU(n,\C)/\SO(n,\R)$, $\theta(g)=\overline{g}$. The intersection of the torus with $\SO(n,\R)$ is isomorphic to $(\Z/2\Z)^{n}/(\Z/2\Z)$, and therefore, by Cartan-Helgason theorem, the spherical representations correspond to partitions with even parts: $$\forall i \in \lle 1,n\rre,\,\,\,\lambda_{i} \equiv 0 \bmod 2.$$\vspace{1mm}

Finally, for the complex Grassmannian varieties $\SU(n,\C)/\mathrm{S}(\unit(n-q,\C) \times \unit(q,\C))$, it is a little simpler to work with $\unit(n,\C)/(\unit(n-q,\C)\times \unit(q,\C))$, which is the same space. An involutive automorphism defining the symmetric pair is then $\theta(g)=K_{n,q}\,g \,K_{n,q}$, where 
$$K_{n,q}=\begin{pmatrix}
&&T_{q} \\
&I_{n-2q}&\\
T_{q} & &  
\end{pmatrix}$$ 
and $T_{q}$ is the $(q\times q)$-anti-diagonal matrix with entries $1$ on the anti-diagonal. The subgroup $T^{\theta}$ is then the set of diagonal matrices $\diag(z_{1},\ldots,z_{q},z_{q+1},\ldots,z_{n-q},z_{q},\ldots,z_{1})$ with the $z_{i}$'s in $\Tor$. The dominant weights $\lambda$ trivial on $T^{\theta}$ correspond then to partitions of length $q$, written as
$$\lambda=(\lambda_{1},\ldots,\lambda_{q},0,\ldots,0,-\lambda_{q},\ldots,-\lambda_{1}).$$
\vspace{-3mm}

\subsubsection{Compact symplectic groups and their quotients}
Considering $\unit\SP(n,\Hq)$ as a subgroup of $\SU(2n,\C)$, a maximal torus is
$$T=\left\{ \diag(z_{1},z_{1}^{-1},\ldots,z_{n},z_{n}^{-1})\,\,\big|\,\,\forall i, \,\,z_{i}\in \Tor \right\},$$
and the Weyl group is the hyperoctahedral group $\mathfrak{H}_{n}=(\Z/2\Z)\wr\sym_{n}$. The simple roots, viewed as elements of $\mathfrak{t}^{*}$, are $\alpha_{i}=e^{i}-e^{i+1}$ for $i \in \lle 1,n-1\rre$ and $\alpha_{n}=2e^{n}$; and the fundamental weights are $\varpi_{i}=e^{1}+\cdots+e^{i}$ for $i \in \lle 1,n\rre$. Here, $e^{i}(\diag(\I t_{1},-\I t_{1},\ldots,\I t_{n},-\I t_{n}))=t_{i}$. The dominant weights can therefore be written as $\lambda_{1}e^{1}+\cdots + \lambda_{n}e^{n},$ where $\lambda=(\lambda_{1}\geq \lambda_{2}\geq \cdots \geq \lambda_{n})$ is any partition of length $n$. This leads to
\begin{align*}
D^{\lambda}&=\prod_{1\leq i<j \leq n}\frac{\lambda_{i}-\lambda_{j}+j-i}{j-i} \prod_{1\leq i\leq j \leq n} \frac{\lambda_{i}+\lambda_{j}+2n+2-i-j}{2n+2-i-j};\\
\chi^{\lambda}(k)&=sc_{\lambda}(z_{1},z_{1}^{-1},\ldots,z_{n},z_{n}^{-1})=\frac{\det(z_{i}^{\lambda_{j}+n-j+1}-z_{i}^{-(\lambda_{j}+n-j+1)})_{1\leq i,j \leq n}}{\det(z_{i}^{n-j+1}-z_{i}^{-(n-j+1)})_{1\leq i,j \leq n}},
\end{align*}
where $z_{1}^{\pm1},\ldots,z_{n}^{\pm 1}$ are the eigenvalues of $k$ viewed as a matrix in $\SU(2n,\C)$. The Casimir coefficient is $-\kappa_{\lambda}=\frac{1}{2n}\sum_{i=1}^{n}\lambda_{i}^{2}+(2n+2-2i)\lambda_{i}$. 
\bigskip

In the spaces of complex structures $\unit\SP(n,\Hq)/\unit(n,\C)$, $\theta(g)=\overline{g}$ (inside $\SU(2n,\C)$). The subgroup $T^{\theta}$ is isomorphic to $(\Z/2\Z)^{n}$, so the spherical representations correspond here again to partitions with even parts. On the other hand, for quaternionic Grassmannian varieties $\unit\SP(n,\Hq)/(\unit\SP(n-q,\Hq)\times \unit\SP(q,\Hq))$, a choice for the involutive automorphism is $\theta(g)=L_{2n,q}\,g\,L_{2n,q}$, where
$$L_{2n,q}=\begin{pmatrix}
T_{4} & & & \\
 & \ddots & & \\
  & & T_{4} & \\
  & &           & I_{2n-4q}
\end{pmatrix},$$
$T_{4}$ appearing $q$ times (with all the computations made inside $\SU(2n,\C)$). Then, $T^{\theta}$ is the set of diagonal matrices $\diag(z_{1},z_{1}^{-1},z_{1}^{-1},z_{1},\ldots,z_{q},z_{q}^{-1},z_{q}^{-1},z_{q},z_{2q+1},z_{2q+1}^{-1},\ldots,z_{n},z_{n}^{-1})$ with the $z_{i}$'s in $\Tor$. The dominant weights $(\lambda_{1},\ldots,\lambda_{n})$ trivial on $T^{\theta}$ write therefore as partitions of length $q$ with all parts doubled:
$$\lambda=(\lambda_{1},\lambda_{1},\ldots,\lambda_{q},\lambda_{q},0,\ldots,0).$$
\vspace{1mm}

\subsubsection{Special orthogonal groups and their quotients}
Odd and even special orthogonal groups do not have the same kind of root system, and on the other hand, $\SO(n,\R)$ is not simply connected and has for fundamental group $\Z/2\Z$ for $n \geq 3$. So in theory, the arguments previously recalled apply only for the universal cover $\mathrm{Spin}(n)$. Nonetheless, most of the results will stay true, and in particular the labeling of the irreducible representations; see the end of \cite[Chapter 5]{BD85} for details on this question. In the odd case, a maximal torus in $\SO(2n+1,\R)$ is 
$$
T = \left\{\diag(R_{\theta_{1}},\ldots,R_{\theta_{n}},1) \,\,\bigg|\,\, \forall i,\,\,R_{\theta_{i}}=\left(\begin{smallmatrix} \cos \theta_{i} &-\sin \theta_{i}\\ \sin \theta_{i} & \cos \theta_{i} \end{smallmatrix} \right) \in \SO(2,\R)\right\},$$
and the Weyl group is again the hyperoctahedral group $\mathfrak{H}_{n}$. The simple roots are $\alpha_{i}=e^{i}-e^{i+1}$ for $i \in \lle 1,n-1\rre$, and $\alpha_{n}=e^{n}$; and the fundamental weights are $\varpi_{i}=e^{1}+\cdots+e^{i}$ for $i \in \lle 1,n-1\rre$, and $\varpi_{n}=\frac{1}{2}(e^{1}+\cdots+e^{n})$. Here, 
$$e^{i}\left(\diag\left(\left(\begin{smallmatrix} 0 &-a_{1} \\ a_{1} & 0\end{smallmatrix} \right),\ldots,\left(\begin{smallmatrix} 0 &-a_{n} \\ a_{n} & 0\end{smallmatrix} \right),0\right) \right)= a_{i}$$
and it corresponds to the morphism $\diag(R_{\theta_{1}},\ldots,R_{\theta_{n}},1) \mapsto \E^{\I\theta_{i}}$. The dominant weights are then the  $\lambda_{1}e^{1}+\cdots + \lambda_{n}e^{n}$, where $\lambda=(\lambda_{1}\geq \lambda_{2}\geq \cdots \geq \lambda_{n})$ is either a partition of length $n$, or an half-partition of length $n$, where by half-partition we mean a non-increasing sequence  of half-integers in $\N'=\N+1/2$. So, one obtains
\begin{align*}
D^{\lambda}&=\prod_{1\leq i<j \leq n}\frac{\lambda_{i}-\lambda_{j}+j-i}{j-i} \prod_{1\leq i\leq j \leq n} \frac{\lambda_{i}+\lambda_{j}+2n+1-i-j}{2n+1-i-j} ; \\
\chi^{\lambda}(k)&=sb_{\lambda}(z_{1},z_{1}^{-1},\ldots,z_{n},z_{n}^{-1},1)=\frac{\det(z_{i}^{\lambda_{j}+n-j+1/2}-z_{i}^{-(\lambda_{j}+n-j+1/2)})_{1\leq i,j \leq n}}{\det(z_{i}^{n-j+1/2}-z_{i}^{-(n-j+1/2)})_{1\leq i,j \leq n}},
\end{align*}
where $z_{1}^{\pm 1},\ldots,z_{n}^{\pm 1},1$ are the eigenvalues of $k$. The Casimir coefficient associated to the highest weight $\lambda$ is $-\kappa_{\lambda}=\frac{1}{2n+1}\sum_{i=1}^{n} \lambda_{i}^{2}+(2n+1-2i)\lambda_{i}$. 
\bigskip

In the even case, a maximal torus in $\SO(2n,\R)$ is
$$
T = \left\{\diag(R_{\theta_{1}},\ldots,R_{\theta_{n}}) \,\,\bigg|\,\, \forall i,\,\,R_{\theta_{i}}=\left(\begin{smallmatrix} \cos \theta_{i} &-\sin \theta_{i}\\ \sin \theta_{i} & \cos \theta_{i} \end{smallmatrix} \right) \in \SO(2,\R)\right\}$$
and the Weyl group is $\mathfrak{H}_{n}^{+}$, the subgroup of $\mathfrak{H}_{n}$ of index $2$ consisting in signed permutations with an even number of signs $-1$. The simple roots are $\alpha_{i}=e^{i}-e^{i+1}$ for $i \in \lle 1,n-1\rre$ and $\alpha_{n}=e^{n-1}+e^{n}$; and the fundamental weights are $\varpi_{i}=e^{1}+\cdots+e^{i}$ for $i \in \lle 1,n-2\rre$ and $\varpi_{n-1,n}=\frac{1}{2}(e^{1}+\cdots+e^{n-1}\pm e^{n})$. The dominant weights are then $\lambda_{1}e^{1}+\cdots+ \lambda_{n-1}e^{n-1}+\eps \lambda_{n}e^{n}$, where $\eps$ is a sign and $(\lambda_{1}\geq \cdots \geq \lambda_{n})$ is either a partition or an half-partition of length $n$. So,
\begin{align*}
D^{\lambda}&=\prod_{1\leq i<j \leq n}\frac{\lambda_{i}-\lambda_{j}+j-i}{j-i} \,\frac{\lambda_{i}+\lambda_{j}+2n-i-j}{2n-i-j}, \\
\chi^{\lambda}(k)&=sd_{\lambda}(z_{1},z_{1}^{-1},\ldots,z_{n},z_{n}^{-1})\\
&=\frac{\det(z_{i}^{\lambda_{j}+n-j}-z_{i}^{-(\lambda_{j}+n-j)})_{1\leq i,j \leq n}+\det(z_{i}^{\lambda_{j}+n-j}+z_{i}^{-(\lambda_{j}+n-j)})_{1\leq i,j \leq n}}{\det(z_{i}^{n-j}+z_{i}^{-(n-j)})_{1\leq i,j \leq n}},
\end{align*}
and $-\kappa_{\lambda}=\frac{1}{2n}\sum_{i=1}^{n} \lambda_{i}^{2}+(2n-2i)\lambda_{i}$. \bigskip

For real Grassmannian varieties $\SO(n,\R)/(\SO(n-q,\R)\times\SO(q,\R))$ and for  spaces of complex structures $\SO(2n,\R)/\unit(n,\C)$, one cannot directly apply the Cartan-Helgason theorem, since $\SO(n,\R)$ is not simply connected. A rigorous way to deal with this problem is to first look at quotients of the spin group $\mathrm{Spin}(n)$. For instance, consider the Grassmannian variety of \emph{non-oriented} vector spaces 
$\Gra^{\pm}(n,q,\R)\simeq \mathrm{Spin}(n)/(\mathrm{Spin}(n-q)\times \mathrm{Spin}(q));$ 
$\Gra(n,q,\R)$ is a $2$-fold covering of $\Gra^{\pm}(n,q,\R)$. The defining map of $\Gra^{\pm}(n,q,\R)$ corresponds to the involution of $\SO(n,\R)$ given by $\theta(g)=N_{n,q}\,g\,N_{n,q}$, where
$$N_{n,q}= \begin{pmatrix}
T_{2} & & & \\
& \ddots & &\\
& & T_{2} & \\
& & 		& I_{n-2q}
\end{pmatrix}$$
with $q$ blocks $T_{2}$. Then $T^{\theta}$ is $(\Z/2\Z)^{q} \times (\SO(2,\R))^{\lfloor \frac{n}{2} \rfloor-q}$, so the dominant weights trivial on $T^{\theta}$  write as $\lambda=(\lambda_{1},\ldots,\lambda_{q},0,\ldots,0)$, with $\lambda_{i} \equiv 0 \bmod 2$ for all $i \in \lle 1,q\rre$.
They are therefore given by an integer partition of length $q$, with all parts even. Now, for the simply connected Grassmannian variety $\Gra(n,q,\R) $, there are twice as many spherical representations, as $T^{\theta}$ is in this case isomorphic to $((\Z/2\Z)^{q}/(\Z/2\Z)) \times \Tor^{\lfloor \frac{n}{2} \rfloor-q}$, instead of $(\Z/2\Z)^{q}\times \Tor^{\lfloor \frac{n}{2} \rfloor-q}$. Therefore, the condition of parity is now
$$\forall i,j \in \lle 1,q \rre,\,\,\,\lambda_{i} \equiv \lambda_{j} \bmod 2.$$  
Similar considerations show that for the spaces $\SO(2n,\R)/\unit(n,\C)$, the dominant weights $\lambda$ trivial on the intersection $T^{\theta}$ are given by
$$\lambda=(\lambda_{1},\lambda_{1},\ldots,\lambda_{m},\lambda_{m}) \,\text{ or }\, \lambda=(\lambda_{1},\lambda_{1},\ldots,\lambda_{m},\lambda_{m},0)$$
that is to say a partition with all non-zero parts that are doubled.
\bigskip

\subsubsection{Summary} Let us summarize the previous results (this is redundant, but very useful in order to follow all the computations of Section \ref{upper}). We denote: $\ym_{n}$ the set of partitions of length $n$; $\mathfrak{Z}_{n}$ the set of non-decreasing sequences of (possibly negative) integers; $\frac{1}{2}\ym_{n}$ the set of partitions and half-partitions of length $n$; 
$2\ym_{n}$ the set of partitions of length $n$ with even parts; $2\ym_{n} \boxplus 1$ the set of partitions of length $n$ with odd parts; and $\ym\ym_{n}$ the set of partitions of length $n$ and with all non-zero parts doubled. It is understood that if $i$ is too big, then $\lambda_{i}=0$ for a partition or an half-partition $\lambda$ of prescribed length.
\begin{theorem}\label{explicitdensity}
The density of the law $\mu_{t}$ of the Brownian motion traced on a classical simple compact Lie group writes as:\begin{align*}
\sum_{\lambda \in \ym_{n-1}}\, &\E^{-\frac{t}{2n}\,\left(-\frac{|\lambda|^{2}}{n}+\sum_{i=1}^{n-1}\lambda_{i}^{2}+(n+1-2i)\lambda_{i}\right)}\left(\prod_{1\leq i<j \leq n}\frac{\lambda_{i}-\lambda_{j}+j-i}{j-i}\right) s_{\lambda}(k);\\
\sum_{\lambda \in \mathfrak{Z}_{n}}\,\,\,\, &\E^{-\frac{t}{2n}\,\sum_{i=1}^{n}\lambda_{i}^{2}+(n+1-2i)\lambda_{i}}\left(\prod_{1\leq i<j \leq n}\frac{\lambda_{i}-\lambda_{j}+j-i}{j-i}\right) s_{\lambda}(k);\\
\sum_{\lambda \in \ym_{n}}\,\,\,&\E^{-\frac{t}{4n}\,\sum_{i=1}^{n}\lambda_{i}^{2}+(2n+2-2i)\lambda_{i}}\left(\prod_{1\leq i<j \leq n}\frac{\lambda_{i}-\lambda_{j}+j-i}{j-i}\prod_{1\leq i \leq j \leq n}\frac{\lambda_{i}+\lambda_{j}+2n+2-i-j}{2n+2-i-j}\right)sc_{\lambda}(k);\\
\sum_{\lambda \in \frac{1}{2}\ym_{n}}\,&\E^{-\frac{t}{4n+2}\,\sum_{i=1}^{n}\lambda_{i}^{2}+(2n+1-2i)\lambda_{i}}\left(\prod_{1\leq i<j \leq n}\frac{\lambda_{i}-\lambda_{j}+j-i}{j-i}\prod_{1\leq i \leq j \leq n}\frac{\lambda_{i}+\lambda_{j}+2n+1-i-j}{2n+1-i-j}\right)sb_{\lambda}(k);\\
\sum_{\lambda\in \frac{1}{2}\ym_{n}}\,&\E^{-\frac{t}{4n}\,\sum_{i=1}^{n}\lambda_{i}^{2}+(2n-2i)\lambda_{i}}\left(\prod_{1\leq i<j \leq n}\frac{(\lambda_{i}-\lambda_{j}+j-i)(\lambda_{i}+\lambda_{j}+2n-i-j)}{(j-i)(2n-i-j)}\right)(sd_{\lambda}(k)+sd_{\eps\lambda}(k))
\end{align*}
respectively for special unitary groups $\SU(n,\C)$, unitary groups $\unit(n,\C)$, symplectic groups $\unit\SP(n,\Hq)$, odd special orthogonal groups $\SO(2n+1,\R)$, and even special orthogonal groups $\SO(2n,\R)$. In this last case, $\eps\lambda=(\lambda_{1},\ldots,\lambda_{n-1},-\lambda_{n})$, and it is agreed that $sd_{\lambda}+sd_{\eps\lambda}$ stands for $sd_{\lambda}$ if $\lambda_{n}=0$.\bigskip

We denote generically $\phi_{\lambda}(x)$ a zonal spherical function associated to a spherical representation (the function depends of course of the implicit type of the space considered). The density of the law $\mu_{t}$ of the Brownian motion traced on a classical simple compact symmetric space writes then as follows:
\begin{align*}
\sum_{\lambda \in 2\ym_{n-1}} &\E^{-\frac{t}{2n}\,\left(-\frac{|\lambda|^{2}}{n}+\sum_{i=1}^{n-1}\lambda_{i}^{2}+(n+1-2i)\lambda_{i}\right)}\left(\prod_{1\leq i<j \leq n}\frac{\lambda_{i}-\lambda_{j}+j-i}{j-i}\right) \phi_{\lambda}(x);\end{align*}
\begin{align*}
\sum_{\lambda \in \ym\ym_{2n-1}} \!\!&\E^{-\frac{t}{4n}\,\left(-\frac{|\lambda|^{2}}{2n}+\sum_{i=1}^{2n-2}\lambda_{i}^{2}+(2n+1-2i)\lambda_{i}\right)}\left(\prod_{1\leq i<j \leq 2n}\frac{\lambda_{i}-\lambda_{j}+j-i}{j-i}\right) \phi_{\lambda}(x);\\
\sum_{\lambda \in \ym_{q}} \,\,\,\,&\E^{-\frac{t}{n}\,\sum_{i=1}^{q} \lambda_{i}^{2}+(n+1-2i)\lambda_{i}}\left(\prod_{1\leq i<j \leq n}\frac{\lambda_{i}-\lambda_{j}+j-i}{j-i}\right) \phi_{\lambda}(x);\\
\sum_{\lambda \in 2\ym_{n}}  \,\,&\E^{-\frac{t}{4n}\,\sum_{i=1}^{n}\lambda_{i}^{2}+(2n+2-2i)\lambda_{i}}\left(\prod_{1\leq i<j \leq n}\frac{\lambda_{i}-\lambda_{j}+j-i}{j-i}\prod_{1\leq i \leq j \leq n}\frac{\lambda_{i}+\lambda_{j}+2n+2-i-j}{2n+2-i-j}\right)\phi_{\lambda}(x);\\
\sum_{\lambda \in \ym\ym_{2q}}  \,&\E^{-\frac{t}{4n}\,\sum_{i=1}^{2q}\lambda_{i}^{2}+(2n+2-2i)\lambda_{i}}\left(\prod_{1\leq i<j \leq n}\frac{\lambda_{i}-\lambda_{j}+j-i}{j-i}\prod_{1\leq i \leq j \leq n}\frac{\lambda_{i}+\lambda_{j}+2n+2-i-j}{2n+2-i-j}\right)\phi_{\lambda}(x);\\
\sum_{\lambda\in \ym\ym_{n} } \,&\E^{-\frac{t}{4n}\,\sum_{i=1}^{n}\lambda_{i}^{2}+(2n-2i)\lambda_{i}}\left(\prod_{1\leq i<j \leq n}\frac{(\lambda_{i}-\lambda_{j}+j-i)(\lambda_{i}+\lambda_{j}+2n-i-j)}{(j-i)(2n-i-j)}\right)\phi_{\lambda}(x);\\
\sum_{\lambda \in 2\ym_{q} \sqcup 2\ym_{q}\boxplus 1}\!\!\!\!&\E^{-\frac{t}{4n+2}\,\sum_{i=1}^{q}\lambda_{i}^{2}+(2n+1-2i)\lambda_{i}}\!\left(\prod_{1\leq i<j \leq n}\!\frac{\lambda_{i}-\lambda_{j}+j-i}{j-i}\prod_{1\leq i \leq j \leq n}\frac{\lambda_{i}+\lambda_{j}+2n+1-i-j}{2n+1-i-j}\right)\!\phi_{\lambda}(x);\\
\sum_{\lambda\in 2\ym_{q} \sqcup2\ym_{q}\boxplus 1} \!\!\!\!&\E^{-\frac{t}{4n}\,\sum_{i=1}^{q}\lambda_{i}^{2}+(2n-2i)\lambda_{i}}\!\left(\prod_{1\leq i<j \leq n}\frac{(\lambda_{i}-\lambda_{j}+j-i)(\lambda_{i}+\lambda_{j}+2n-i-j)}{(j-i)(2n-i-j)}\right)\!\phi_{\lambda}(x)
\end{align*}
for real structures $\SU(n,\C)/\SO(n,\R)$, quaternionic structures $\SU(2n,\C)/\unit\SP(n,\Hq)$, complex Grassmannian varieties $\Gra(n,q,\C)$, complex structures $\unit\SP(n,\Hq)/\unit(n,\C)$, quaternionic Grassmannian varieties $\Gra(n,q,\Hq)$, complex structures $\SO(2n,\R)/\unit(n,\C)$, odd real Grassmannian varieties $\Gra(2n+1,q,\R)$ and even real Grassmannian varieties $\Gra(2n,q,\R)$.
\end{theorem}

\begin{remark} In the case of complex Grassmannian varieties, it is understood that $\lambda_{n+1-i}=-\lambda_{i}$ as explained before. We have not tried to reduce the expressions in the previous formulas, so some simplifications can be made by replacing the indexing sets of type $2\ym_{p}$ or $\ym\ym_{p}$ by $\ym_{p}$. On the other hand, it should be noticed that in each case, the ``degree of freedom'' in the choice of partitions labeling the irreducible or spherical representations is exactly the rank of the Riemannian variety under consideration, that is to say the maximal dimension of flat totally geodesic sub-manifolds.
\end{remark}
\bigskip

\begin{example}[Brownian motions on spheres and projective spaces]
Let us examine the case $q=1$ for Grassmannian varieties: it corresponds to real spheres $\mathbb{S}^{n}(\R)=\SO(n+1,\R)/\SO(n,\R)$, to complex projective spaces $\mathbb{P}^{n}(\C)=\SU(n+1,\C)/\mathrm{S}(\unit(n,\C)\times \unit(1,\C))$ and to quaternionic projective spaces $\mathbb{P}^{n}(\Hq)=\unit\SP(n+1,\Hq)/(\unit\SP(n,\Hq)\times \unit\SP(1,\Hq))$. In each case, spherical representations are labelled by a single integer $k \in \N$. So, the densities are:
\begin{align}
p_{t}^{\mathbb{S}^{n}(\R)}(x)&=\sum_{k=0}^{\infty} \E^{-\frac{k(k+n-1)\,t}{2n+2}}\,\frac{(n-2+k)!}{(n-1)!\,k!}\,(2k+n-1)\,\,\phi_{n,k}^{\R}(x);\label{realsphere}\\
p_{t}^{\mathbb{P}^{n}(\C)}(x)&=\sum_{k=0}^{\infty} \E^{-\frac{k(k+n)\,t}{n+1}}\,\frac{((n-1+k)!)^{2}}{(n-1)!\,n!\,(k!)^{2}}\,(2k+n)\,\,\phi_{n,k}^{\C}(x);\label{complexprojective}\\
p_{t}^{\mathbb{P}^{n}(\Hq)}(x)&=\sum_{k=0}^{\infty}\E^{-\frac{k(k+2n+1)\,t}{2(n+1)}}\,\frac{(2n+k)!\,(2n-1+k)!}{(2n+1)!\,(2n-1)!\,(k+1)!\,k!}\,(2k+2n+1)\,\,\phi_{n,k}^{\Hq}(x).\label{quaternionicprojective}
\end{align}
In particular, one recovers the well-known fact that, up to the aforementioned normalization factor $(n+1)$, the eigenvalues of the Laplacian on the $n$-sphere are the $k(k+n-1)$, each with multiplicity $\frac{(n-2+k)!}{(n-1)!\,k!}\,(2k+n-1)$; see \emph{e.g.} \cite[\S3.3]{SC94}.
\end{example}
\medskip

\begin{example}[Torus and Fourier analysis]
Take the circle $\Tor=\unit(1,\C)=\mathbb{S}^{1}(\R)$. The Brownian motion on $\Tor$ is the projection of the real Brownian motion of density $p_{t}^{\R}(\theta)=\frac{1}{\sqrt{2\pi t}}\,\E^{-\theta^{2}/2t}$ by the map $\theta \mapsto \E^{\I \theta}$. Thus,
$$
p_{t}^{\Tor}(\E^{\I \theta})=2\pi\,\sum_{m=-\infty}^{\infty} p_{t}^{\R}(\theta+2m\pi)= \sqrt{\frac{2\pi}{t}}\sum_{m=-\infty}^{\infty}\E^{-\frac{(\theta+2m\pi)^{2}}{2t}}= \sqrt{\frac{2\pi}{t}}\,S(\theta,t).
$$
The series $S(\theta,t)$ is smooth and $2\pi$-periodic, so it is equal to its Fourier series $\sum_{n=-\infty}^{\infty} c_{k}(S(t))\,\E^{k\I \theta}$, with
$$
c_{k}(S(t))=\sum_{m=-\infty}^{\infty}\int_{0}^{2\pi} \E^{-\frac{(\theta+2m\pi)^{2}}{2t}}\,\E^{-k\I \theta}\,\frac{d\theta}{2\pi}=\frac{1}{2\pi}\int_{\R}\E^{-\frac{y^{2}}{2t}-k\I y}\,dy=\sqrt{\frac{t}{2\pi}}\,\E^{-\frac{k^{2}t}{2}}.
$$
Thus, the density of the Brownian motion on the circle with respect to the Haar measure $\frac{d\theta}{2\pi}$ is
$$
p_{t}^{\Tor}(\E^{\I \theta})=\sum_{k=-\infty}^{\infty}\E^{-\frac{k^{2}t}{2}}\,\E^{k\I \theta}=1+2\sum_{k=1}^{\infty}\E^{-\frac{k^{2}t}{2}}\,\cos k\theta,
$$
Since $s_{(k)}(\E^{\I \theta})=\E^{k\I \theta}$, this is indeed a specialization of the second formula of Theorem \ref{explicitdensity}, for $\unit(1,\C)$.
\end{example}
\medskip

\begin{example}[Brownian motion on the $3$-dimensional sphere] Consider the Brownian motion on $\unit\SP(1,\Hq)$, which is also $\SU(2,\C)$ by one of the exceptional isomorphisms. The specialization of the first formula of Theorem \ref{explicitdensity} for $\SU(2,\C)$ gives
$$
p_{t}^{\SU(2,\C)}(g)=\sum_{k=0}^{\infty} \E^{-\frac{k(k+2)\,t}{8}}\,(k+1)\,\frac{\sin (k+1)\theta}{\sin \theta},
$$
if $\E^{\pm\I\theta}$ are the eigenvalues of $g \in \SU(2,\C)$. It agrees with the example of \cite[\S4]{Liao04paper}, and also with Formula \eqref{realsphere} when $n=3$, since the group of unit quaternions is topologically a $3$-sphere.
\end{example}
\medskip

\begin{remark}
The previous examples show that the restrictions $n \geq n_{0}$ are not entirely necessary for the formulas of Theorem \ref{explicitdensity} to hold. One should only beware that the root systems of type $\mathrm{B}_{1}$, $\mathrm{C}_{1}$, $\mathrm{D}_{1}$ and $\mathrm{D}_{2}$ are somewhat degenerated, and that the dominant weights do not have the same indexing set as for $\mathrm{B}_{n\geq 2}$ or $\mathrm{C}_{n \geq 2}$ or $\mathrm{D}_{n\geq 3}$. For instance, for the special orthogonal group $\SO(3,\R)$, the only positive root is $e^{1}$, and the only fundamental weight is also $e^{1}$. Consequently, irreducible representations have highest weights $k\,e^{1}$ with $k\in \N$; the dimension of the representation of label $k$ is $2k+1$, and the corresponding character is again $\frac{\sin (k+1)\theta}{\sin \theta}$ if $\E^{\I\theta}$ and $\E^{-\I \theta}$ are the non-trivial eigenvalues of the considered rotation. So
$$
p_{t}^{\SO(3,\R)}(g)=\sum_{k=0}^{\infty} \E^{-\frac{k(k+1)\,t}{3}}\,(2k+1) \,\frac{\sin (k+1)\theta}{\sin \theta}
$$
if $g$ is a rotation of angle $\theta$ around some axis. 
\end{remark}
\bigskip\bigskip

\section{Upper bounds after the cut-off time}\label{upper}
Let $\mu$ be a probability measure on a compact Lie group $K$ or compact symmetric space $G/K$, that is absolutely continuous with respect to the Haar measure $\eta$, and with density $p$. Cauchy-Schwarz inequality ensures that
$$4\,\dtv(\mu,\eta)^{2}=\left(\int_{X} |p(x)-1|\,dx\right)^{2} \leq \int_{X} |p(x)-1|^{2}\,dx=\|p-1\|_{\leb^{2}(X)}^{2}.$$
The discussion of Section \ref{fourier} allows now to relate the right-hand side of this inequality with the harmonic analysis on $X$. Let us first treat the case of a compact Lie group $K$. If one assumes that $p$ is invariant by conjugacy, then Parseval's identity \eqref{parseval} shows that the right-hand side is $\sum_{\lambda \in \widehat{K}} |\chi^{\lambda}(p-1)|^{2}$. However, by orthogonality of characters, for any non-trivial irreducible representation of $K$ --- \emph{i.e.}, not equal to $\mathbf{1}_{K} : k \in K \mapsto 1$ --- one has 
$$\chi^{\lambda}(1)=\int_{K} \chi^{\lambda}(k)\,dk=\int_{K} \chi^{\lambda}(k)\,\chi^{\mathbf{1}_{K}}(k^{-1})\,dk=0.$$
On the other hand, for any measure $\mu$ on the group, $\chi^{\mathbf{1}_{K}}(\mu)=\int_{K}\chi^{\mathbf{1}}(k)\,\mu(dk)=\int_{K} \mu(dk)=1$. Hence, the inequality now takes the form 
$$4\,\dtv(\mu,\eta_{K})^{2}\leq \sum_{\lambda \in \widehat{K}}^{\prime} |\chi^{\lambda}(p)|^{2} ,$$
where the $\prime$ indicates that we remove the trivial representation from the summation. Similarly, on a compact symmetric space $G/K$, supposing that $p$ is $K$-invariant, Parseval's identity reads $\|p-1\|_{\leb^{2}(G/K)}^{2}=\sum_{\lambda \in \hatG^{K}} D^{\lambda}\,|\phi^{\lambda}(p-1)|^{2}$. However, for any non-trivial representation $\lambda$, 
$$\phi^{\lambda}(1)=\scal{e^{\lambda}}{\int_{G}\rho^{\lambda}(g)(e^{\lambda})\,dg}=0.$$ 
Indeed, using only elementary properties of the Haar measure, one sees that $\widehat{1}(\lambda)=\int_{G}\rho^{\lambda}(g)\,dg=0$, because it is a projector and it has trace $\chi^{\lambda}(1)=0$. So again, the previous inequality can be simplified and it becomes
$$4\,\dtv(\mu,\eta_{G/K})^{2}\leq \sum_{\lambda \in \widehat{G}^{K}}^{\prime} D^{\lambda}\,|\phi^{\lambda}(p)|^{2}.$$
In the setting and with the notations of Proposition \ref{abstractdensity}, a bound at time $t$ on $4 \,\dtv(\mu_{t},\eta_{K})^{2}$ (respectively, on $4\,\dtv(\mu_{t},\eta_{G/K})^{2}$) is then
$$ \sum_{\lambda \in \widehat{K}}^{\prime} \E^{-t\scal{\lambda+2\rho}{\lambda}}\,(D^{\lambda})^{2}\quad;\quad\text{respectively},\quad \sum_{\lambda \in \hatG^{K}}^{\prime} \E^{-t\scal{\lambda+2\rho}{\lambda}}\,D^{\lambda}.$$
\begin{proposition}\label{scaryseries}
In every classical case, $4\,\dtv(\mu_{t},\mathrm{Haar})^{2}$ is bounded by $\sum^{\prime}_{\lambda \in W_{n}} A_{n}(\lambda)\,\E^{-t\, B_{n}(\lambda)}$, where the indexing sets $W_{n}$ and the constants $B_{n}(\lambda)$ are the same as in Theorem \ref{explicitdensity}, and $A_{n}(\lambda)=(D^{\lambda})^{2}$ for compact Lie groups and $D^{\lambda}$ for compact symmetric spaces.
\end{proposition}\bigskip

This section is now organized as follows. In \S\ref{order}, we compute the weights that minimize $B_{n}(\lambda)$; they will give the correct order of decay of the whole series after cut-off time. In \S\ref{versus}, we then show case-by-case that all the other terms of the series $S_{n}(t)$ of Proposition \ref{scaryseries} can be controlled uniformly. Essentially, we adapt the arguments of \cite{Ros94,Por96a,Por96b}, though we also introduce new computational tricks. As explained in the introduction, the main reason why one has a good control over $S_{n}(t)$ after cut-off time is that each term $T_{n}(\lambda,t)=A_{n}(\lambda)\,\E^{-t\,B_{n}(\lambda)}$ of the series $S_{n}(t)$ stays bounded when $t=t_{\text{cut-off}}$, for every $n$, every class $\lambda$ and in every case. We have unfortunately not found a way to factorize all the computations needed to prove this, so each case will have to be treated separately. However, the scheme of the proof will always be the same, and the reader will find the main arguments in \S\ref{sympupper} (for symplectic groups and their quotients), so he can safely skip \S\ref{oddupper}-\ref{unitaryupper} if he does not want to see the minor modifications needed to treat the other cases.\bigskip

\subsection{Guessing the order of decay of the dominating series}\label{order}
Remember the restriction $n \geq 2$ (respectively, $n \geq 3$ and $n \geq 10$) when studying special unitary groups (resp., compact symplectic groups and special orthogonal groups) and their quotients. We use the superscript $\star$ to denote a set of partitions or half-partitions minus the trivial partition $(0,0,\ldots,0)$. The lemma hereafter allows to guess the correct order of decay of the series under study.
\begin{lemma}\label{decay}
Each weight $\lambda_{\min}$ indicated in the table hereafter corresponds to an irreducible representation in the case of compact groups, and to a spherical irreducible representation in the case of symmetric spaces of type non-group. The table also gives the corresponding values of $A_{n}$ and $B_{n}$. In the group case, $B_{n}(\lambda_{\min})$ is minimal among $\{B_{n}(\lambda),\,\,\lambda \in W_{n}^{\star}\}$.
\end{lemma}
\begin{remark}
For symmetric spaces of type non-group, one can also check the minimality of $B_{n}(\lambda_{\min})$, except for certain real Grassmannian varieties $\Gra(n,q,\R)$. For instance, if $q=1$, then $(1)_{q}$ labels the geometric representation of $\SO(n,\R)$ on $\C^{n}$, which has indeed an invariant vector by $\SO(n-1,\R)\times\SO(1,\R)$; and the corresponding value of $B(\lambda)$ is $(n-1)/n < 2$. Fortunately, $\lambda_{\min}$, though not minimal, will still yield in this case the correct order of decay of the series $S(t)$.
\end{remark}

\begin{remark}
To each ``minimal''  weight $\lambda_{\min}$ corresponds a very natural representation. Namely, for a special orthogonal group $\SO(n,\R)$ (respectively, a compact symplectic group $\unit\SP(n,\Hq)$), the minimizer is the ``geometric'' representation over $\C^{n}$ (respectively $\C^{2n}$) corresponding to the embedding $\SO(n,\R) \hookrightarrow \SO(n,\C) \hookrightarrow \GL(n,\C)$ (respectively $\unit\SP(n,\Hq) \hookrightarrow \SU(2n,\C)\hookrightarrow \GL(2n,\C)$). For a special unitary group $\SU(n,\C)$, one has again the geometric representation over $\C^{n}$, and its compose with the involution $k \mapsto (k^{t})^{-1}$ corresponds to the label $(1,\ldots,1)_{n-1}$, which also minimizes $B_{n}(\lambda)$. The case of spherical minimizers is more involved but still workable: we shall detail it in Section \ref{lower}.
\end{remark}

\begin{proof} To avoid any ambiguity, we shall use indices to precise the length of a partition or half-partition. Let us first find the minimizers of $B_{n}(\lambda)$ in the group case:
\vspace{2mm}

$$
\label{specialpage}
\begin{tabular}{|c|c|c|c|}
\hline\vspace{-2.5mm}&&&\\
 $K$ or $G/K$&  $\lambda_{\min}$ & $B_{n}(\lambda_{\min})$ &$A_{n}(\lambda_{\min})$  \\
 \vspace{-2.5mm}&&&\\
\hline\hline \vspace{-2.5mm}&&&\\
$\SO(2n+1,\R)$ &$(1,0,\ldots,0)_{n}$  & $\frac{2n}{2n+1}$ & $(2n+1)^{2}$\\
\vspace{-2.5mm}&&& \\
\hline \vspace{-2.5mm}&&&\\ 
$\SO(2n,\R)$ &$(1,0,\ldots,0)_{n}$ & $\frac{2n-1}{2n}$ & $4n^{2}$\\
\vspace{-2.5mm}&&&\\
\hline \vspace{-2.5mm}&&&\\
$\SU(n,\C)$  & $(1,0,\ldots,0)_{n-1}$ & $1-\frac{1}{n^{2}}$ & $n^{2}$\\
\vspace{-2.5mm}&&&\\
\hline\vspace{-2.5mm} &&& \\
$\unit\SP(n,\Hq)$ & $(1,0,\ldots,0)_{n}$ &$\frac{2n+1}{2n}$& $4n^{2}$\\
\vspace{-2.5mm}&&&\\
\hline\hline\vspace{-2.5mm}&&&\\
 $\Gra(2n+1,q,\R)$ & $(2,0,\ldots,0)_{q}$ & $2$& $2n^{2}+3n$\\
 \vspace{-2.5mm}&&&\\
\hline\vspace{-2.5mm}&&&\\
 $\Gra(2n,q,\R)$ & $(2,0,\ldots,0)_{q}$ & $2$ &$2n^{2}+n-1 $\\
 \vspace{-2.5mm}&&&\\
 \hline\vspace{-2.5mm}&&&\\
 $\Gra(n,q,\C)$ & $(1,0,\ldots,0)_{q}$ & $2$ & $n^{2}-1$\\
 \vspace{-2.5mm}&&&\\
 \hline\vspace{-2.5mm}&&&\\
 $\Gra(n,q,\Hq)$ & $(1,1,0,\ldots,0)_{2q}$ & $2$ & $(n-1)(2n+1)$ \\
 \vspace{-2.5mm}&&&\\
\hline\hline\vspace{-2.5mm}&&&\\
 $\SO(2n,\R)/\unit(n,\C)$ & $(1,1,0,\ldots,0)_{n}$ & $\frac{2(n-1)}{n}$ & $n(2n-1)$\\
 \vspace{-2.5mm}&&&\\
\hline\vspace{-2.5mm}&&&\\
 $\SU(n,\C)/\SO(n,\R)$  & $(2,0,\ldots,0)_{n-1}$ & $\frac{2(n-1)(n+2)}{n^{2}}$ &$\frac{n(n+1)}{2}$\\
 \vspace{-2.5mm}&&&\\
\hline\vspace{-2.5mm}&&&\\
 $\SU(2n,\C)/\unit\SP(n,\Hq)$  & $(1,1,0,\ldots,0)_{2n-1}$ &$\frac{(n-1)(2n+1)}{n^{2}}$ & $n(2n-1)$\\
 \vspace{-2.5mm}&&&\\
\hline\vspace{-2.5mm}&&&\\
 $\unit\SP(n,\Hq)/\unit(n,\C)$  & $(2,0,\ldots,0)_{n}$ &$\frac{2(n+1)}{n}$ & $n(2n+1)$\\
&&&\\
\hline
\end{tabular}
 $$
\begin{itemize}\vspace{2mm}
\item $\SU(n)$: one has to minimize
$$-\frac{|\lambda|^{2}}{n}+\sum_{i=1}^{n-1} \lambda_{i}^{2}+(n+1-2i) \lambda_{i}=\frac{1}{n}\left(\sum_{1 \leq i<j\leq n}(\lambda_{i}-\lambda_{j})^{2}\right)+\left(\sum_{i=1}^{n-1} i(n-i)(\lambda_{i}-\lambda_{i+1})\right)=A+B
$$
over $\ym_{n-1}^{\star}$. In $B$, at least one term is non-zero, so $$B \geq \left(\min_{i \in \lle 1,n-1\rre}i(n-i)\right)=n-1,$$ with equality if and only if $\lambda=(1,0,\ldots,0)_{n-1}$ or $\lambda=(1,\ldots,1)_{n-1}$. In both cases, $A$ is then equal to $\frac{n-1}{n}$. However, $\frac{n-1}{n}$ is also the minimum value of $A$ over $\ym_{n-1}^{\star}$. Indeed, there is at least one index $l \in \lle 1,n-1\rre$ such that $\lambda_{l}>\lambda_{l+1}$. Then all the $(\lambda_{i}-\lambda_{j})^{2}$ with $i \leq l$ and $j \geq l+1$ give a contribution at least equal to $1$, and there are $l(n-l)$ such contributions. Thus $$A \geq \frac{l(n-l)}{n}\geq \frac{n-1}{n},$$ and one concludes that $\min B_{n}(\lambda)$ is obtained only for the two aforementioned partitions, and is equal to $\frac{1}{n}(A_{\min}+B_{\min})=1-\frac{1}{n^{2}}$.  \vspace{2mm}
\item $\SO(2n)$: the quantity to minimize over $\frac{1}{2}\ym_{n}^{\star}$ is
$$
\left(\sum_{i=1}^{n} \lambda_{i}^{2} \right)+ \left(\sum_{i=1}^{n-2} i(2n-1-i)(\lambda_{i}-\lambda_{i+1})\right)+n(n-1)\lambda_{n-1}=A+B+C,
$$
again with $A$, $B$ and $C$ non-negative in each case. Only $A$ involves $\lambda_{n}$, so a minimizer satisfies necessarily $\lambda_{n}=0$ (partitions) or $\lambda_{n}=\frac{1}{2}$ (half-partitions). In the case of partitions, a minimizer of $B+C$ is $(1,0,\ldots,0)_{n}$, which gives the value $\min_{i\in\lle 1,n-1\rre}i(2n-1-i)=2n-2$. The same sequence minimizes $A$ over $\ym_{n}^{\star}$, so the minimal value of $A+B+C$ over non-trivial partitions is $2n-1$ and it is obtained only for $(1,0,\ldots,0)_{n}$. On the other hand, over half-partitions, the minimizer is $\left(\frac{1}{2},\ldots,\frac{1}{2}\right)_{n}$, giving the value $$\frac{n}{4}+\frac{n(n-1)}{2}=\frac{n(2n-1)}{4}.$$ Since we assume $2n\geq 10$ and therefore $n \geq 5$, this value is strictly bigger than $2n-1$, so the only minimizer of $B_{n}(\lambda)$ in $\frac{1}{2}\ym_{n}^{\star}$ is $(1,0,\ldots,0)_{n}$.
\vspace{2mm}
\item $\SO(2n+1)$: exactly the same reasoning gives the unique minimizer $(1,0,\ldots,0)_{n}$, with corresponding value $2n$ for $A+B+C=(2n+1)\,B_{n}(\lambda)$. \vspace{2mm}
\item $\unit\SP(n)$: here one has only to look at partitions, and the same reasoning as for $\SO(2n)$ and $\SO(2n+1)$ yields the unique minimizer $(1,0,\ldots,0)_{n}$, corresponding to the value $2n+1$ for $2n\,B_{n}(\lambda)$.\vspace{2mm}
\end{itemize}
The spherical minimizers are obtained by the same techniques; however, some cases (with $n$ or $q$ too small) are exceptional, so we have only retained in the statement of our Lemma the ``generic'' minimizer. The corresponding values of $A_{n}(\lambda)$ and $B_{n}(\lambda)$ are easy calculations.
\end{proof}
\bigskip

Suppose for a moment that the series $S_{n}(t)$ of Proposition \ref{scaryseries} has the same behavior as its ``largest term'' $A_{n}(\lambda_{\min})\,\E^{-t\,B_{n}(\lambda_{\min})}$. We shall show in a moment that this is indeed true just after cut-off time (for $n$ big enough). Then, $S_{n}(t)$ is a $O(\cdot)$ of\vspace{1mm}
\begin{itemize}
\item $n^{2}\,\E^{-t}$ for classical simple compact Lie groups;\vspace{2mm}
\item $n^{2}\,\E^{-2t}$ for classical simple compact symmetric spaces of type non-group.\vspace{2mm}
\end{itemize}
Set then $t_{n,\eps}=\alpha\,(1+\eps)\,\log n$, with $\alpha=2$ in the case $n^{2}\,\E^{-t}$, and $\alpha=1$ in the case $n^{2}\,\E^{-2t}$. Under the assumption $S_{n}(t)\sim A_{n}(\lambda_{\min})\,\E^{-t\,B_{n}(\lambda_{\min})}$, one has $S_{n}(t_{n,\eps}) = O(n^{-2\eps}).$
Thus, the previous computations lead to the following guess: the cut-off time is\vspace{1mm}
\begin{itemize}
\item $2 \log n$ for classical simple compact Lie groups;\vspace{2mm}
\item $\log n$ for classical simple compact symmetric spaces of type non-group.\vspace{2mm}
\end{itemize}

\subsection{Growth of the dimensions versus decay of the Laplace-Beltrami eigenvalues}\label{versus}
The estimate $S_{n}(t_{n,\eps})\sim A_{n}(\lambda_{\min})\,\E^{-t_{n,\eps}\,B_{n}(\lambda)}=O(n^{-2\eps})$ might seem very optimistic; nonetheless, we are going to prove that the sum of all the other terms $A_{n}(\lambda)\,\E^{-t_{n,\eps}\,B_{n}(\lambda)}$ in $S_{n}(t)$ does not change too much this bound, and that one still has at least $ S(t_{n,\eps})=O\left(n^{-\frac{\eps}{2}}\right).$ We actually believe that at least in the group case, the exponent $2\eps$ is good, \emph{cf.} the remark before \S\ref{order} --- the previous discussion shows that it is then optimal.\bigskip

Suppose that one can bound $A_{n}(\lambda)\,\E^{-t_{n,\eps}\,B_{n}(\lambda)}$ by $c(n)^{|\lambda|}$, where $|\lambda|$ is the size of the partition and $c(n)$ is some function of $n$ that goes to $0$ as $n$ goes to infinity (say, $Cn^{-\delta \eps}$). We can then use:
\begin{lemma}\label{partition}
Assume $x \leq \frac{1}{2}$. Then, the sum over all partitions $\sum_{\lambda} x^{|\lambda|}$, which is convergent, is smaller than $1+5x$. Consequently,
$$\sum_{\substack{\lambda \in \ym_{n} \\
\lambda \neq (0,\ldots,0)}} x^{|\lambda|}\leq 5x.$$
\end{lemma}
\begin{proof}
The power series $P(x)=\sum_{\lambda} x^{|\lambda|} = \prod_{i=1}^{\infty}\frac{1}{1-x^{i}}=1+x+2x^{2}+3x^{3}+5x^{4}+\cdots$ has radius of convergence $1$, and it is obviously convex on $\R_{+}$. Thus, it suffices to verify the bound at $x=0$ and $x=\frac{1}{2}$. However,
$$P(0)=1=1+(5\times 0)\quad;\quad P\left(\frac{1}{2}\right)\leq 3.463\leq 1+\left(5\times \frac{1}{2}\right).\vspace{-7mm}$$
\end{proof}
\vspace{2mm}
\noindent With this in mind, the idea is then to control the growth of the coefficients $A_{n}(\lambda)$, starting from the trivial partition $(0,\ldots,0)$. This is also what is done in \cite{Por96a,Por96b}, but the way we make our partitions grow is different. The simplest cases to treat in this perspective are the compact symplectic groups and their quotients.
\bigskip

\subsubsection{Symplectic groups and their quotients}\label{sympupper}
Set $t_{n,\eps}=2(1+\eps)\log n$; in particular, $t_{n,0}=2\log n$. We fix a partition $\lambda \in \ym_{n}$, and for $k \leq \lambda_{n}$, we denote $\rho_{k,n}$ the quotient of the dimensions $D^{\lambda}$ associated to the two rectangular partitions
\begin{equation}
(k,\ldots,k)_{n} \quad\text{and} \quad(k-1,\ldots,k-1)_{n}.\label{evolutionfirst}
\end{equation}
Using the formula given in \S\ref{explicit} in the case of compact symplectic groups, one obtains:
\begin{align*}
\rho_{k,n}=\prod_{1 \leq i\leq j \leq n} \frac{2k+2n+2-i-j}{2k+2n-i-j}=\prod_{1 \leq i\leq j \leq n} 1+\frac{2}{2k+2n-i-j}\leq \exp\left(\sum_{1\leq i\leq j\leq n} \frac{2}{2k+2n-i-j}\right).
\end{align*}
The double sum can be estimated by standard comparison techniques between sums and integrals. Namely, since $x,y \mapsto \frac{1}{2k+2n-x-y}$ is convex on $\{(x,y)\,\,|\,\, x \geq 0,\,\, y \geq 0,\,\,2k+2n \geq x+y\}$, one can bound each term by
$$\frac{2}{2k+2n-i-j} \leq \iint_{\left[i-\frac{1}{2},i+\frac{1}{2}\right]\times\left[j-\frac{1}{2},j+\frac{1}{2}\right] }\, \frac{2}{2k+2n-x-y}\,dx\,dy.$$
We use this bound for non-diagonal terms with indices $i<j$, and for diagonal terms with $i=j$, we use the simpler bound
$$\sum_{i=1}^{n}\frac{1}{k+n-i} = \sum_{u=0}^{n-1} \frac{1}{k+u}=H_{k+n-1}-H_{k-1}\leq \frac{1}{k}+\log(k+n-1)-\log k $$
where $H_{n}$ denotes the $n$-th harmonic sum. So,
\begin{align*}\log \rho_{k,n} &\leq \sum_{1 \leq i\leq j\leq n}\frac{2}{2k+2n-i-j} \leq H_{k+n-1}-H_{k-1}+\iint_{\left[\frac{1}{2},n+\frac{1}{2}\right]^{2}} \frac{1}{2k+2n-x-y}\,dx\,dy \\ 
&\leq \frac{1}{k}+\log(k+n-1)-\log k\\
&\quad+(2k+2n-1)\log(2k+2n-1)+(2k-1)\log(2k-1)-2(2k+n-1)\log(2k+n-1).\end{align*}
On the other hand, the same transformation on partitions makes $-t_{n,0}\,B_{n}(\lambda)$ evolve by $-(2k+n)\log n$. So, if $\eta_{k,n}^{2}$ is the quotient of the quantities $(D^{\lambda})^{2}\,\E^{-t_{n,0}\,B_{n}(\lambda)}$ with $\lambda$ as in Equation \eqref{evolutionfirst}, then 
\begin{align*}
\log \eta_{k,n} &\leq -\frac{2k+n}{2}\log n + \frac{1}{k}+\log(k+n-1)-\log k\\
&\quad+(2k+2n-1)\log(2k+2n-1)+(2k-1)\log(2k-1)-2(2k+n-1)\log(2k+n-1).
\end{align*}
Suppose $k \geq 2$. Then, one can fix $n\geq 3$ and study the previous expression as a function of $k$. Its derivative is then always negative, so $\log \eta_{k,n}\leq \log \eta_{2,n}$, which is also always negative. From this, one deduces that 
$$D^{\lambda}\,\E^{-\frac{t_{n,0}}{2}\,B_{n}(\lambda)} \leq \eta_{1,n}$$ for any rectangular partition $(\lambda_{n},\ldots,\lambda_{n})_{n}$; indeed, the left-hand side is the product of the contributions $\eta_{k,n}$ for $k$ in $\lle 1,\lambda_{n}\rre$. However, $\eta_{1,n}$ is also smaller than $1$: in this case, the dimension is given by the exact formula
$$D^{(1,\ldots,1)_{n}}=\mathrm{Cat}_{n+1}=\frac{1}{n+2}\binom{2n+2}{n+1},$$
so $\eta_{1,n}=\mathrm{Cat}_{n+1}\,\E^{-\frac{n+2}{2}\log n}$, which can be checked to be smaller than $1$ for every $n \geq 3$. So in fact,
$$D^{\lambda}\,\E^{-\frac{t_{n,0}}{2}\,B_{n}(\lambda)} \leq 1$$ for any rectangular partition $(\lambda_{n},\ldots,\lambda_{n})_{n}$. 
\bigskip

The previous discussion hints at the more general result:
\begin{proposition}\label{superboundsymplectic}
In the case of compact symplectic groups, at cut-off time, 
$$D^{\lambda}\,\E^{-\frac{t_{n,0}}{2}\,B_{n}(\lambda)} \leq \frac{14}{3}$$
for any integer partition $\lambda$ of length $n$ (not only the rectangular partitions).
\end{proposition}

\figcap{\includegraphics{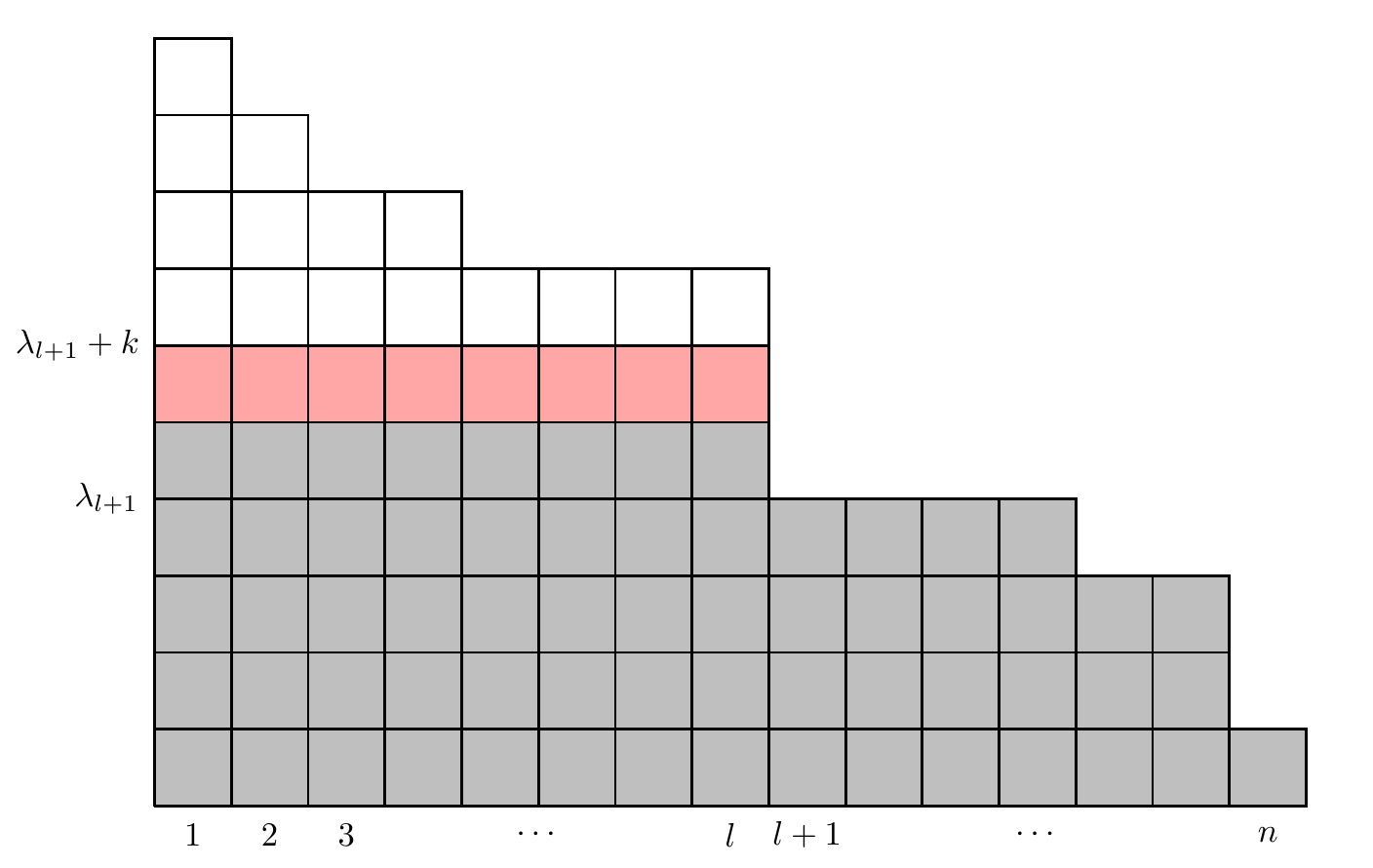}}{One makes the partitions grow layer by layer, starting from the bottom.\label{grow}}

\begin{proof}
We fix $l \in \lle 1,n-1\rre$, and the idea is again to study the quotient $\rho_{k,l}$ of the dimensions associated to the two partitions
\begin{equation}
(k+\lambda_{l+1},\ldots,k+\lambda_{l+1},\lambda_{l+1},\ldots,\lambda_{n})_{n} \quad\text{and} \quad(k-1+\lambda_{l+1},\ldots,k-1+\lambda_{l+1},\lambda_{l+1},\ldots,\lambda_{n})_{n},\label{evolution}
\end{equation}
where $k$ is some integer smaller than $\lambda_{l}-\lambda_{l+1}$ --- in other words, the $n-l$ last parts of our partition have already been constructed, and one adds $k$ to the $l$ first parts, until $k=\lambda_{l}-\lambda_{l+1}$; see Figure \ref{grow}.\bigskip

The transformation on partitions described by Equation \eqref{evolution} makes the quantity $-t_{n,0}\,B_{n}(\lambda)$ change by $-\frac{l(2k'+2n-l)}{n}\log n$. We shall prove that this variation plus $\log \rho_{k,l}$ is almost always negative. For convenience, we will treat separately the cases $l=1$ or $2$ and the case $l\geq 3$; hence, suppose first that $l \in \lle 3,n-1\rre$. The quotients of Vandermonde determinants can be simplified as follows:
$$\rho_{k,l}=\!\prod_{j=l+1}^{n} \frac{k+j-1+\lambda_{l+1}-\lambda_{j}}{k+j-l-1+\lambda_{l+1}-\lambda_{j}}\,\frac{k+\lambda_{l+1}+\lambda_{j}+2n+1-j}{k+\lambda_{l+1}+\lambda_{j}+2n+1-j-l} \prod_{1\leq i \leq j \leq l}\!\!\frac{2k+2\lambda_{l+1}+2n+2-i-j}{2k+2\lambda_{l+1}+2n-i-j}.$$
Notice that the second product $\rho_{k,l,(2)}$ in this formula is very similar to $\rho_{k,n}$; the main difference is that indices $i,j$ are now smaller than $l$ (instead of $n$). Hence, by adapting the arguments, one obtains
\begin{align*}
\log \rho_{k,l,(2)}&\leq \sum_{1\leq i\leq j \leq l} \frac{2}{2k'+2n-i-j} \leq H_{k'+n-1}-H_{k'+n-l-1} +\iint_{\left[\frac{1}{2},l+\frac{1}{2}\right]^{2}}\frac{1}{2k'+2n-x-y}\,dx\,dy\\
&\leq \frac{1}{k'+n-l}+\log(k'+n-1)-\log(k'+n-l)+(2k'+2n-1)\log(2k'+2n-1)\\
&\quad+(2k'+2n-2l-1)\log(2k'+2n-2l-1)-2(2k'+2n-l-1)\log(2k'+2n-l-1)
\end{align*}
where $k'$ stands for $k+\lambda_{l+1}$. So, if $(\eta_{k,l})^{2}$ is the quotient of the quantities $(D^{\lambda})^{2}\,\E^{-t_{n,0}\,B_{n}(\lambda)}$ with $\lambda$ as in Equation \eqref{evolution}, then $\log \eta_{k,l} \leq \log \widetilde{\eta}_{k,l}+\log \rho_{k,l,(1)}$, where $\log \widetilde{\eta}_{k,l}$ is given by
\begin{align*}
&-\frac{l(2k'+2n-l)}{2n}\log n+\frac{1}{k'+n-l}+\log(k'+n-1)-\log(k'+n-l)\\
&+(2k'+2n-1)\log(2k'+2n-1)+(2k'+2n-2l-1)\log(2k'+2n-2l-1)\\
&-2(2k'+2n-l-1)\log(2k'+2n-l-1),
\end{align*}
and $\rho_{k,l,(1)}$ is the first product in the expansion of $\rho_{k,l}$. Let us analyze these two quantities separately.\vspace{2mm}
\begin{itemize}
\item $\log \widetilde{\eta}_{k,l}$: here the technique is really the same as for $\log \eta_{k,n}$. Namely, with $n$ and $l$ fixed, $\log \widetilde{\eta}_{k,l}$ appears as a decreasing function of $x=k'$, because its derivative with respect to $x$ is
\begin{align*}&-\frac{l\,\log n}{n}-\frac{1}{(x+n-l)^{2}}+\frac{1}{x+n-1}-\frac{1}{x+n-l}\\
&+2\big(\log(2x+2n-1)+\log(2x+2n-2l-1)-2\log(2x+2n-l-1)\big).
\end{align*}
A upper bound on the first line is $-\frac{(l-1)\log n}{n}\leq 0$ (remember that $n \geq 3$ and therefore $\log n \geq 1$), and the second line is negative by concavity of the logarithm.  From this, one deduces that $\log\widetilde{\eta}_{k,l}\leq\log\widetilde{\eta}_{1,l}$, and we shall use this estimate in order to compensate the other part of $\log \eta_{k,l}$:
\begin{align*}
\log \widetilde{\eta}_{k,l}&\leq -\frac{l(2v+2+2n-l)}{2n}\log n+\frac{1}{v+n+1-l}+\log(v+n)-\log(v+n+1-l)\\
&\quad+(2v+2n+1)\log(2v+2n+1)+(2v+2n-2l+1)\log(2v+2n-2l+1)\\
&\quad-2(2v+2n-l+1)\log(2v+2n-l+1)
\end{align*}
where $v$ stands for $\lambda_{l+1}$.
\vspace{2mm}
\item $\log \rho_{k,l,(1)}$: in the product $\rho_{k,l,(1)}$, each term of index $j$ writes as
\begin{align*}
\frac{(k'+n)^{2}-(\lambda_{j}+n+1-j)^{2}}{(k'+n-l)^{2}-(\lambda_{j}+n+1-j)^{2}}&\leq \frac{(k'+n)^{2}-(\lambda_{l+1}+n+1-j)^{2}}{(k'+n-l)^{2}-(\lambda_{l+1}+n+1-j)^{2}} \\
&\leq \frac{k+j-1}{k+j-l-1}\,\frac{k''+2n+1-j}{k''+2n+1-j-l}
\end{align*}
with $k''=k+2\lambda_{l+1}=k+2v$; and multiplying all these bounds together, one gets
$$
\rho_{k,l,(1)}\leq \frac{(k+n-1)!}{(k+l-1)!}\,\frac{(k-1)!}{(k+n-l-1)!}\,\frac{(k''+2n-l)!}{(k''+n)!}\,\frac{(k''+n-l)!}{(k''+2n-2l)!}.
$$
Again, this is decreasing in $k$, so
$$
\rho_{k,l,(1)}\leq \frac{n!\,(2v+2n-l+1)!\,(2v+n-l+1)!}{l!\,(n-l)!\,(2v+n+1)!\,(2v+2n-2l+1)!}.
$$
Recall the classical Stirling estimates: for $m \geq 1$,
$$\log m! =m\log m +\frac{1}{2}\log m-m+\log \sqrt{2\pi}+\frac{1}{12m}-r_{m},\quad\text{with }\,0 \leq r_{m} \leq \frac{1}{360m^{3}}.$$
It enables us to bound $\log \rho_{k,l,(1)}$ by the sum of the following quantities:\vspace{2mm}
\begin{itemize}
\item[{$\star$}] $A=(2v+2n-l+1)\log(2v+2n-l+1)+(2v+n-l+1)\log(2v+n-l+1)$\\
\phantom{blop}~$-(2v+n+1)\log(2v+n+1)-(2v+2n-2l+1)\log(2v+2n-2l+1)$.\vspace{2mm}
\item[{$\star$}] $B=\frac{1}{2}(\log(2v+2n-l+1)+\log(2v+n-l+1)-\log(2v+n+1)-\log(2v+2n-2l+1))$,
which is non-positive by concavity of the logarithm.\vspace{2mm}
\item[{$\star$}] $C=n\log n - l \log l - (n-l) \log (n-l)$.\vspace{2mm}
\item[{$\star$}] $D=\frac{1}{2}(\log n - \log l - \log(n-l))$. This is non-positive unless $n=l+1$ --- recall that we assume for the moment $l \in \lle 3,n-1\rre$. In that case, it is smaller than $\frac{1}{2(n-1)}$.\vspace{2mm} 
\item[{$\star$}] $E=\frac{1}{12}\left(\frac{1}{n}-\frac{1}{l}-\frac{1}{n-l}+\frac{1}{2v+2n-l+1}+\frac{1}{2v+n-l+1}-\frac{1}{2v+n+1}-\frac{1}{2v+2n-2l+1}\right)$.  \vspace{2mm}
\item[{$\star$}] $F=\frac{1}{360}\left(\frac{1}{l^{3}}+\frac{1}{(n-l)^{3}}+\frac{1}{(2v+n+1)^{3}}+\frac{1}{(2v+2n-2l+1)^{3}}\right)$.\vspace{2mm}
\end{itemize}
The sum of the two last terms $E\!F=E+F$ happens to be negative. Indeed, $E$ and $F$ are decreasing in $v$ (we use the convexity of $x \mapsto \frac{1}{x^{2}}$ to show that $\frac{dE}{dv}\leq 0$), so it suffices to check the result when $v=0$. Then, with $l$ fixed, 
\begin{align*}E\!F(n,l)&=\frac{1}{12}\left(\frac{1}{n}-\frac{1}{l}-\frac{1}{n-l}+\frac{1}{2n-l+1}+\frac{1}{n-l+1}-\frac{1}{n+1}-\frac{1}{2n-2l+1}\right)\\
&\quad+\frac{1}{360}\left(\frac{1}{l^{3}}+\frac{1}{(n-l)^{3}}+\frac{1}{(n+1)^{3}}+\frac{1}{(2n-2l+1)^{3}}\right)
\end{align*}
is decreasing in $n$, hence smaller than its value when $n=l+1$. So, it suffices to look at $E\!F(l+1,l)$, which is now increasing in $l$, but still negative. Thus, in the following, we shall use the bound
$$\log \rho_{k,l,(1)}\leq A+C+D\leq A+C+\frac{1}{2n-2}.$$\vspace{2mm}
\end{itemize}
\noindent Adding together the bounds previously demonstrated, we get
\begin{align*}
\log \eta_{k,l}&\leq-\frac{l(2v+2+2n-l)}{2n}\log n+\frac{1}{2n-2}+\frac{1}{v+n+1-l}+\log(v+n)-\log(v+n+1-l)\\
&\quad+(2v+2n+1)\log(2v+2n+1)-(2v+2n-l+1)\log(2v+2n-l+1)\\
&\quad+(2v+n-l+1)\log(2v+n-l+1)-(2v+n+1)\log(2v+n+1)\\
&\quad+n \log n - l \log l -(n-l)\log(n-l).
\end{align*}
By concavity of $x \log x$, the sum of the second and third rows is non-positive. What remains is decreasing in $l$ and in $v$, and when $l=3$ and $v=0$, we get
$$\frac{3}{2n}\log n + \frac{1}{2n-2}+\frac{1}{n-2}+\log \left(\frac{n}{n-2}\right)+(n-3)\log \left(\frac{n}{n-3}\right)-3\log 3$$
which is maximal for $n=5$, and still (barely) negative at this value. Thus, we have shown so far that $\eta_{k,l} \leq 1$ for any $k$, any $l \in \lle 3,n-1\rre$, and any partition $\lambda$ that we fill as in Figure \ref{grow}.\bigskip

When $l=1$ or $l=2$, the approximations on $\log \eta_{k,l}$ that we were using before are not good enough, but we can treat these cases separately. When $l=1$,
\begin{align*}
\rho_{k,1}&=\frac{\lambda_{2}+k+n}{\lambda_{2}+k+n-1}\,\prod_{j=2}^{n}\frac{k+j-1+\lambda_{2}-\lambda_{j}}{k+j-2+\lambda_{2}-\lambda_{j}}\,\frac{k+\lambda_{2}+\lambda_{j}+2n+1-j}{k+\lambda_{2}+\lambda_{j}+2n-j}\\
&\leq \frac{k+n}{k+n-1}\,\prod_{j=2}^{n}\frac{k+j-1}{k+j-2}\,\frac{k+2n+1-j}{k+2n-j}=\frac{k+2n-1}{k};\\
\eta_{k,1}&\leq \frac{k+2n-1}{k}\,\E^{-\frac{2k+2n-1}{2n}\log n}.
\end{align*}
If $k=1$, which only happens once when one makes the partition grow, then the bound above is $2n\,\E^{-\frac{2n+1}{2n}\log n}\leq 2$. On the other hand, if $k \geq 2$, then the bound is decreasing in $k$ and therefore smaller than $\left(n+\frac{1}{2}\right)\E^{-\frac{2n+3}{2n}\log n}\leq1$. So, one also has $\eta_{k,1}\leq 1$ for any $k$ but $k=1$, where a correct bound is $2$. Similarly, when $l=2$,
\begin{align*}
\rho_{k,2}&=\frac{\lambda_{3}+k+n}{\lambda_{3}+k+n-2}\,\frac{2\lambda_{3}+2k+2n-1}{2\lambda_{3}+2k+2n-3}\,\prod_{j=3}^{n}\frac{k+j-1+\lambda_{3}-\lambda_{j}}{k+j-3+\lambda_{3}-\lambda_{j}}\,\frac{k+\lambda_{3}+\lambda_{j}+2n+1-j}{k+\lambda_{3}+\lambda_{j}+2n-1-j}\\
&\leq \frac{k+n}{k+n-2}\,\frac{2k+2n-1}{2k+2n-3}\,\prod_{j=3}^{n}\frac{k+j-1}{k+j-3}\,\frac{k+2n+1-j}{k+2n-1-j}= \frac{k+2n-2}{k}\,\frac{k+2n-1}{k+1}\,\frac{2k+2n-1}{2k+2n-3};\\
\eta_{k,2}&\leq \frac{k+2n-2}{k}\,\frac{k+2n-1}{k+1}\,\frac{2k+2n-1}{2k+2n-3}\,\E^{-\frac{2n+2k-2}{n}\log n}.
\end{align*}
Again, the last bound is decreasing in $k$, smaller than $2+\frac{1}{n}\leq \frac{7}{3}$ when $k=1$ and smaller than $1$ when $k=2$. Hence, $\eta_{k,2}\leq 1$ unless $k=1$, where a correct bound is $\frac{7}{3}$ (and again this situation occurs at most once whence making the partition grow).\bigskip

\noindent Conclusion: every quotient $\eta_{k,l}$ satisfies $\eta_{k,l}\leq 1$, but the two exceptions: $k=1$ and $l=1$ or $2$. The product of the bounds on these two exceptions is $2 \times \frac{7}{3}=\frac{14}{3}$, so for every partition $\lambda$, one has indeed
$$D^{\lambda}\,\E^{-\frac{t_{n,0}}{2}\,B_{n}(\lambda)} =\prod_{l=1}^{n}\prod_{k=1}^{\lambda_{l}-\lambda_{l+1}}\eta_{k,l} \leq \frac{14}{3}.\vspace{-7mm}$$
\end{proof}
\vspace{2mm}
\begin{remark}
A small refinement of the previous proof shows that the worst case is in fact the partition $(2,1,0,\ldots,0)_{n}$ --- by that we mean that any other partition has quotients $\rho_{k,l}$ that are smaller. Its dimension is provided by the exact formula 
$$D^{\lambda}=\frac{8n(n^{2}-1)}{3},$$
so one can replace the bound $\frac{14}{3}$ of Proposition \ref{superboundsymplectic} by $\frac{8}{3}$. 
\end{remark}
\comment{\begin{remark}
The main trick in the proof of Proposition \ref{superboundsymplectic} is the way we make our partitions grow. One could have tried something simpler, namely, make the first part $\lambda_{1}$ grow box by box, then the second part $\lambda_{2}$, \emph{etc.} However, with the same technique of estimation of quotients of dimensions, we were then not able to prove something better than 
$$D^{\lambda}\leq C\,\E^{-t_{n,0}\,B_{n}(\lambda)},$$
instead of $C\,\E^{-\frac{t_{n,0}}{2}\,B_{n}(\lambda)}$ --- this is in accordance with the bound $2\alpha\log n$ mentioned in Saloff-Coste's Theorem \ref{window}. By making the partitions grow layer by layer, we use in a much better way the fact that the parts $\lambda_{1},\lambda_{2},\ldots,\lambda_{n}$ are ordered, and the ``compensations'' in the growth of $D^{\lambda}$ given by the determinantal structure of its formula.
\end{remark}}
\bigskip

The upper bound \eqref{mainupper} is now an easy consequence of Lemma \ref{partition} and Proposition \ref{superboundsymplectic}. For any partition $\lambda$, notice that 
$$B_{n}(\lambda)\geq \frac{1}{2n}\sum_{i=1}^{n} (2n+2-2i)\lambda_{i}= \frac{1}{2n} \sum_{i=1}^{n} i(2n+1-i)(\lambda_{i}-\lambda_{i+1})\geq \frac{1}{2}\sum_{i=1}^{n}i(\lambda_{i}-\lambda_{i+1})=\frac{|\lambda|}{2}.$$
From this, one deduces that in the case of compact symplectic groups,
$$S_{n}(t_{n,\eps})=\sum_{\lambda \in \ym_{n}^{\star}} (D^{\lambda})^{2}\,\E^{-t_{n,\eps}\,B_{n}(\lambda)} \leq \frac{64}{9} \sum_{\lambda \in \ym_{n}^{\star}} \E^{-\eps |\lambda|\log n} \leq \frac{320}{9n^{\eps}}\leq \frac{36}{n^{\eps}}$$
if one assumes that $\frac{1}{n^{\eps}} \leq \frac{1}{2}$ (in order to apply Lemma \ref{partition}). By Proposition \ref{scaryseries}, one concludes that
$$\dtv^{\unit\SP(n,\Hq)}(\mu_{2(1+\eps)\log n},\mathrm{Haar}) \leq \frac{3}{n^{\frac{\eps}{2}}}.$$
Here one can remove the assumption $\frac{1}{n^{\eps}} \leq \frac{1}{2}$: otherwise, the right-hand side is bigger than $1$ and therefore the inequality is trivially satisfied. This ends the proof of the upper bound in the case of compact symplectic groups. For their quotients, one can still use Proposition \ref{superboundsymplectic}, as follows. For quaternionic Grassmannians, 
$$S_{n}\left(\frac{t_{n,\eps}}{2}\right)=\sum_{\lambda \in \ym\ym_{2q}^{\star}} D^{\lambda}\,\E^{-\frac{t_{n,\eps}}{2}\,B_{n}(\lambda)} \leq \frac{8}{3} \sum_{\lambda \in \ym_{n}^{\star}} \E^{-\frac{\eps}{2}|\lambda|\log n} \leq \frac{40}{3n^{\frac{\eps}{2}}}\leq \frac{16}{n^{\frac{\eps}{2}}}$$
assuming $\frac{1}{n^{\frac{\eps}{2}}}\leq \frac{1}{2}$. This implies that
$$\dtv^{\Gra(n,q,\Hq)}(\mu_{(1+\eps)\log n},\mathrm{Haar}) \leq \frac{2}{n^{\frac{\eps}{4}}}.$$
Again, the assumption on $n^{\frac{\eps}{2}}$ is superfluous, since otherwise the right-hand side is bigger than $1$. Exactly the same proof works for the spaces $\unit\SP(n)/\unit(n)$, with the same bound (it may be improved by using the fact that one looks only at even partitions).
\bigskip

\subsubsection{Odd special orthogonal groups and their quotients}\label{oddupper}
Though the same reasoning holds in every case, we unfortunately have to check case by case that everything works. For odd special orthogonal groups $\SO(2n+1,\R)$, set $t_{n,\eps}=2\,(1+\eps)\log (2n+1)$, with in particular $t_{n,0}=2\log (2n+1)$. The main difference between $\SO(2n+1)$ and $\unit\SP(n)$ is the appearance of half-partitions, which is solved by:
\begin{lemma}\label{halfpartition}
For any integer partition $\lambda$, denote $\lambda \boxplus \frac{1}{2}$ the half-partition $\lambda_{1}+\frac{1}{2},\lambda_{2}+\frac{1}{2},\ldots,\lambda_{n}+\frac{1}{2}$. 
$$\frac{D^{\lambda\boxplus \frac{1}{2}}}{D^{\lambda}}\,\E^{-\frac{t_{n,0}}{2}\left(B_{n}(\lambda\boxplus \frac{1}{2})-B_{n}(\lambda)\right)}\leq \E^{n\left(\log 2-\frac{\log(2n+1)}{4}\right)}\leq 2.$$
\end{lemma}
\begin{proof}
The quotient of dimensions is
$$\prod_{1 \leq i\leq j \leq n} \frac{\lambda_{i}+\lambda_{j}+2n+2-i-j}{\lambda_{i}+\lambda_{j}+2n+1-i-j}\leq\prod_{1 \leq i\leq j \leq n} \frac{2n+2-i-j}{2n+1-i-j}=2^{n},$$
and the difference $\frac{t_{n,0}}{2}\left(B_{n}(\lambda\boxplus \frac{1}{2})-B_{n}(\lambda)\right)$ is equal to
$$\frac{\log(2n+1)}{2n+1} \left(\sum_{i=1}^{n} \lambda_{i}+\frac{1}{4}+\frac{2n+1-2i}{2}\right) \geq \frac{\log(2n+1)}{2n+1}\left(\frac{n}{4}+\frac{n^{2}}{2}\right)=\frac{n \log (2n+1)}{4} .$$
This yields the first part of the inequality, and the second part is an easy analysis of the variations of the bound with respect to $n$.
\end{proof}
\bigskip

Then, for any integer partition $\lambda$, one can as before prove a uniform bound on $D^{\lambda}\,\E^{-\log(2n+1)\,B_{n}(\lambda)}$; the differences are tiny, \emph{e.g.}, in many formulas, $2n+2$ is replaced by $2n+1$, or $\frac{1}{2n}$ is replaced by $\frac{1}{2n+1}$. We refer to Appendix \ref{annexsoupper} for these computations. 
\begin{proposition}\label{tobeprovedinannexsoupper}
In the case of odd special orthogonal groups, at cut-off time,
$$D^{\lambda}\, \E^{-\frac{t_{n,0}}{2}\,B_{n}(\lambda)}\leq \frac{11}{10}$$
for any integer partition $\lambda$ of length $n$. For half-integer partitions, the bound is replaced by $\frac{11}{5}$.
\end{proposition}
\bigskip

There is one last computation that needs to be done, namely, the special case $\lambda=(\frac{1}{2},\ldots,\frac{1}{2})_{n}=(0,\ldots,0)_{n}\boxplus \frac{1}{2}$ --- it corresponds to the \emph{spin representation} of  $\SO(2n+1,\R)$. The value of $B_{n}(\lambda)$ is then $\frac{n}{4}$, and $D^{\lambda}=2^{n}$. Thus, in this special case,
$$(D^{\lambda})^{2} \,\E^{-t_{n,\eps}\,B_{n}(\lambda)} \leq \E^{n\log 4-\frac{n\log (2n+1)}{2}}\,\E^{-\frac{\eps n \log(2n+1)}{2}}\leq \frac{11}{4}\,\frac{1}{(2n+1)^{\eps}}$$
for every $n \geq 5$. On the other hand, 
$$B_{n}(\lambda)=\frac{1}{2n+1}\sum_{i=1}^{n}\lambda_{i}^{2}+i(2n-i)(\lambda_{i}-\lambda_{i+1}) \geq \frac{|\lambda|}{2n+1}+\frac{n}{2n+1}\sum_{i=1}^{n}i(\lambda_{i}-\lambda_{i+1})=\frac{(n+1)|\lambda|}{2n+1}\geq \frac{|\lambda|}{2},$$
so we can now write:
\begin{align*}
S_{n}(t_{n,\eps}) &\leq  \frac{11}{4}\,\frac{1}{(2n+1)^{\eps}}+ \sum_{\lambda \in \ym_{n}^{\star}} (D^{\lambda})^{2}\,\E^{-t_{n,\eps}\,B_{n}(\lambda)}+ (D^{\lambda \boxplus \frac{1}{2}})^{2}\,\E^{-t_{n,\eps}\,B_{n}(\lambda\boxplus \frac{1}{2})}\\
&\leq\frac{11}{4}\,\frac{1}{(2n+1)^{\eps}}+ \sum_{\lambda \in \ym_{n}^{\star}} \left((D^{\lambda})^{2}\,\E^{-t_{n,0}\,B_{n}(\lambda)}+ (D^{\lambda \boxplus \frac{1}{2}})^{2}\, \E^{-t_{n,0}\,B_{n}(\lambda\boxplus \frac{1}{2})}\right)\E^{-2\eps\log(2n+1)\,B_{n}(\lambda)}\\
&\leq\frac{11}{4}\,\frac{1}{(2n+1)^{\eps}}+ \sum_{\lambda \in \ym_{n}^{\star}} \left(\frac{121}{100}+\frac{121}{25}\right)\E^{-\eps|\lambda|\log(2n+1)}\\
&\leq\frac{11}{4}\,\frac{1}{(2n+1)^{\eps}}+ \frac{121}{20}\sum_{\lambda \in \ym_{n}^{\star}} \frac{1}{(2n+1)^{\eps|\lambda|}} \leq \frac{33}{(2n+1)^{\eps}}\leq \frac{144}{(2n+1)^{\eps}}
\end{align*}
if one assumes $\frac{1}{(2n+1)^{\eps}}\leq \frac{1}{2}$. Thus, by Proposition \ref{scaryseries},
$$\dtv^{\SO(2n+1,\R)}(\mu_{2\,(1+\eps)\log (2n+1)},\mathrm{Haar})\leq \frac{6}{(2n+1)^{\frac{\eps}{2}}}.$$
and again we can now remove the assumption $\frac{1}{(2n+1)^{\eps}}\leq \frac{1}{2}$. The same technique applies to odd real Grassmannians, with
$$ S_{n}\left(\frac{t_{n,\eps}}{2}\right)= \sum_{\lambda \in (2\ym_{q}\sqcup 2\ym_{q} \boxplus 1)^{\star}} D^{\lambda}\,\E^{-\frac{t_{n,\eps}}{2}\,B_{n}(\lambda)}\leq \frac{11}{10}\sum_{\lambda \in \ym_{n}^{\star}}\E^{-\frac{\eps}{2} |\lambda|\log(2n+1)} \leq \frac{55}{10(2n+1)^{\frac{\eps}{2}}}\leq \frac{16}{(2n+1)^{\frac{\eps}{2}}},$$
and therefore 
$$ \dtv^{\Gra(2n+1,q,\R)}(\mu_{(1+\eps)\log (2n+1)},\mathrm{Haar})\leq \frac{2}{(2n+1)^{\frac{\eps}{4}}}.$$
\bigskip

\subsubsection{Even special orthogonal groups and their quotients}
Though the computations have to be done once again, we shall prove exactly the same bounds as before for even special orthogonal groups and even real Grassmannians. Denote $t_{n,\eps}=2\,(1+\eps)\log(2n)$. The possibility of a sign $\pm$ for the last part $\lambda_{n}$ of the partitions leads to a coefficient $2$ in the series $S_{n}(t)$, and on the other hand, the case of half-partitions is reduced to the case of partitions by way of an analogue of Lemma \ref{halfpartition}. Indeed,
$$\frac{D^{\lambda\boxplus \frac{1}{2}}}{D^{\lambda}}\,\E^{-\frac{t_{n,0}}{2}\left(B_{n}(\lambda\boxplus \frac{1}{2})-B_{n}(\lambda)\right)}\leq \E^{n\log 2-\frac{(2n-1)\log(2n)}{8}}\leq \frac{12}{5}$$
for any $n \geq 5$ and any partition. Again, we put the proof of the following Proposition at the end of the paper, in Appendix \ref{annexsouppereven}.

\begin{proposition}\label{tobeprovedinannexsouppereven}
In the case of even special orthogonal groups, at cut-off time,
$$D^{\lambda}\, \E^{-\frac{t_{n,0}}{2}\,B_{n}(\lambda)}\leq \frac{4}{3} \quad\big(\text{respectively},\,\,\,\frac{48}{15}\big)$$
for any integer partition (resp. any half-partition) $\lambda$ of length $n$.
\end{proposition}\bigskip

Besides, the same proof as in the case of odd special orthogonal groups shows that $B_{n}(\lambda)\geq \frac{|\lambda|}{2}$ for any partition. For the special half-partition $\lambda = (0,\ldots,0)_{n}\boxplus \frac{1}{2}$ that cannot be treated by combining Lemmas \ref{partition} and \ref{halfpartition}, one has $D^{\lambda}=2^{n-1}$ and $B_{n}(\lambda)=\frac{n}{4}$, hence
$$(D^{\lambda})^{2} \,\E^{-t_{n,\eps}\,B_{n}(\lambda)} \leq \E^{(n-1)\log 4-\frac{n\log (2n)}{2}}\,\E^{-\frac{\eps n \log(2n)}{2}}\leq \frac{1}{(2n)^{\eps}}$$
for $n \geq 5$. We conclude that
\begin{align*}
\frac{1}{2} S_{n}(t_{n,\eps}) &\leq  \frac{1}{(2n)^{\eps}}+ \sum_{\lambda \in \ym_{n}^{\star}} (D^{\lambda})^{2}\,\E^{-t_{n,\eps}\,B_{n}(\lambda)}+ (D^{\lambda \boxplus \frac{1}{2}})^{2}\,\E^{-t_{n,\eps}\,B_{n}(\lambda\boxplus \frac{1}{2})}\\
&\leq\frac{1}{(2n)^{\eps}}+ \sum_{\lambda \in \ym_{n}^{\star}} \left(\frac{16}{9}+\frac{2304}{225}\right)\E^{-\eps|\lambda|\log(2n)}\leq\frac{2749}{45(2n)^{\eps}}\leq \frac{72}{(2n^{\eps})},
\end{align*}
and therefore, by Proposition \ref{scaryseries}, 
$$\dtv^{\SO(2n,\R)}(\mu_{2(1+\eps)\log (2n)},\mathrm{Haar})\leq \frac{6}{(2n)^{\frac{\eps}{2}}}.$$
For even real Grassmannian varieties, 
$$S_{n}\left(\frac{t_{n,\eps}}{2}\right)=\sum_{\lambda \in (2\ym_{q}\sqcup 2\ym_{q} \boxplus 1)^{\star}}D^{\lambda}\,\E^{-\frac{t_{n,\eps}}{2}\,B_{n}(\lambda)} \leq \frac{4}{3} \sum_{\lambda \in \ym_{n}^{\star}}\E^{-\frac{\eps}{2}|\lambda|\log(2n)} \leq \frac{20}{3(2n)^{\frac{\eps}{2}}}\leq \frac{16}{(2n)^{\frac{\eps}{2}}},$$
and again, the total variation distance is bounded by $2/(2n)^{\frac{\eps}{4}}$. So, the inequalities take the same form for even and odd special orthogonal groups or real Grassmannians, and the proof of the upper bound in this case is done. The same inequality holds also for the spaces of structures $\SO(2n)/\unit(n)$.
\bigskip

\subsubsection{Special unitary groups and their quotients}\label{unitaryupper}
 Set $t_{n,\eps}=2(1+\eps)\log n$. For special unitary groups, Weyl's dimension formula fortunately takes a much simpler form than before, but on the other hand, the computations on $B_{n}(\lambda)$ are this time a little more subtle. We shall still prove that almost every quotient $\eta_{k,l}$ of the quantities $D^{\lambda}\,\E^{-t_{n,0}\,B_{n}(\lambda)}$ with $\lambda$ going from
$$ (\lambda_{l+1}+k-1,\ldots,\lambda_{l+1}+k-1,\lambda_{l+1},\ldots,\lambda_{n-1})_{n-1}\quad \text{to}\quad  (\lambda_{l+1}+k,\ldots,\lambda_{l+1}+k,\lambda_{l+1},\ldots,\lambda_{n-1})_{n-1}$$
is smaller than $1$; but in practice, what will happen is that the negative exponentials may be much larger than before, whereas the quotients of dimensions $\rho_{k,l}$ will be much smaller. Consider for a start $\eta_{k,n-1}$. One has
$$\rho_{k,n-1}=\prod_{i=1}^{n-1}\frac{k+n-i}{k-1+n-i}=\frac{k+n-1}{k},$$
whereas $B_{n}(\lambda)$ is changed by $\frac{(n-1)(n+2k-1)}{n^{2}}$. So,
$$\eta_{k,n-1}=\frac{k+n-1}{k}\,\E^{-\frac{(n-1)(n+2k-1)}{n^{2}} \log n}\leq \begin{cases}
n\,\E^{-\frac{n^{2}-1}{n^{2}}\log n}=\E^{\frac{\log n}{n^{2}}} \leq 2^{\frac{1}{4}} &\text{if }k=1,\\
\frac{n+1}{2}\,\E^{-\frac{n^{2}+2n-3}{n^{2}}\log n} \leq \frac{n+1}{2n} \leq 1 &\text{if }k\geq 2,
\end{cases}$$
by using the decreasing behavior with respect to $k$. Notice that $\rho_{1,n-1}$ is indeed much smaller than before (linear in $n$ whereas before it grew exponentially in $n$), but $B_{n}(\lambda)$ for $k=1$ is almost constant instead of linear in $n$.\bigskip

In the general case,
$$\rho_{k,l}=\prod_{j=l+1}^{n}\frac{k'-\lambda_{j}+j-1}{k'-\lambda_{j}+j-l-1}\leq \prod_{j=l+1}^{n}\frac{k+j-1}{k+j-l-1}$$
with the usual notation $k'=k+\lambda_{l+1}$. On the other hand, the transformation on partitions makes $B_{n}(\lambda)$ change by
$$\frac{-l(n-l)(n+2k'-1)+2l |\lambda|_{l+1,n}}{n^{2}},$$
where $|\lambda|_{l+1,n}$ is the restricted size $\sum_{j=l+1}^{n}\lambda_{j}$. Notice now that 
$$-(n-l)k'+|\lambda|_{l+1,n}=\sum_{j=l+1}^{n}\lambda_{j}-\lambda_{l+1}-k \leq \sum_{j=l+1}^{n} -k = -(n-l)k.$$
So,
$$ \eta_{k,l} \leq \prod_{j=l+1}^{n}\frac{k+j-1}{k+j-l-1}\,\,\E^{-\frac{l(n-l)(n+2k-1)}{n^{2}}\log n} \leq \binom{n}{l}\,\E^{-\frac{l(n-l)(n+1)}{n^{2}}\log n}$$
which can as usual be estimated by Stirling (this is the same kind of computations as before). Hence, with $l \geq 3$, the last bound is always smaller than $1$, and also if $l=2$ unless $n=4$. If $n=4$ and $l=2$, then
$$\eta_{k,2}\leq \frac{(k+2)(k+3)}{k(k+1)}\,\E^{-\frac{3+2k}{2}\log2}\leq \begin{cases} \frac{3}{2^{3/2}} &\text{if }k=1,\\
1&\text{if }k\geq 2.
\end{cases}$$
Finally, when $l=1$, one has exactly the same bound as for $l=n-1$, so $2^{\frac{1}{4}}$ when $k=1$ and $1$ for $k=2$, Multiplying together all the bounds ($3/2^{\frac{3}{2}}$ and twice $2^{\frac{1}{4}}$), we obtain:
\begin{proposition}
In the case of special unitary groups, at cut-off time,
$$D^{\lambda}\, \E^{-\frac{t_{n,0}}{2}\,B_{n}(\lambda)}\leq \frac{3}{2}$$
for any integer partition $\lambda$ of length $n-1$.
\end{proposition}
\bigskip

Another big difference with the previous cases is that one cannot use Lemma \ref{partition} anymore. Indeed, for $\lambda=(k,\ldots,k)_{n-1}$, $B_{n}(\lambda)=\frac{k(n-1)}{n}=\frac{|\lambda|}{n}$, so there is no hope to have an inequality of the type $B_{n}(\lambda)\geq \alpha\,|\lambda|$ for any partition. That said, set $\delta_{i}=\lambda_{i}-\lambda_{i+1}$; then,
$$B_{n}(\lambda)=\frac{1}{n^{2}}\sum_{1 \leq i<j\leq n}(\lambda_{i}-\lambda_{j})^{2}+\frac{1}{n}\sum_{i=1}^{n-1}i(n-i)\,\delta_{i}\geq \sum_{i=1}^{n-1} \frac{i(n-i)}{n}\,\delta_{i}.$$ 
This leads us to study the series $$T_{n}(x)=\sum_{\delta_{1},\ldots,\delta_{n-1} \geq 0} x^{\sum_{i=1}^{n-1} \frac{i(n-i)}{n}\delta_{i}}=\prod_{i=1}^{n-1} \frac{1}{1-x^{\frac{i(n-i)}{n}}}.$$
Clearly, each $T_{n}(x)$ is convex on $\R_{+}$, so if we can show for example that $T_{n}\left(\frac{1}{8}\right)$ stays smaller than $1+\frac{K}{8}$ for every $n$, then we will also have the inequality $T_{n}(x)\leq 1+Kx$ for every $0\leq x \leq \frac{1}{8}$. Set $U_{n}(x)=\log(T_{n}(x))$; one has
$$U_{n}(x)=\sum_{i=1}^{n-1}-\log\left(1-x^{\frac{i(n-i)}{n}}\right) \leq \sum_{i=1}^{n-1}x^{\frac{i(n-i)}{n}}\leq 2 \sum_{i=1}^{\lfloor\frac{n}{2}\rfloor} x^{\frac{i}{2}} \leq \frac{2}{1-x^{\frac{1}{2}}}$$
for $0 \leq x \leq \frac{1}{8}$. It follows that $T_{n}(x)\leq 1+Kx$ with $K\leq 169$. Suppose $\frac{1}{n^{2\eps}}\leq \frac{1}{8}$. Then,
\begin{align*}
S_{n}(t_{n,\eps})&= \sum_{\lambda \in \ym_{n-1}^{\star}}(D^{\lambda})^{2}\,\E^{-t_{n,\eps}\,B_{n}(\lambda)}\leq \frac{9}{4}\sum_{\lambda \in \ym_{n-1}^{\star}} \left(\frac{1}{n^{2\eps}}\right)^{B_{n}(\lambda)} \leq \frac{9}{4} \left(T_{n}\!\left(\frac{1}{n^{2\eps}}\right)-1\right)\leq \frac{1521}{4n^{2\eps}}\leq \frac{400}{n^{2\eps}},
\end{align*}
which leads to
$$\dtv^{\SU(n,\C)}(\mu_{2(1+\eps)\log n},\mathrm{Haar})\leq \frac{10}{n^{\eps}}.$$
If $\frac{1}{n^{2\eps}}\geq \frac{1}{8}$, then this inequality is also trivially satisfied. Hence, the case of special unitary groups is done. For the quotients $\SU(n)/\SO(n)$, one obtains
$$S_{n}\left(\frac{t_{n,\eps}}{2}\right) \leq \frac{3}{2}\left( T_{n}\!\left(\frac{1}{n^{\eps}}\right)-1 \right)\leq \frac{507}{2n^{\eps}} \leq \frac{256}{n^{\eps}}$$
and therefore
$$\dtv^{\SU(n,\C)/\SO(n,\R)}(\mu_{(1+\eps)\log n},\mathrm{Haar})\leq \frac{8}{n^{\frac{\eps}{2}}}.$$
The proof is exactly the same for $\SU(2n)/\unit\SP(n)$ and gives the same inequality, however with $(2n)^{\frac{\eps}{2}}$ instead of $n^{\frac{\eps}{2}}$.\bigskip

For the complex Grassmannian varieties, we have seen that it was easier to see them as quotients of $\unit(n)$ (instead of $\SU(n)$), and this forces us to do some additional computations. Though the cut-off phenomenon also holds in the case of $\unit(n)$, the set of irreducible representations is then labelled by sequences of possibly negative integers, which makes our scheme of growth of partitions a little bit more cumbersome to apply. Fortunately, for Grassmannians, the spherical representations can be labelled by true partitions, but then the dimensions are given by a different formula and we have to do once again the estimates of quotients $\rho_{k,l}$ and $\eta_{k,l}$. We refer to Appendix \ref{annexsuupper} for a proof of the following: 
$$A_{n}(\lambda)\,\E^{-\log n\,B_{n}(\lambda)}\leq 1$$
for any partition. Then, one can compare directly $B_{n}(\lambda)$ to $|\lambda|$:
$$B_{n}(\lambda)=\frac{2}{n}\sum_{i=1}^{p}\lambda_{i}^{2}+(n+1-2i)\lambda_{i} \geq 2\sum_{i=1}^{p}\frac{i(n-i)}{n}(\lambda_{i}-\lambda_{i+1}) \geq \sum_{i=1}^{p}i(\lambda_{i}-\lambda_{i+1})=|\lambda|.$$
We conclude that
$$S_{n}\left(\frac{t_{n,\eps}}{2}\right)\leq \sum_{\lambda \in \ym_{q}^{*}} \E^{-\eps|\lambda|\log n} \leq \frac{5}{n^{\eps}}\leq \frac{16}{n^{\eps}}\quad;\quad\dtv^{\Gra(n,q,\C)} (\mu_{(1+\eps)\log n},\mathrm{Haar})\leq \frac{2}{n^{\frac{\eps}{2}}}$$
and this ends the proof of all the upper bounds of type \eqref{mainupper}.
\bigskip\bigskip

\section{Lower bounds before the cut-off time}\label{lower}
The proofs of the lower bounds before cut-off time rely on the following simple ideas. Denote $\lambda_{\min}$ the (spherical) irreducible representation ``of minimal eigenvalue'' identified in Section \ref{order}.  We then consider the random variable:
\begin{equation}
\Omega=\begin{cases}\chi^{\lambda_{\min}}(k) &\text{in the case of groups},\\
\sqrt{D^{\lambda_{\min}}}\,\phi^{\lambda_{\min}}(gK)&\text{in the case of symmetric spaces of type non-group}.
\end{cases}\label{discrimination}
\end{equation}
In this equation, $k$ or $gK$ will be taken at random either under the Haar measure of the space, or under a marginal law $\mu_{t}$ of the Brownian motion; we shall denote $\esper_{\infty}$ and $\esper_{t}$ the corresponding expectations. When $\Omega$ is real valued, we also denote $\Var_{\infty}$ and $\Var_{t}$ the corresponding variances:
$$\Var[\Omega]=\esper[\Omega^{2}]-\esper[\Omega]^{2}=\esper\!\left[(\Omega-\esper[\Omega])^{2}\right].$$
In the case of unitary groups and their quotients, $\Omega$ will be complex valued, and we shall use the notations $\Var_{\infty}$ and $\Var_{t}$ for the expectation of the square \emph{of the module} of $\Omega-\esper[\Omega]$:
$$\Var[\Omega]=\esper\!\left[|\Omega|^{2}\right]-\left|\esper[\Omega]\right|^{2}=\esper\!\left[|\Omega-\esper[\Omega]|^{2}\right].$$
The normalization of Equation \eqref{discrimination} is actually chosen so that $\Omega$ is in any case of mean $0$ and variance $1$ under the Haar measure. 
\begin{remark} In fact, much more is known about the asymptotic distribution of these functions under Haar measure, when $n$ goes to infinity; see \cite{DS94}. For instance, over the unitary groups, the moments of order smaller than $n_{0}$ of $\chi^{(1,0,\ldots,0)}(g)=\tr\, g$ agree with those of a standard complex gaussian variable as soon as $n$ is bigger than $n_{0}$. In particular, if $g$ is distributed according to the Haar measure of $\unit(n,\C)$, then $\tr\, g$ converges (without any normalization) towards a standard complex gaussian variable. One has similar results for orthogonal and symplectic groups, this time with standard real gaussian variables. As far as we know, the same problem with spherical functions on the classical symmetric spaces is still open, and certain computations performed in this section are related to this question. \end{remark}
\bigskip
One will also prove that under a marginal law $\mu_{t}$, the variance of $\Omega$ stays small for every value of $t$, whereas its mean before cut-off time is big (not at all near zero). Standard methods of moments allow then to prove that the probability of a event
$$E_{\alpha}=\{k\,\,|\,\,|\Omega(k)| \geq\alpha \}\quad\text{or}\quad\{gK\,\,|\,\,|\Omega(gK)| \geq\alpha \}$$
is before cut-off time near $1$ under $\mu_{t}$, and near $0$ under Haar measure (for an adequate choice of $\alpha$). This is sufficient to prove the lower bounds, see \S\ref{bienayme}; in other words, $\Omega$ is a discriminating random variable for the cut-off phenomenon. 
\bigskip

The method presented above reduces the problem mainly to the expansion in irreducible characters or in spherical zonal functions of $\Omega^{2}$ or of $|\Omega|^{2}$; \emph{cf.} \S\ref{zonal}. In the case of compact groups, this amounts simply to understand the tensor product of $V^{\lambda_{\min}}$ with itself, or with its conjugate when the character $\Omega$ is complex valued. However, for compact symmetric spaces of type non-group, this is far less obvious. Notice that a zonal spherical function $\phi^{\lambda}$ can be uniquely characterized by the following properties:\vspace{2mm}
\begin{itemize}
\item it is a linear combination of matrix coefficients of the representation $V^{\lambda}$: $$\phi^{\lambda}(gK)=\sum_{i=1}^{D^{\lambda}}\sum_{j=1}^{D^{\lambda}} c^{ij} \rho^{\lambda}_{ij}(gK).$$
\item it is in $\leb^{2}(G/K)^{K}$, \emph{i.e.}, it is $K$-bi-invariant; and it is normalized so that $\phi^{\lambda}(eK)=1$.\vspace{2mm}
\end{itemize} 
Consequently, if $(V^{\lambda_{\min}})^{\otimes 2}=V^{\nu_{1}}\oplus \cdots \oplus V^{\nu_{s}} \oplus V^{\epsilon_{1}}\oplus \cdots \oplus V^{\epsilon_{t}}$ with the $V^{\nu_{i}}$ spherical irreducible representations and the $V^{\epsilon_{j}}$ non-spherical irreducible representations, then there exists an expansion 
\begin{equation}
(\phi^{\lambda_{\min}})^{2}=c_{\nu_{1}}\phi^{\nu_{1}}+c_{\nu_{2}}\phi^{\nu_{2}}+\cdots+c_{\nu_{s}}\phi^{\nu_{s}}.
\label{squarezonal}
\end{equation}
Nonetheless, it seems difficult to guess at the same time the values of the coefficients $c_{\nu}$ in this expansion. The only ``easy'' computation is the coefficient of the constant function in $(\phi^{\lambda})^{2}$, or more generally in a product $\phi^{\lambda}\,\phi^{\rho}$:
 $$c_{\phi^{\mathbf{1}_{G}}}[\phi^{\lambda}\,\phi^{\rho}] = \int_{X} \phi^{\lambda}(x)\,\phi^{\rho}(x)\,dx=\begin{cases} 0 &\text{if }\phi^{\rho}\neq \overline{\phi^{\lambda}},\\
\frac{1}{D^{\lambda}}&\text{otherwise}.\end{cases}$$

\noindent As far as we know, for a general zonal spherical function, there is a definitive solution to Equation \eqref{squarezonal} only in the case of symmetric spaces of rank $1$, see \cite{Gas70}. For our problem, one can fortunately give in every case a geometric description of the discriminating spherical representation and of the corresponding spherical vector. This yields an expression of $\phi^{\lambda_{\min}}(gK)$ as a degree $2$ polynomial of the matrix coefficients of $g$. Now it turns out that the joint moments of these coefficients under $\mu_{t}$ and $\mu_{\infty}=\mathrm{Haar}$ can be calculated by mean of the stochastic differential equations defining the $G$-valued Brownian motion; see Lemma \ref{expectationcoefficients}, which we reproduce from \cite[Proposition 1.4]{Lev11}. As $(\phi^{\lambda_{\min}}(gK))^{2}$ or $|\phi^{\lambda_{\min}}(gK)|^{2}$ is also a polynomial in the coefficients $g_{ij}$, one can therefore compute its expectation under $\mu_{t}$, and this actually gives back the coefficients in the expansion \eqref{squarezonal}. Thus, the algebraic difficulties raised in our proof of the lower bounds will be solved by arguments of stochastic analysis.
\bigskip

\subsection{Expansion of the square of the discriminating zonal spherical functions}\label{zonal}
The orthogonality of characters or of zonal spherical functions ensures that for every non-trivial (spherical) irreducible representation $\lambda$,
\begin{align*}
\esper_{\infty}[\chi^{\lambda}]&=\esper_{\infty}[\chi^{\lambda}(k)\,\chi^{\mathbf{1}_{K}}(k)]=\scal{\chi^{\lambda}}{\chi^{\mathbf{1}_{K}}}_{\leb^{2}(K)}=0;\\
\esper_{\infty}\!\left[\sqrt{D^{\lambda}}\,\phi^{\lambda}\right]&=\sqrt{D^{\lambda}}\,\esper_{\infty}[\phi^{\lambda}(gK)\,\phi^{\mathbf{1}_{G}}(gK)]=\sqrt{D^{\lambda}}\,\scal{\phi^{\lambda}}{\phi^{\mathbf{1}_{K}}}_{\leb^{2}(G/K)}=0.
\end{align*}
The function corresponding to the trivial representation, which is just the constant function equal to $1$, has of course mean $1$ under the Haar measure, and also under $\mu_{t}$. On the other hand, Theorem \ref{explicitdensity} allows one to compute the mean of a non-trivial irreducible character of zonal spherical function under $\mu_{t}$:
\begin{align*}
\esper_{t}[\chi^{\lambda}]&=\int_{K} p_{t}^{K}(k)\,\chi^{\lambda}(k)\,dk=[\chi^{\lambda}](p_{t}^{K})=D^{\lambda}\,\E^{-\frac{t}{2}\,B_{n}(\lambda)}=\left\{A_{n}(\lambda)\,\E^{-t\,B_{n}(\lambda)}\right\}^{\frac{1}{2}}\\
\esper_{t}\!\left[\sqrt{D^{\lambda}}\,\phi^{\lambda}\right]&=\sqrt{D^{\lambda}}\int_{X=G/K} p_{t}^{X}(x)\,\phi^{\lambda}(x)\,dx=\sqrt{D^{\lambda}}\,\frac{[\phi^{\lambda}](p_{t}^{X})}{D^{\lambda}}=\left\{A_{n}(\lambda)\,\E^{-t\,B_{n}(\lambda)}\right\}^{\frac{1}{2}}
\end{align*}
with the notations of Proposition \ref{scaryseries}, and where $[\chi^{\lambda}](f)$ or $[\phi^{\lambda}](f)$ denotes the coefficient of $\chi^{\lambda}$ or $\phi^{\lambda}$ in the expansion of $f$. So, with the help of the table of Lemma \ref{decay}, we can compute readily $\esper_{t}[\Omega]$ in each case, and also $\esper_{\infty}[\Omega]$.
\bigskip

In order to estimate $\Var_{t}[\Omega]$ and $\Var_{\infty}[\Omega]$, we now need to find a representation-theoretic interpretation of either $\Omega^{2}$ when $\Omega$ is real-valued, or of $|\Omega|^{2}$ when $\Omega$ is complex-valued. We begin with compact groups:
\begin{lemma}\label{squaregroup}
Suppose $G=\SO(2n,\R)$ or $\SO(2n+1,\R)$ or $\unit\SP(n,\Hq)$. Then $\Omega=\chi^{(1,0,\ldots,0)_{n}}$ is real-valued, and 
\begin{equation}
\Omega^{2}=(\chi^{(1,0,\ldots,0)_{n}})^{2}=\chi^{(2,0,\ldots,0)_{n}}+\chi^{(1,1,0,\ldots,0)_{n}}+\chi^{(0,0,\ldots,0)_{n}}.\label{squareso}
\end{equation}
On the other hand, when $G=\SU(n,\C)$, $\Omega$ is complex-valued, and 
\begin{equation}
|\Omega|^{2}=\chi^{(1,0,\ldots,0)_{n-1}}\,\chi^{(1,\ldots,1)_{n-1}}=\chi^{(2,1,\ldots,1)_{n-1}}+ \chi^{(0,0,\ldots,0)_{n-1}}.\label{squareunit}
\end{equation}
\end{lemma}
\begin{proof}
In each case, $\Omega(k)=\tr\,k$, up to the map \eqref{doublequaternion} in the symplectic case; this explains why $\Omega$ is real-valued in the orthogonal and symplectic case, and complex-valued in the unitary case. Then, the simplest way to prove the identities \eqref{squareso} and \eqref{squareunit} is by manipulating the Schur functions of type $\mathrm{A}$, $\mathrm{B}$, $\mathrm{C}$ and $\mathrm{D}$; indeed, these polynomials evaluated on the eigenvalues are known to be the irreducible characters of the corresponding groups, see \S\ref{explicit}. We start with the special orthogonal groups. In type $\mathrm{B}_{n}$, $(z_{1}+\cdots+z_{n}+z_{1}^{-1}+\cdots+z_{n}^{-1}+1)^{2}$ is indeed equal to the sum of the three terms
\begin{align*}
sb_{(2,0,\ldots,0)}(Z,Z^{-1},1)&=\left(\sum_{1\leq i\leq j\leq n}z_{i}z_{j}+z_{i}z_{j}^{-1}+z_{i}^{-1}z_{j}+z_{i}^{-1}z_{j}^{-1}\right)+\left(\sum_{i=1}^{n}z_{i}+z_{i}^{-1} \right)-n;\\
sb_{(1,1,0,\ldots,0)}(Z,Z^{-1},1)&=\left(\sum_{1\leq i< j\leq n}z_{i}z_{j}+z_{i}z_{j}^{-1}+z_{i}^{-1}z_{j}+z_{i}^{-1}z_{j}^{-1}\right)+\left(\sum_{i=1}^{n}z_{i}+z_{i}^{-1} \right)+n;\\
sb_{(0,\ldots,0)}(Z,Z^{-1},1)&=1;
\end{align*}
whereas in type $\mathrm{D}_{n}$, $(z_{1}+\cdots+z_{n}+z_{1}^{-1}+\cdots+z_{n}^{-1})^{2}$ is equal to the sum of the three terms
\begin{align*}
sd_{(2,0,\ldots,0)}(Z,Z^{-1})&=\left(\sum_{1\leq i\leq j\leq n}z_{i}z_{j}+z_{i}z_{j}^{-1}+z_{i}^{-1}z_{j}+z_{i}^{-1}z_{j}^{-1}\right)-n-1;\\
sd_{(1,1,0,\ldots,0)}(Z,Z^{-1})&=\left(\sum_{1\leq i< j\leq n}z_{i}z_{j}+z_{i}z_{j}^{-1}+z_{i}^{-1}z_{j}+z_{i}^{-1}z_{j}^{-1}\right)+n;\\
sd_{(0,\ldots,0)}(Z,Z^{-1})&=1.
\end{align*}
For compact symplectic groups, hence in type $\mathrm{C}_{n}$, $(z_{1}+\cdots+z_{n}+z_{1}^{-1}+\cdots+z_{n}^{-1})^{2}$ is indeed equal to the sum of the three terms
\begin{align*}
sc_{(2,0,\ldots,0)}(Z,Z^{-1})&=\left(\sum_{1\leq i\leq j\leq n}z_{i}z_{j}+z_{i}z_{j}^{-1}+z_{i}^{-1}z_{j}+z_{i}^{-1}z_{j}^{-1}\right)-n;\\
sc_{(1,1,0,\ldots,0)}(Z,Z^{-1})&=\left(\sum_{1\leq i< j\leq n}z_{i}z_{j}+z_{i}z_{j}^{-1}+z_{i}^{-1}z_{j}+z_{i}^{-1}z_{j}^{-1}\right)+n-1;\\
sc_{(0,\ldots,0)}(Z,Z^{-1})&=1;
\end{align*}
and this is also $(sc_{(1,0,\ldots,0)}(Z,Z^{-1}))^{2}=(\chi^{(1,0,\ldots,0)}(k))^{2}=\Omega(k)^{2}$. Thus, Formula \eqref{squareso} is proved. In type $\mathrm{A}_{n-1}$, notice that for every character $\chi^{\lambda}$, $\overline{\chi^{\lambda}(k)}=\chi^{\lambda}(k^{-1})=\chi^{\lambda^{*}}(k)$, where $\lambda^{*}$ is the sequence obtained from $\lambda$ by the simple transformation
\begin{equation}
(\lambda_{1}\geq \lambda_{2}\geq \cdots \geq \lambda_{n-1})_{n-1} \mapsto (\lambda_{1} \geq \lambda_{1}- \lambda_{n-1} \geq \cdots \geq \lambda_{1}- \lambda_{2})_{n-1}.\label{conjugaterepresentation}
\end{equation}
Indeed, if $z_{1},\ldots,z_{n}$ are the eigenvalues of $k$, then
\begin{align*}
\overline{\chi^{\lambda}(k)}&=s_{(\lambda_{1},\ldots,\lambda_{n-1})_{n-1}}(z_{1}^{-1},\ldots,z_{n}^{-1})=s_{(\lambda_{1},\ldots,\lambda_{n-1},0)_{n}}(z_{1}^{-1},\ldots,z_{n}^{-1})=s_{(0,-\lambda_{n-1},\ldots,-\lambda_{1})_{n}}(z_{n},\ldots,z_{1})\\
&=s_{(\lambda_{1},\lambda_{1}-\lambda_{n-1},\ldots,0)_{n}}(z_{1},\ldots,z_{n})=s_{(\lambda_{1},\lambda_{1}-\lambda_{n-1},\ldots,\lambda_{1}-\lambda_{2})_{n-1}}(z_{1},\ldots,z_{n})=\chi^{\lambda^{*}}(k)
\end{align*}
Here, one uses the relation $z_{1}z_{2}\cdots z_{n}=1$ for every element of the torus of $\SU(n,\C)$, which enables one to transform a $n$-vector of possibly negative integers into a $(n-1)$-vector of non-negative integers. In particular, 
$|\Omega(k)|^{2}=|\chi^{(1,0,\ldots,0)_{n-1}}(k)|^{2}=\chi^{(1,0,\ldots,0)_{n-1}}(k)\,\chi^{(1,1,\ldots,1)_{n-1}}(k).$
Then, a simple calculation with symmetric functions yields Formula \eqref{squareunit}:
\begin{align*}
\chi^{(1,0,\ldots,0)_{n-1}}(k)\,\chi^{(1,1,\ldots,1)_{n-1}}(k)&=(z_{1}+\cdots+z_{n})(z_{1}^{-1}+\cdots+z_{n}^{-1})\\
&=\left(n-1+\sum_{i<j}z_{i}z_{j}^{-1}+z_{i}^{-1}z_{j} \right)+1\\
&=s_{(1,0,\ldots,0,-1)_{n}}(Z)+s_{(0,\ldots,0)_{n}}(Z)=s_{(2,1,\ldots,1)_{n-1}}(Z)+s_{(0,\ldots,0)_{n-1}}(Z)\\
&=\chi^{(2,1,\ldots,1)_{n-1}}(k)+\chi^{(0,\ldots,0)_{n-1}}(k)
\end{align*}
where $Z=\{z_{1},\ldots,z_{n}\}$ is the alphabet of the eigenvalues of $k$.
\end{proof}
\bigskip

\subsubsection{Values of the zonal functions and abstract expansions of their squares} As explained in the introduction of this part, the case of compact symmetric spaces of type non-group is much more involved. We start by finding an expression of $\Omega(gK)$ in terms of the matrix coefficients $g_{ij}$ of the matrix $g$.
\begin{proposition}\label{coeffsphericalfunction}
In terms of the matrix coefficients of $g$, $\phi^{\lambda_{\min}}(gK)$ is given by:\vspace{2mm}
$$
\begin{tabular}{|c|c|c|c|}
\hline\vspace{-2.5mm}&&&\\
$G/K$&  $V^{\lambda_{\min}}$  & $\phi^{\lambda_{\min}}(gK)$ & $\Bbbk$ \\
 \vspace{-2.5mm}&&&\\
\hline\hline \vspace{-2.5mm}&&&\\
$\Gra(n,q,\R)$ &$ \mathfrak{so}^{\perp}(n,\C)$  & $\frac{1}{p}\sum_{i=1}^{p}\sum_{j=1}^{p}(g_{ij})^{2}+\frac{1}{q}\sum_{i=p+1}^{n}\sum_{j=p+1}^{n}(g_{ij})^{2}-1$ & $\R$\\
\vspace{-2.5mm}&& &\\
\hline \vspace{-2.5mm}&& &\\ 
$\Gra(n,q,\C)$ &$\mathfrak{sl}(n,\C)$  & $\frac{1}{p}\sum_{i=1}^{p}\sum_{j=1}^{p}|g_{ij}|^{2}+\frac{1}{q}\sum_{i=p+1}^{n}\sum_{j=p+1}^{n}|g_{ij}|^{2}-1$ & $\R$\\
\vspace{-2.5mm}&& &\\
\hline \vspace{-2.5mm}&& &\\ 
$\Gra(n,q,\Hq)$ &$ \!\mathfrak{sp}^{\perp}(2n,\C) \!$  & $\frac{1}{p}\sum_{i=1}^{p}\sum_{j=1}^{p}|g_{ij}|^{2}+\frac{1}{q}\sum_{i=p+1}^{n}\sum_{j=p+1}^{n}|g_{ij}|^{2}-1$ & $\R$\\
&& &\\
\hline
\end{tabular}$$
$$
\begin{tabular}{|c|c|c|c|}
\hline \vspace{-2.5mm}&&&\\
$\SO(2n,\R)/\unit(n,\C)$ &$\mathcal{A}^{2}(\C^{2n})$  & $\frac{1}{n}\sum_{i=1}^{n}\sum_{j=1}^{n}g_{(2i)(2j)}g_{(2i-1)(2j-1)}-g_{(2i)(2j-1)}g_{(2i-1)(2j)}  $ & $\R$\\
\vspace{-2.5mm}&& &\\
\hline \vspace{-2.5mm}&& &\\ 
$\SU(n,\C)/\SO(n,\R)$ &$\mathcal{S}^{2}(\C^{n})$  & $\frac{1}{n}\sum_{i=1}^{n}\sum_{j=1}^{n}(g_{ij})^{2}$ & $\C$\\
\vspace{-2.5mm}&& &\\
\hline \vspace{-2.5mm}&& &\\ 
$\SU(2n,\C)/\unit\SP(n,\Hq)$ &$\mathcal{A}^{2}(\C^{2n})$  & $\frac{1}{n}\sum_{i=1}^{n}\sum_{j=1}^{n}g_{(2i)(2j)}g_{(2i-1)(2j-1)}-g_{(2i)(2j-1)}g_{(2i-1)(2j)} $ & $\C$\\
\vspace{-2.5mm}&& &\\
\hline \vspace{-2.5mm}& &&\\ 
$\unit\SP(n,\Hq)/\unit(n,\C)$ &$\mathcal{S}^{2}(\C^{2n})$  & $\frac{1}{n}\sum_{i=1}^{n}\sum_{j=1}^{n} ([1](g_{ij}))^{2}+([\mathrm{j}](g_{ij}))^{2}-([\I](g_{ij}))^{2}-([\mathrm{k}](g_{ij}))^{2}$ & $\R$\\
&&&\\
\hline
\end{tabular}
$$\vspace{0mm}

\noindent For real Grassmannians, $\mathfrak{so}^{\perp}(n,\C)$ denotes the orthogonal complement of $\mathfrak{so}(n,\C)$ in $\mathfrak{sl}(n,\C)$; and for quaternionic Grassmannians, $\mathfrak{sp}^{\perp}(2n,\C)$ denotes the orthogonal complement of $\mathfrak{sp}(2n,\C)$ in $\mathfrak{sl}(2n,\C)$.
\end{proposition}

\begin{proof}
Each space $V^{\lambda_{\min}}$ described in the statement of our proposition is endowed with a natural action of $G=\SO(n)$ or $\SU(n)$ or $\unit\SP(n)$, namely, the action by conjugation in the case of Grassmannian varieties, and the diagonal action on tensors in the case of spaces of structures. Then, to say that
\begin{align*}
&V^{(2,0,\ldots,0)_{\lfloor \frac{n}{2}\rfloor}}_{\SO(n,\R)}=\mathfrak{so}^{\perp}(n,\C)\qquad;\qquad V^{(1,0,\ldots,0,-1)_{n}}_{\unit(n,\C)}=\mathfrak{sl}(n,\C) \qquad;\qquad V^{(1,1,0,\ldots,0)_{n}}_{\unit\SP(n,\Hq)}=\mathfrak{sp}^{\perp}(2n,\C)\qquad;\\
&V^{(1,1,0,\ldots,0)_{n}}_{\SO(2n,\R)}=\mathcal{A}^{2}(\C^{n}) \qquad;\qquad V^{(2,0,\ldots,0)_{n-1}}_{\SU(n,\C)}=\mathcal{S}^{2}(\C^{n}) \qquad;\qquad V^{(1,1,0,\ldots,0)_{2n-1}}_{\SU(2n,\C)}=\mathcal{A}^{2}(\C^{2n}) \qquad;\\
&V^{(2,0,\ldots,0)_{n}}_{\unit\SP(n,\Hq)}=\mathcal{S}^{2}(\C^{2n})
\end{align*}
is equivalent to the following statements: the trace of $g \in \SO(n,\R)$ acting on $\mathfrak{so}^{\perp}(n,\C)$ is given by the Schur function of type $\mathrm{B}$ or $\mathrm{D}$ and label $(2,0,\ldots,0)_{\lfloor \frac{n}{2}\rfloor}$; the trace of $g \in \unit(n,\C)$ acting on $\mathfrak{sl}(n,\C)$ is given by the Schur function of type $\mathrm{A}$ and label $(1,0,\ldots,0,-1)_{n}$; \emph{etc.} Let us detail for instance this last case. We have seen in the previous Lemma that 
$$s_{(1,0,\ldots,0,-1)_{n}}(Z)=(z_{1}+\cdots+z_{n})(z_{1}^{-1}+\cdots+z_{n}^{-1})-1.$$
On the other hand, the module $\mathfrak{gl}(n,\C)$ on which $\SU(n,\C)$ acts by conjugation is the tensor product of modules $(\C^{n})\otimes (\C^{n})^{*}$. It follows that the trace of the action by conjugation of $g \in \SU(n,\C)$ on $\mathfrak{gl}(n,\C)$ is 
$$\chi(g)=(\tr g) \,(\tr (g^{-1})^{t})=(z_{1}+\cdots+z_{n})(z_{1}^{-1}+\cdots+z_{n}^{-1})$$
if $z_{1},\ldots,z_{n}$ are the eigenvalues of $g$. Subtracting $1$ amounts to look at the irreducible submodule $\mathfrak{sl}(n,\C)$ inside $\mathfrak{gl}(n,\C)$. The other cases are entirely similar, and the corresponding values of the Schur functions have all been computed in Lemma \ref{squaregroup}.\bigskip

Once the discriminating representations have been given a geometric interpretation, it is easy to find the corresponding $K$-invariant (spherical) vectors. We endow each space of matrices with the invariant scalar product $\scal{M}{N}=\tr\,MN^{\dagger}$, and each space of tensors with the scalar product $\scal{x_{1}\otimes x_{2}}{y_{1} \otimes y_{2}}=\scal{x_{1}}{y_{1}}\,\scal{x_{2}}{y_{2}}$, where $\scal{v}{w}$ is the usual Hermitian scalar product on $\C^{n}$ or $\C^{2n}$. We also denote $(e_{i})_{i}$ the canonical basis of $\C^{n}$ or $\C^{2n}$. Then, the $K$-spherical vectors write as:\vspace{2mm}
$$
\begin{tabular}{|c|c|c|}
\hline\vspace{-2.5mm}&&\\
$G$&  $K$  & $e^{\lambda_{\min}}$ \\
 \vspace{-2.5mm}&&\\
\hline\hline \vspace{-2.5mm}&&\\
$\SO(n)$ & $\SO(p)\times \SO(q)$  & $\frac{1}{\sqrt{npq}}\,\left(\begin{smallmatrix} -qI_{p} & 0 \\ 0 & pI_{q} \end{smallmatrix}\right)$\\
\vspace{-2.5mm}&& \\
\hline \vspace{-2.5mm}&& \\ 
$\SU(n)$ &$\mathrm{S}(U(p)\times U(q))$ & $\frac{1}{\sqrt{npq}}\,\left(\begin{smallmatrix} -qI_{p} & 0 \\ 0 & pI_{q} \end{smallmatrix}\right)$\\
\vspace{-2.5mm}&& \\
\hline \vspace{-2.5mm}&& \\ 
$\unit\SP(n)$ & $\unit\SP(p) \times \unit\SP(q)$ &  $\frac{1}{\sqrt{2npq}}\,\left(\begin{smallmatrix} -qI_{2p} & 0 \\ 0 & pI_{2q} \end{smallmatrix}\right)$\\
&& \\
\hline
\end{tabular}$$
$$
\begin{tabular}{|c|c|c|}
\hline \vspace{-2.5mm}&&\\
$\SO(2n)$ & $\unit(n) $ & $\frac{1}{\sqrt{2n}} \sum_{i=1}^{n} e_{2i}\otimes e_{2i-1}-e_{2i-1}\otimes e_{2i}$ \\
\vspace{-2.5mm}&& \\
\hline \vspace{-2.5mm}&& \\ 
$\SU(n)$ &$\SO(n)$   & $\frac{1}{\sqrt{n}}\sum_{i=1}^{n} e_{i} \otimes e_{i}$\\
\vspace{-2.5mm}&& \\
\hline \vspace{-2.5mm}&& \\ 
$\SU(2n)$ &$\unit\SP(n)$  & $\frac{1}{\sqrt{2n}} \sum_{i=1}^{n} e_{2i}\otimes e_{2i-1}-e_{2i-1}\otimes e_{2i}$ \\
\vspace{-2.5mm}&& \\
\hline \vspace{-2.5mm}& &\\ 
$\unit\SP(n)$ &$\unit(n)$  & $\frac{1}{\sqrt{2n}}\sum_{i=1}^{2n} e_{i} \otimes e_{i}$\\
&&\\
\hline
\end{tabular}
$$\vspace{2mm}

In each case, $e^{\lambda_{\min}}$ belongs trivially to $V^{\lambda_{\min}}$ and is of norm $1$, so the only thing to check then is the $K$-invariance. In the case of Grassmannian varieties, the matrix $e^{\lambda_{\min}}$ commutes indeed with $\mathrm{G}(p) \times \mathrm{G}(q)$, since it is also $(p,q)$-block-diagonal and with scalar multiples of the identity matrix in each diagonal block. The notation $\unit\SP(n,\Hq)$ used in this paper was meant to avoid any confusion between $\SP(2n,\C)$ and its compact form, the compact symplectic group. For $\unit(n)$ inside $\SO(2n)$, we use the well-known fact that inside $\SL(2n,\C)$,
\begin{equation} \SO(2n,\R) \cap \SP(2n,\C) \simeq \unit(n,\C),\label{magicintersection}
\end{equation}
the isomorphism being given by the map \eqref{doublecomplex}. This implies in particular that $\unit(n)$ leaves invariant the skew-symmetric tensor  $\sum_{i=1}^{n} e_{2i}\otimes e_{2i-1}-e_{2i-1}\otimes e_{2i}$ corresponding to the skew-symmetric form defining $\SP(2n,\C)$. The intersection formula \eqref{magicintersection} also proves that $\unit(n)$ leaves invariant the symmetric tensor $\sum_{i=1}^{2n} e_{i} \otimes e_{i}$, whence the value of the spherical vector for $\unit(n)$ inside $\unit\SP(n)$. Finally, for $\SO(n)$ inside $\SU(n)$ and $\unit\SP(n)$ inside $\SU(2n)$, we use again the defining symmetric bilinear form or skew-symmetric bilinear form associated to the group $K$ to construct a $K$-invariant vector. \bigskip

The value of $\phi^{\lambda_{\min}}$ is then given by the formula $\phi^{\lambda}(g)=\scal{e^{\lambda}}{\rho^{\lambda}(g)e^{\lambda}}$, that is to say
$$
\tr(M_{p,q}\,g\,M_{p,q}\,g^{t})\quad;\quad \tr(M_{p,q}\,g\,M_{p,q}\,g^{\dagger}) \quad;\quad \frac{1}{2}\,\tr(\widetilde{M}_{p,q}\,\widetilde{g}\,\widetilde{M}_{p,q}\,\widetilde{g}^{\dagger})
$$
for real, complex and quaternionic Grassmannians; 
$$\frac{1}{n}\sum_{i,j=1}^{n}(g_{ij})^{2} \quad;\quad \frac{1}{2n}\sum_{i,j=1}^{2n}(\widetilde{g}_{ij})^{2}$$
for $\SU(n)/\SO(n)$ and $\unit\SP(n)/\unit(n)$; and
$$\frac{1}{n}\sum_{i=1}^{n}\sum_{j=1}^{n}g_{(2i)(2j)}g_{(2i-1)(2j-1)}-g_{(2i)(2j-1)}g_{(2i-1)(2j)} $$
for $\SO(2n)/\unit(n)$ and $\SU(2n)/\unit\SP(n)$. Here by $\widetilde{g}$ we mean the complex matrix of size $2n\times 2n$ obtained from a quaternionic matrix of size $n \times n$ by the map \eqref{doublequaternion}. In this last case, the computations can in fact be done inside $\matr(n,\Hq)$: indeed,
\begin{align*}
&(\widetilde{g}_{(2i-1)(2j-1)})^{2}+(\widetilde{g}_{(2i-1)(2j)})^{2}+(\widetilde{g}_{(2i)(2j-1)})^{2}+(\widetilde{g}_{(2i)(2j)})^{2}\\
&=2\big(([1](g_{ij}))^{2}+([\mathrm{j}](g_{ij}))^{2}-([\I](g_{ij}))^{2}-([\mathrm{k}](g_{ij}))^{2}\big),
\end{align*}
whereas $\widetilde{M^{\star}}=(\widetilde{M})^{\dagger}$ and $\frac{1}{2}\,\tr \widetilde{M}=\Re(\tr\, M)$. Thus, the formulas for the discriminating spherical functions of the spaces of structures are entirely proved, and for Grassmannian varieties, it suffices to check that for any unitary quaternionic matrix $N$,
$$\Re(\tr\,M_{p,q}NM_{p,q}N^{\star})=\frac{1}{p}\sum_{i=1}^{p}\sum_{j=1}^{p}|g_{ij}|^{2}+\frac{1}{q}\sum_{i=p+1}^{n}\sum_{j=p+1}^{n}|g_{ij}|^{2}-1;$$
indeed the real and complex cases are specializations of this formula. This is easily done.
\end{proof}\bigskip

\begin{lemma}\label{abstractexpansion}
There exists coefficients $a,b,c,\ldots$ (different on each line, and depending on $n$ and $q$) such that the following expansions hold:
\begin{align*}
&\Gra(n,q,\R): \quad\left(\phi^{(2,0,\ldots,0)_{\lfloor \frac{n}{2} \rfloor}}\right)^{2}= \frac{2}{n^{2}+n-2}+a\, \phi^{(2,0,\ldots,0)_{\lfloor \frac{n}{2} \rfloor}}+b\,\phi^{(1,1,0\ldots,0)_{\lfloor \frac{n}{2} \rfloor}}\\
&\qquad\qquad\qquad\qquad\qquad\qquad\qquad +c \,\phi^{(2,2,0,\ldots,0)_{\lfloor \frac{n}{2} \rfloor}}+d\, \phi^{(3,1,0,\ldots,0)_{\lfloor \frac{n}{2} \rfloor}}+e \, \phi^{(4,0,\ldots,0)_{\lfloor \frac{n}{2} \rfloor}};\\
&\Gra(n,q,\C):\quad \left(\phi^{(2,1,\ldots,1)_{n-1}}\right)^{2}= \frac{1}{n^{2}-1}+a\,\phi^{(2,1,\ldots,1)_{n-1}}+b\,\phi^{(4,2,\ldots,2)_{n-1}}+c\,\phi^{(2,2,1,\ldots,1,0)_{n-1}};\\
&\Gra(n,q,\Hq): \quad\left(\phi^{(1,1,0,\ldots,0)_{n}}\right)^{2}= \frac{1}{2n^{2}-n-1}+a\,\phi^{(1^2,0,\ldots,0)_{n-1}}+b\,\phi^{(1^4,0,\ldots,0)_{n}}+c\,\phi^{(2,2,0,\ldots,0)_{n}};\\
&\SO(2n)/\unit(n): \quad\left(\phi^{(1,1,0,\ldots,0)_{n}}\right)^{2}= \frac{1}{2n^{2}-n}+a\,\phi^{(1^{2},0,\ldots,0)_{n}}+b\,\phi^{(1^{4},0,\ldots,0)_{n}}+c\,\phi^{(2,2,0,\ldots,0)_{n}};\\
&\SU(n)/\SO(n): \quad\left|\phi^{(2,0,\ldots,0)_{n-1}}\right|^{2}= \frac{2}{n^{2}+n}+a\,\phi^{(4,2,\ldots,2)_{n-1}};\\
&\SU(2n)/\unit\SP(n):\quad \left|\phi^{(1,1,0,\ldots,0)_{2n-1}}\right|^{2}= \frac{1}{2n^{2}-n}+ a \,\phi^{(2,2,1,\ldots,1,0)_{2n-1}};\\
&\unit\SP(n)/\unit(n): \quad\left(\phi^{(2,0,\ldots,0)_{n}}\right)^{2}= \frac{1}{2n^{2}+n}+a\,\phi^{(2,0,\ldots,0)_{n}}+b\,\phi^{(2,2,0,\ldots,0)_{n}}+c\,\phi^{(4,0,\ldots,0)_{n}}.
\end{align*}
In these formulas, it is understood that if the label $\lambda$ of the spherical function $\phi^{\lambda}$ does not make sense for a choice of $n$ and $q$, then this term does not appear in the expansion.
\end{lemma}
\begin{proof}
Each time, one computes the expansion in irreducible representations of $V^{\lambda_{\min}}\otimes V^{\lambda_{\min}}$ in the case of real-valued spherical functions, and of $V^{\lambda_{\min}}\otimes V^{\lambda_{\min}^{*}}$ in the case of complex-valued spherical functions, where $\lambda \mapsto \lambda^{*}$ is the transformation of weights given by Equation \eqref{conjugaterepresentation}. This expansion can be found with Schur functions; let us detail for instance the case of complex Grassmannian varieties $\Gra(n,q,\C)$. With an alphabet of eigenvalues $Z=\{z_{1},\ldots,z_{n}\}$ such that $z_{1}z_{2}\cdots z_{n}=1$, one has
\begin{align*}
s_{(0,\ldots,0)_{n-1}}(Z)&=1\\
s_{(2,1,\ldots,1)_{n-1}}(Z)&=s_{(1,0,\ldots,0,-1)_{n}}(Z)=\left(\sum_{i,j=1}^{n} z_{i}z_{j}^{-1}\right)-1\\
s_{(4,2,\ldots,2)_{n-1}}(Z)&=s_{(2,0,\ldots,0,-2)_{n}}(Z)=\left(\sum_{1 \leq i \leq j \leq n}\sum_{1 \leq k \leq l \leq n}z_{i}z_{j}z_{k}^{-1}z_{l}^{-1}\right)-\left(\sum_{i,j=1}^{n}z_{i}z_{j}^{-1}\right)\\
s_{(2,2,1,\ldots,1,0)_{n-1}}(Z)&=s_{(1,1,0,\ldots,0,-1,-1)_{n}}(Z)=\left(\sum_{1 \leq i < j \leq n}\sum_{1 \leq k < l \leq n}z_{i}z_{j}z_{k}^{-1}z_{l}^{-1}\right)-\left(\sum_{i,j=1}^{n}z_{i}z_{j}^{-1}\right)\\
s_{(3,1,\ldots,1,0)_{n-1}}(Z)&=s_{(2,0,\ldots,0,-1,-1)_{n}}(Z)=\left(\sum_{1 \leq i \leq j \leq n}\sum_{1 \leq k < l \leq n}z_{i}z_{j}z_{k}^{-1}z_{l}^{-1}\right)-\left(\sum_{i,j=1}^{n}z_{i}z_{j}^{-1}\right) +1\\
s_{(3,3,2,\ldots,2)_{n-1}}(Z)&=s_{(1,1,0,\ldots,0,-2)_{n}}(Z)=\left(\sum_{1 \leq i < j \leq n}\sum_{1 \leq k \leq l \leq n}z_{i}z_{j}z_{k}^{-1}z_{l}^{-1}\right)-\left(\sum_{i,j=1}^{n}z_{i}z_{j}^{-1}\right) +1.
\end{align*}
Consequently,
\begin{align*}(V^{(2,1,\ldots,1)_{n-1}})^{\otimes 2}&=V^{(0,\ldots,0)_{n-1}} \oplus 2\,V^{(2,1,\ldots,1)_{n-1}} \oplus V^{(4,2,\ldots,2)_{n-1}} \oplus V^{(2,2,1,\ldots,1,0)_{n-1}}\\
&\quad \oplus V^{(3,3,2,\ldots,2)_{n-1}} \oplus V^{(3,1,\ldots,1,0)_{n-1}},
\end{align*}
because the same equality with Schur functions holds. The second line corresponds to non spherical representations, so only the terms of the first line can contribute to $(\phi^{(2,1,\ldots,1)_{n-1}})^{2}$. Entirely similar calculations yield:\vspace{2mm}
\begin{itemize}
\item $\Gra(n,q,\R)$:
\begin{align*}\left(V^{(2,0,\ldots,0)_{\lfloor \frac{n}{2} \rfloor}}\right)^{\otimes 2}&=V^{(0,\ldots,0)_{\lfloor \frac{n}{2} \rfloor}}\oplus V^{(2,0,\ldots,0)_{\lfloor \frac{n}{2} \rfloor}}\oplus V^{(1,1,0\ldots,0)_{\lfloor \frac{n}{2} \rfloor}}\\
&\quad \oplus V^{(2,2,0,\ldots,0)_{\lfloor \frac{n}{2} \rfloor}}\oplus V^{(3,1,0,\ldots,0)_{\lfloor \frac{n}{2} \rfloor}}\oplus V^{(4,0,\ldots,0)_{\lfloor \frac{n}{2} \rfloor}}
\end{align*}
\item $\Gra(n,q,\Hq)$:
\begin{align*}
\left(V^{(1,1,0,\ldots,0)_{n}}\right)^{\otimes 2}&=V^{(0,\ldots,0)_{n}}\oplus V^{(1,1,0,\ldots,0)_{n}} \oplus V^{(1,1,1,1,0,\ldots,0)_{n}}\oplus V^{(2,2,0,\ldots,0)_{n}}\\
&\quad \oplus V^{(2,1,1,0,\ldots,0)_{n}}.
\end{align*}
Only the terms on the first line are spherical.\vspace{2mm}
\item $\SO(2n,\R)/\unit(n,\C)$:
\begin{align*}
\left(V^{(1,1,0,\ldots,0)_{n}}\right)^{\otimes 2}&= V^{(0,\ldots,0)_{n}}\oplus V^{(1,1,0,\ldots,0)_{n}} \oplus V^{(1,1,1,1,0,\ldots,0)_{n}} \oplus V^{(2,2,0,\ldots,0)_{n}}\\
&\quad \oplus V^{(2,0,\ldots,0)_{n}}\oplus V^{(2,1,1,0,\ldots,0)_{n}},
\end{align*}
again with non-spherical representations gathered on the second line.\vspace{2mm}
\item $\SU(n,\C)/\SO(n,\R)$: 
$$V^{(2,0,\ldots,0)_{n-1}}\otimes V^{(2,2,\ldots,2)_{n-1}}=V^{(0,\ldots,0)_{n-1}}  \oplus V^{(4,2,\ldots,2)_{n-1}}\oplus V^{(2,1,\ldots,1)_{n-1}},$$
and the last term is not a spherical representation. \vspace{2mm}
\item $\SU(2n,\C)/\unit\SP(n,\Hq)$: 
$$V^{(1,1,0,\ldots,0)_{2n-1}}\otimes V^{(1,\ldots,1,0)_{2n-1}}=V^{(0,\ldots,0)_{2n-1}} \oplus V^{(2,2,1,\ldots,1,0)_{2n-1}}\oplus V^{(2,1,\ldots,1)_{2n-1}},$$
and again the last term is not spherical.\vspace{2mm} 
\item $\unit\SP(n,\Hq)/\unit(n,\C)$:
\begin{align*}
\left(V^{(2,0,\ldots,0)_{n}}\right)^{\otimes 2}&=V^{(0,\ldots,0)_{n}}\oplus V^{(2,0,\ldots,0)_{n}}  \oplus V^{(2,2,\ldots,0)_{n}} \oplus V^{(4,0,\ldots,0)_{n}}\\
&\quad\oplus V^{(1,1,0,\ldots,0)_{n}}\oplus V^{(3,1,0,\ldots,0)_{n}}.
\end{align*}
The terms on the second line corresponds to non-spherical representations.\vspace{2mm}
\end{itemize}
As mentioned before, the coefficient of the constant function in $|\phi^{\lambda_{\min}}|^{2}$ is then always equal to $\frac{1}{D^{\lambda_{\min}}}$.
\end{proof}
\bigskip

For the spaces $\SU(n)/\SO(n)$ and $\SU(2n)/\unit\SP(n)$, the remaining coefficient $a$ can be found by evaluating the spherical functions at $e_{G}$. Thus,
\begin{align*}
&\SU(n)/\SO(n): \quad\left|\phi^{(2,0,\ldots,0)_{n-1}}\right|^{2}= \frac{2}{n^{2}+n}+\frac{n^{2}+n-2}{n^{2}+n}\,\phi^{(4,2,\ldots,2)_{n-1}};\\
&\SU(2n)/\unit\SP(n):\quad \left|\phi^{(1,1,0,\ldots,0)_{2n-1}}\right|^{2}= \frac{1}{2n^{2}-n}+ \frac{2n^{2}-n-1}{2n^{2}-n} \,\phi^{(2,2,1,\ldots,1,0)_{2n-1}}.
\end{align*}
But in the other cases, the values of the spherical functions appearing in the right-hand side of the formulas of Lemma \ref{abstractexpansion} are unfortunately not known \emph{a priori}, which makes finding the coefficients $a,b,c,\ldots$ quite difficult. However, since one only needs to compute $\esper_{t}[(\phi^{\lambda_{\min}})^{2}]$, and since $\phi^{\lambda_{\min}}$ is explicit in terms of matrix coefficients, one can use the following Lemma (\emph{cf.} \cite[Proposition 1.4]{Lev11}).
\begin{lemma}\label{expectationcoefficients}
Let $k\geq 1$ be any integer, and $(g_{t})_{t \in \R_{+}}$ be the Brownian motion on $\SO(n)$ or $\SU(n)$. The joint moments of order $k$ of the matrix coefficients of $g_{t}$ are given by 
\begin{equation}\esper[g_{t}^{\otimes k}]=\exp\left(t\,\frac{k\, \alpha_{\glie}}{2}\,(I_{n})^{\otimes k}+t \sum_{1\leq i< j \leq k} \eta_{i,j}(C_{\glie})\right)\label{levy}\end{equation}
where $\alpha_{\glie}$ is the coefficient introduced on page \pageref{casimircoefficient}; $C_{\glie}$ is the Casimir operator; and $\eta_{i,j}$ is the linear map from $\matr(n,\Bbbk)^{\otimes 2}$ to $\matr(n,\Bbbk)^{\otimes k}$ defined on simple tensors $X \otimes Y$ by
$$X \otimes Y \mapsto (I_{n})^{\otimes (i-1)}\otimes X \otimes (I_{n})^{\otimes (j-i-1)} \otimes Y \otimes (I_{n})^{\otimes k-j}.$$
In the complex case, one has also:
$$\esper[(g_{t})^{\otimes k}\otimes (\overline{g_{t}})^{\otimes l}]=\exp\left(t\,\frac{(k+l)\, \alpha_{\glie}}{2}\,(I_{n})^{\otimes k+l}+t \sum_{1\leq i< j \leq k+l} \widetilde\eta_{i,j}(C_{\glie})\right)\label{levymore},$$
with 
$$
\widetilde\eta_{i,j}(X\otimes Y)=\begin{cases}
\eta_{i,j}(X\otimes Y) &\text{if }i,j \in \lle 1,k\rre;\\
-\eta_{i,j}(X \otimes Y^{t}) &\text{if }i \in \lle 1,k\rre \text{ and }j \in \lle k+1,k+l\rre;\\
\eta_{i,j}(X^{t} \otimes Y^{t}) &\text{if }i,j \in \lle k+1,k+l\rre.
\end{cases}
$$
\end{lemma}

\begin{proof}
In the complex case, recall the stochastic differential equation satisfied by $g_{t}$, and therefore by $\overline{g_{t}}$:
$$dg_{t}=g_{t}\,dB_{t}+\frac{\alpha_{\glie}}{2}\,g_{t}\,dt\qquad;\qquad d\overline{g_{t}}=-\overline{g_{t}}\,dB_{t}^{t}+\frac{\alpha_{\glie}}{2}\,\overline{g_{t}}\,dt.$$
It\^{o}'s formula yields then a stochastic differential equation for $(g_{t})^{\otimes k}\otimes (\overline{g_{t}})^{\otimes l}$:
\begin{align*}
d(g^{\otimes k} \otimes \overline{g}^{\otimes l})_{t}&=(g_{t})^{\otimes k}\otimes (\overline{g_{t}})^{\otimes l}\left(\frac{(k+l)\,\alpha_{\glie}}{2}\,dt+\sum_{1\leq i<j\leq k}\widetilde\eta_{i,j}(dB_{t}\otimes dB_{t}) \right)\\
&\!\!\!\!\!\!\!\!+(g_{t})^{\otimes k}\otimes (\overline{g_{t}})^{\otimes l}\left(\sum_{i=1}^{k} (I_{n})^{\otimes i-1}\otimes dB_{t} \otimes (I_{n})^{\otimes k+l-i}-\sum_{i=k+1}^{k+l} (I_{n})^{\otimes i-1}\otimes dB_{t}^{t} \otimes (I_{n})^{\otimes k+l-i}\right).
\end{align*}
The quadratic variation of $B_{t}$ is given by the Casimir operator: $dB_{t}\otimes dB_{t}=C_{\glie}\,dt$. Taking expectations in the formula above leads now to a differential equation for $\esper[(g_{t})^{\otimes k}\otimes (\overline{g_{t}})^{\otimes l}]$, whose solution is the exponential of matrices in the statement of this lemma. The real case is the specialization $l=0$ of the previous discussion, though with a different Casimir operator. In the quaternionic case, one has to be more careful. In particular, since the quaternionic conjugate of $pq$ is $q^{\star} p^{\star}$ instead of $p^{\star}q^{\star}$, in the previous argument the SDE for $g_{t}^{\star}$ does not take  the same form. A way to overcome this problem is to use the doubling map \eqref{doublequaternion}. Thus, we write an equation for $\widetilde{g}_{t}$ instead of $g_{t}$:
$$\esper\!\left[(\widetilde{g}_{t})^{\otimes k}\right]=\exp\left(t\,\frac{k\, \alpha_{\glie}}{2}\,(I_{2n})^{\otimes k}+t \sum_{1\leq i< j \leq k} \eta_{i,j}(C_{\glie})\right),$$
where the Casimir is now considered as an element of $(\hendo(\C^{2n}))^{\otimes 2}$. As we shall see later, joint moments of the entries of $g$ and $g^{\star}$ are combinations of the joint moments of the entries of $\widetilde{g}$, so the previous formula will prove sufficient to solve our problem in the quaternionic case. 
\end{proof}
\bigskip

It turns out that in each case important for our computations, the matrix $\sum_{1\leq i<j \leq k} \widetilde\eta_{i,j}(C_{\glie})$ can be explicitly diagonalized, with a basis of eigenvectors that is quite tractable (to be honest, with the help of a computer). In the following, we describe the eigenvalues and eigenvectors of these matrices, and leave the reader check that they are indeed eigenvalues and eigenvectors: this is each time an immediate computation with elementary matrices, though quite tedious if $k=4$ or $k+l=4$. For simplification, we write $e[i_{1},i_{2},\ldots,i_{r}]=e_{i_{1}}\otimes e_{i_{2}}\otimes \cdots \otimes e_{i_{r}}$.\bigskip

\subsubsection{Quotients of orthogonal groups} For special orthogonal groups, set $\frac{1}{n}M_{n,k}=\sum_{1\leq i< j \leq k} \eta_{i,j}(C_{\solie(n)})$, to be considered as an element of $\mathrm{End}((\R^{n})^{\otimes k})$. If $k=2$, then the eigenvalues and eigenvectors of $M_{n,2}=\sum_{1\leq i<j\leq n}(E_{ij}-E_{ji})^{\otimes 2}$ are:\vspace{1mm}
$$
\begin{tabular}{|c|c|c|}
\hline\vspace{-2.5mm}&&\\
eigenvalue & multiplicity & eigenvectors \\
\vspace{-2.5mm}&&\\
\hline \vspace{-2.5mm}&& \\ 
$n-1$ & $1$ &  $\sum_{i=1}^{n}e[i,i]$ \\
\vspace{-2.5mm}&&\\
\hline \vspace{-2.5mm}&& \\ 
$1$ & $\frac{n(n-1)}{2}$ & $e[i,j] - e[j,i],\,\, i<j $ \\
\vspace{-2.5mm}&&\\
\hline \vspace{-2.5mm}&& \\ 
$-1$ & $\frac{(n+2)(n-1)}{2}$ & $e[i,j] + e[j,i],\,\, i<j$ \\
\vspace{-2mm}&&\\
& & $e[i,i]-e[i+1,i+1],\,\, i \leq n-1$ \\
&&\\
\hline
\end{tabular}
$$\vspace{1mm}

\noindent This allows to compute $\exp(-\frac{(n-1)t}{n})\,\exp(\frac{t}{n}M_{n,2})$, which is the right-hand side of Formula \eqref{levy} in the case of $\SO(n,\R)$ and for $k=2$. One obtains:
\begin{align*}
\esper[(g_{ii})^{2}]&=\frac{1}{n}+\left(1-\frac{1}{n}\right)\E^{-t}\quad;\quad\esper[(g_{ij})^{2}]=\frac{1}{n}\left(1-\E^{-t}\right)\quad;\\
\esper[g_{ii}g_{jj}]&=\frac{1}{2}\left(\E^{-t}+\E^{-\frac{n-2}{n}t}\right)\quad;\quad \esper[g_{ij}g_{ji}]=\frac{1}{2}\left(\E^{-t}-\E^{-\frac{n-2}{n}t}\right)\quad;
\end{align*}
and all the other mixed moments vanish (\emph{e.g.}, $\esper[g_{ii}g_{ij}]$ or $\esper[g_{ij}g_{kl}]$). Now, if $k=4$, then the eigenvalues and eigenvectors of $M_{n,4}$ are:\vspace{1mm}
$$
\begin{tabular}{|c|c|c|}
\hline\vspace{-2.5mm}&&\\
eigenvalue & multiplicity & eigenvectors (not exhaustive, some repetitions)\\
\vspace{-2.5mm}&&\\
\hline \vspace{-2.5mm}&& \\ 
$2n-2$ & $3$ & $\sum_{k=1}^{n}\sum_{l=1}^{n}e[k,k,l,l], \,\,\star$ \\
\vspace{-2.5mm}&&\\
\hline \vspace{-2.5mm}&& \\ 
$n$ & $3n(n-1)$ & $\sum_{k=1}^{n}e[i,j,k,k]-e[j,i,k,k],\,\, i<j, \,\,\star $ \\
\vspace{-2.5mm}&&\\
\hline \vspace{-2.5mm}&& \\ 
$n-2$ & $3(n+2)(n-1)$ & $\sum_{k=1}^{n}e[i,j,k,k]+e[j,i,k,k],\,\, i<j, \,\,\star$ \\
\vspace{-2mm}&&\\
 &  & $\sum_{k=1}^{n}e[i,i,k,k]-e[i+1,i+1,k,k],\,\, i \leq n-1, \,\,\star$ \\
\vspace{-2.5mm}&&\\
\hline \vspace{-2.5mm}&& \\ 
$6$ & $\frac{n(n-1)(n-2)(n-3)}{24}$ & $\sum_{\sigma \in \sym_{4}} \eps(\sigma)\,e[i,j,k,l]^{\sigma},\,\,i<j<k<l$ \\
\vspace{-2.5mm}&&\\
\hline \vspace{-2.5mm}&& \\ 
$2$&$\frac{3n(n+2)(n-1)(n-3)}{8}$&$D_{1}^{\eta}(i,j,k,l),\,\,D_{2}^{\eta}(i,j,k,l),\,\,D_{3}^{\eta}(i,j,k,l),\,\,i \neq j \neq k \neq l$\\
\vspace{-2.5mm}&&\\
\hline \vspace{-2.5mm}&& \\ 
$0$&$\frac{n(n+1)(n+2)(n-3)}{6}$&$S_{1}(i,j,k,l),\,\,S_{2}(i,j,k,l),\,\,i\neq j \neq k \neq l$\\
\vspace{-2mm}&&\\
&&$K_{1}(i,j,k,l),\,\,K_{2}(i,j,k,l),\,\,i\neq j \neq k \neq l$\\
\vspace{-2.5mm}&&\\
\hline \vspace{-2.5mm}&& \\ 
$-2$&$\frac{3(n-1)(n-2)(n+1)(n+4)}{8}$&$\left(\substack{e[i,j]^{\otimes 2}+e[j,k]^{\otimes 2}+e[k,i]^{\otimes 2}\\-e[j,i]^{\otimes 2}-e[k,j]^{\otimes 2}-e[i,k]^{\otimes 2}}\right),\,\,i\neq j\neq k,\,\,\star$\\
\vspace{-2mm}&&\\
&&$D_{1}^{\theta}(i,j,k,l),\,\,D_{2}^{\theta}(i,j,k,l),\,\,D_{3}^{\theta}(i,j,k,l),\,\,i \neq j \neq k \neq l$\\
\vspace{-2.5mm}&&\\
\hline \vspace{-2.5mm}&& \\ 
$-6$ & $\frac{n(n-1)(n+1)(n+6)}{24}$ & $\sum_{\sigma \in \sym_{4}} e[i,j,k,l]^{\sigma},\,\,i <j<k<l$ \\
\vspace{-2mm}&&\\
& & $e[i,i,i,i]+e[j,j,j,j]-\sum_{\sigma \in \sym_{4}}' e[i,i,j,j]^{\sigma},\,\,i<j$ \\
&&\\
\hline
\end{tabular}
$$
\vspace{1mm}

The star $\star$ means that the eigenvectors of a basis are listed up to action of $\sym_{4}$; and the symbols $\sum_{\sigma\in \sym_{4}}'$ mean that we take the sum of all \emph{distinct} permutations of the tensors. For the eigenvectors associated to the value $2$, denote $\mathfrak{D}_{4,(1)}=\langle (1,3,2,4),(1,2)\rangle$, $\mathfrak{D}_{4,(2)}=\langle (1,2,3,4),(1,3)\rangle$ and $\mathfrak{D}_{4,(3)}=\langle(1,2,4,3), (1,4)\rangle$ the three dihedral groups of order $4$ (hence cardinality $8$) that can be found inside $\sym_{4}$. Each dihedral group of order $4$ has for presentation
$$\mathfrak{D}_{4}=\langle r,s \,\,| \,\, r^{4}=s^{2}=(rs)^{2}=1\rangle,$$
so the parity $\eta(\sigma)$ of the number of occurrences of $s$ in a reduced writing of $\sigma \in \mathfrak{D}_{4}$ is well-defined, and provides a morphism $\eta : \mathfrak{D}_{4,(v)} \to \{ \pm 1\}$ for $v=1,2,3$. Then, it can be checked that for every $i \neq j \neq k \neq l$ and any $v$, 
$$D_{v}^{\eta}(i,j,k,l)=\sum_{\sigma \in \mathfrak{D}_{4,(v)}} \eta(\sigma) \,e[i,j,k,l]^{\sigma}$$
is in $V_{2}$. The eigenvectors $D_{1}^{\theta}(i,j,k,l)$, $D_{2}^{\theta}(i,j,k,l)$ and $D_{3}^{\theta}(i,j,k,l)$ associated to the eigenvalue $-2$ are defined exactly the same way, but with the morphism $\theta : \mathfrak{D}_{4,(v)} \to\{ \pm 1\}$ associated to the parity of the number of occurrences of $r$ in a reduced decomposition of $\sigma \in \mathfrak{D}_{4}$ (again it is well defined):
$$D_{v}^{\theta}(i,j,k,l)=\sum_{\theta \in \mathfrak{D}_{4,(v)}} \theta(\sigma) \,e[i,j,k,l]^{\sigma}.$$
For the eigenvectors associated to the value $0$, $S_{1}(i,j,k,l)$ is defined by
\begin{align*}
&\,\,e[i,i,k,k]+e[k,k,i,i]+e[j,j,l,l]+e[l,l,j,j]-e[i,i,l,l]-e[l,l,i,i]-e[j,j,k,k]-e[k,k,j,j]\\
&-e[i,k,k,i]-e[k,i,i,k]-e[j,l,l,j]-e[l,j,j,l]+e[i,l,l,i]+e[l,i,i,l]+e[j,k,k,j]+e[k,j,j,k]
\end{align*}
and $S_{2}(i,j,k,l)$ is obtained by replacing each term $a \otimes b \otimes b \otimes a$ by $a\otimes b \otimes a \otimes b$ in the previous formula. On the other hand, if $\mathfrak{K}_{4}=\{\id,(1,2)(3,4),(1,3)(2,4),(1,4)(2,3)\}$ denotes the Klein group, then $K_{1}(i,j,k,l)$ and $K_{2}(i,j,k,l)$ are defined as
\begin{align*}K_{1}(i,j,k,l)&=\sum_{\sigma \in \mathfrak{K}_{4}} e[i,j,k,l]^{\sigma}\,\,-\!\!\sum_{\sigma \in (1,2,3)\mathfrak{K}_{4}}e[i,j,k,l]^{\sigma};\\
 K_{2}(i,j,k,l)&=\sum_{\sigma \in \mathfrak{K}_{4}} e[i,j,k,l]^{\sigma}\,\,-\!\!\sum_{\sigma \in (1,3,2)\mathfrak{K}_{4}}e[i,j,k,l]^{\sigma}.
 \end{align*}\vspace{3mm}

That said, the deduction of the mixed moments of order $4$ of the coefficients of $g$ goes as follows. One notices that 
\begin{align*}(n+2)\sum_{i=1}^{n}e_{i}^{\otimes 4}&=\left(\sum_{k,l=1}^{n} e[k,k,l,l]+e[k,l,k,l]+e[k,l,l,k]\right)+\sum_{i<j} \left(e_{i}^{\otimes 4}+e_{j}^{\otimes 4}-\sum_{\sigma \in \sym_{4}}' e[i,i,j,j]^{\sigma}\right)
\end{align*}
with the first sum in the eigenspace $V_{2n-2}$ and the second sum in the eigenspace $V_{-6}$. On the other hand, for any $i \neq j$, 
\begin{align*}
&(n+4)(e_{i}^{\otimes 4}-e_{j}^{\otimes 4})=6\,e[i,i,i,i] + \sum_{k \neq i,j} \sum_{\sigma \in \sym_{4}}'e[i,i,k,k]^{\sigma} - \sum_{k \neq i,j} \sum_{\sigma \in \sym_{4}}'e[j,j,k,k]^{\sigma} - 6\,e[j,j,j,j]  \\
&+\sum_{k \neq i,j}\left(e[i,i,i,i] + e[k,k,k,k] - \sum_{\sigma \in \sym_{4}}' e[i,i,k,k]^{\sigma}\right)-\left(e[j,j,j,j] + e[k,k,k,k] - \sum_{\sigma \in \sym_{4}}' e[j,j,k,k]^{\sigma}\right),
\end{align*}
with the first line in $V_{n-2}$ and the second in $V_{-6}$. Since 
$e_{i}^{\otimes 4}=\frac{1}{n}\sum_{j=1}^{n} e_{j}^{\otimes 4}+ \frac{1}{n}\sum_{j=1}^{n}(e_{i}^{\otimes 4}-e_{j}^{\otimes 4}),$
one concludes that
\begin{align*}
e_{i}^{\otimes 4}&=\frac{1}{n(n+2)}\sum_{k,l=1}^{n}e[k,k,l,l]+e[k,l,k,l]+e[k,l,l,k]\\
&\quad+\frac{1}{n(n+4)} \sum_{\sigma \in \sym_{4}}' \sum_{k,l=1}^{n}  (e[i,i,k,k]-e[l,l,k,k])^{\sigma}\end{align*}
\begin{align*}
&\quad+\frac{n+1}{(n+2)(n+4)}\sum_{k \neq i}\left(e[i,i,i,i]+e[k,k,k,k]-\sum_{\sigma \in \sym_{4}}'e[i,i,k,k]^{\sigma}\right)\\
&\quad-\frac{1}{(n+2)(n+4)}\sum_{(k<l) \neq i}\left(e[k,k,k,k]+e[l,l,l,l]-\sum_{\sigma \in \sym_{4}}'e[k,k,l,l]^{\sigma}\right),
\end{align*}
each line being in a different eigenspace: $V_{2n-2}$, $V_{n-2}$, $V_{-6}$ and $V_{-6}$. The technique is now the following: to compute $\esper[g_{ij_{1}}g_{ij_{2}}g_{ij_{3}}g_{ij_{4}}]$, one counts the number of occurrences of $e[j_{1},j_{2},j_{3},j_{4}]$ in each term of the previous expansion. This leads to:
\begin{align*}
\esper[(g_{ii})^{4}]&=\frac{3}{n(n+2)}+\frac{6(n-1)}{n(n+4)}\,\E^{-t}+\frac{(n+1)(n-1)}{(n+2)(n+4)}\,\E^{-\frac{2n+4}{n}t};\\
\esper[(g_{ij})^{4}]&=\frac{3}{n(n+2)}-\frac{6}{n(n+4)}\,\E^{-t}+\frac{3}{(n+2)(n+4)}\,\E^{-\frac{2n+4}{n}t};\\
\esper[(g_{ii})^{2}(g_{ij})^{2}]&=\frac{1}{n(n+2)}+\frac{(n-2)}{n(n+4)}\,\E^{-t}-\frac{(n+1)}{(n+2)(n+4)}\,\E^{-\frac{2n+4}{n}t};\\
\esper[(g_{ij})^{2}(g_{ik})^{2}]&=\frac{1}{n(n+2)}-\frac{2}{n(n+4)}\,\E^{-t}+\frac{1}{(n+2)(n+4)}\,\E^{-\frac{2n+4}{n}t};
\end{align*}
and one sees also that the other expectations $\esper[g_{ij}g_{ik}g_{il}g_{im}]$ vanish (\emph{e.g.}, $\esper[g_{ij}g_{ik}(g_{il})^{2}]$ with $i \neq j \neq k \neq l$). Similar manipulations yield the decomposition in eigenvectors of $e_{i}^{\otimes 2}\otimes e_{j}^{\otimes 2}$:
\begin{align*}
&\frac{n+1}{n(n-1)(n+2)}\sum_{k,l=1}^{n}e[k,k,l,l]-\frac{1}{(n-1)n(n+2)}\sum_{k,l=1}^{n}e[k,l,k,l]+e[k,l,l,k]\\
&\text{- - - - - - - - - - - - - - - - - - - - - - - - - - - - - - - - - - - - - - - - - - - - - - - - - - - - - - - - - - - - - - -}\\
&+\frac{1}{n(n-2)}\sum_{k \neq i,j}\left(\sum_{l=1}^{n}e[i,i,l,l]-e[k,k,l,l]+\sum_{l=1}^{n}e[l,l,j,j]-e[l,l,k,k]\right)\\
&-\frac{1}{n(n-2)(n+4)}\sum_{k \neq i,j}\sum_{\sigma \in \sym_{4}}' \left(\sum_{l=1}^{n}(e[i,i,l,l]-e[k,k,l,l])^{\sigma}+\sum_{l=1}^{n}(e[l,l,j,j]-e[l,l,k,k])^{\sigma} \right)\\
&\text{- - - - - - - - - - - - - - - - - - - - - - - - - - - - - - - - - - - - - - - - - - - - - - - - - - - - - - - - - - - - - - -}\\
&+\frac{1}{6(n-1)(n-2)}\sum_{(k<l)\neq i,j}S_{1}(i,k,j,l)+S_{1}(i,l,j,k)+S_{2}(i,k,j,l)+S_{2}(i,l,j,k)\\
&\text{- - - - - - - - - - - - - - - - - - - - - - - - - - - - - - - - - - - - - - - - - - - - - - - - - - - - - - - - - - - - - - -}\\
&+\frac{1}{2n}\sum_{k\neq i,j}e[i,i,j,j]+e[j,j,k,k]+e[k,k,i,i]-e[j,j,i,i]-e[k,k,j,j]-e[i,i,k,k]\\
&\text{- - - - - - - - - - - - - - - - - - - - - - - - - - - - - - - - - - - - - - - - - - - - - - - - - - - - - - - - - - - - - - -}\\
&+\frac{1}{6(n+4)}\left(\sum_{k \neq i}\left(e_{i}^{\otimes 4}+e_{k}^{\otimes 4}-\sum_{\sigma \in \sym_{4}}'e[i,i,k,k]^{\sigma}\right)+\sum_{k\neq j}\left(e_{j}^{\otimes 4}+e_{k}^{\otimes 4}-\sum_{\sigma \in \sym_{4}}'e[j,j,k,k]^{\sigma}\right)\right)\\
&-\frac{1}{6}\left(e_{i}^{\otimes 4}+e_{j}^{\otimes 4}-\sum_{\sigma \in \sym_{4}}'e[i,i,j,j]^{\sigma}\right)-\frac{1}{3(n+2)(n+4)}\sum_{(k<l)} \left(e_{k}^{\otimes 4}+e_{l}^{\otimes 4}-\sum_{\sigma \in \sym_{4}}'e[k,k,l,l]^{\sigma}\right),
\end{align*}
where the eigenspaces associated to each part are $V_{2n-2}$, $V_{n-2}$, $V_{0}$, $V_{-2}$ and $V_{-6}$. As a consequence, 
\begin{align*}
\esper[(g_{ii})^{2}(g_{jj})^{2}]&=\frac{n+1}{(n-1)n(n+2)}+\frac{2(n+3)}{n(n+4)}\,\E^{-t}+\frac{n-3}{3(n-1)}\,\E^{-\frac{2n-2}{n}t}+\frac{n-2}{2n}\,\E^{-2t}\\
&\quad+\frac{n^{2}+4n+6}{6(n+2)(n+4)}\,\E^{-\frac{2n+4}{n}t};\\
\esper[(g_{ij})^{2}(g_{ji})^{2}]&=\frac{n+1}{n(n-1)(n+2)}-\frac{2}{n(n+4)}\,\E^{-t}+\frac{n-3}{3(n-1)}\,\E^{-\frac{2n-2}{n}t}-\frac{n-2}{2n}\,\E^{-2t}\\
&\quad+\frac{n^{2}+4n+6}{6(n+2)(n+4)}\,\E^{-\frac{2n+4}{n}t};\\
\esper[(g_{ii})^{2}(g_{jk})^{2}]&=\frac{n+1}{n(n-1)(n+2)}+\frac{n^{2}-8}{n(n-2)(n+4)}\,\E^{-t}-\frac{n-3}{3(n-1)(n-2)}\,\E^{-\frac{2n-2}{n}t}\\
&\quad-\frac{1}{2n}\,\E^{-2t}-\frac{n}{6(n+2)(n+4)}\,\E^{-\frac{2n+4}{n}t};\\
\esper[(g_{ij})^{2}(g_{jk})^{2}]&=\frac{n+1}{n(n-1)(n+2)}-\frac{2}{(n-2)(n+4)}\,\E^{-t}-\frac{n-3}{3(n-1)(n-2)}\,\E^{-\frac{2n-2}{n}t}\\
&\quad+\frac{1}{2n}\,\E^{-2t}-\frac{n}{6(n+2)(n+4)}\,\E^{-\frac{2n+4}{n}t};\\
\esper[(g_{ij})^{2}(g_{kl})^{2}]&=\frac{n+1}{n(n-1)(n+2)}-\frac{2(n+2)}{n(n-2)(n+4)}\,\E^{-t}+\frac{2}{3(n-1)(n-2)}\,\E^{-\frac{2n-2}{n}t}\\
&\quad+\frac{1}{3(n+2)(n+4)}\,\E^{-\frac{2n+4}{n}t};\\
\esper[g_{ii}g_{ij}g_{jj}g_{ji}]&=-\frac{1}{n(n-1)(n+2)}-\frac{2}{n(n+4)}\,\E^{-t}-\frac{n-3}{6(n-1)}\,\E^{-\frac{2n-2}{n}t}\\
&\quad+\frac{n^{2}+4n+6}{6(n+2)(n+4)}\,\E^{-\frac{2n+4}{n}t};\\
\esper[g_{ik}g_{il}g_{jk}g_{jl}]&=-\frac{1}{n(n-1)(n+2)}+\frac{4}{n(n-2)(n+4)}\,\E^{-t}-\frac{1}{3(n-1)(n-2)}\,\E^{-\frac{2n-2}{n}}\\
&\quad+\frac{1}{3(n+2)(n+4)}\,\E^{-\frac{2n+4}{n}t}.
\end{align*}
\vspace{2mm}

Finally, the elementary tensor $e_{i} \otimes e_{j} \otimes e_{k} \otimes e_{l}$ with $i \neq j \neq k \neq l$ can be expanded as a combination of eigenvectors in $V_{6}$, $V_{2}$, $V_{0}$, $V_{-2}$ and $V_{-6}$. This expansion is related to a remarkable identity in the group algebra $\C\sym_{4}$, which can be considered as a relation of orthogonality of characters, but that only involves one-dimensional representations. Denote $$D_{1}^{\eta}=\sum_{\sigma \in \mathfrak{D}_{4,(1)}}\eta(\sigma)\,\sigma,$$ and similarly for $D_{2}^{\eta}$, $D_{3}^{\eta}$, $D_{1}^{\theta}$, $D_{2}^{\theta}$ and $D_{3}^{\theta}$. We also introduce $I = \sum_{\sigma \in \sym_{4}} \sigma$, $E = \sum_{\sigma \in \sym_{4}} \eps(\sigma)\,\sigma$, and
$$K_{1}=\sum_{\sigma \in \mathfrak{K}_{4}}\sigma\,\,\,\,-\!\!\!\!\sum_{\sigma \in (1,2,3)\mathfrak{K}_{4}}\!\!\!\!\sigma \qquad;\qquad K_{2}=\sum_{\sigma \in \mathfrak{K}_{4}}\sigma\,\,\,\,-\!\!\!\!\sum_{\sigma \in (1,3,2)\mathfrak{K}_{4}}\!\!\!\!\sigma.$$
As explained before, all these sums correspond to eigenvectors in $V_{6}$, $V_{2}$, $V_{0}$, $V_{-2}$ and $V_{-6}$. Then,
$$\mathrm{id}_{\lle 1,4\rre}=\frac{1}{24}\,I+\frac{1}{8}\,(D_{1}^{\eta}+D_{2}^{\eta}+D_{3}^{\eta})+\frac{1}{12}(K_{1}+K_{2})+\frac{1}{8}\,(D_{1}^{\theta}+D_{2}^{\theta}+D_{3}^{\theta})+\frac{1}{24}\,E.\label{crazys4}$$
\comment{Indeed, $\frac{1}{24}\,I+\frac{1}{24}\,E$ is one twelfth of the sum of all even permutations in the alternate group $\mathfrak{A}_{4}$. By adding $\frac{1}{12}(K_{1}+K_{2})$ to this quantity, one removes all the permutations that are in $\mathfrak{A}_{4}$ and not in $\mathfrak{K}_{4}$, so
$$\frac{1}{24}\,I+\frac{1}{12}(K_{1}+K_{2})+\frac{1}{24}\,E=\frac{1}{4}\sum_{\sigma \in \mathfrak{K}_{4}}\sigma.$$
On the other hand, 
$$\frac{1}{8}\left(D_{1}^{\eta}+D_{1}^{\theta}\right)=\frac{1}{4}\,\id_{\lle 1,4\rre}+\frac{1}{4}\,(1,2)(3,4)-\frac{1}{4}\,(1,3)(2,4)-\frac{1}{4}\,(1,4)(2,3)$$
and similarly for the two other dihedral groups. Thus, the whole sum on the right-hand side of \eqref{crazys4} is indeed $\id_{\lle 1,4\rre}$. }
As a consequence, $e[i,j,k,l]$ is equal to
\begin{align*}&\frac{1}{24}\sum_{\sigma \in \sym_{4}}\eps(\sigma)\, e[i,j,k,l]^{\sigma}+\frac{1}{8}(D_{1}^{\eta}(i,j,k,l)+D_{2}^{\eta}(i,j,k,l)+D_{3}^{\eta}(i,j,k,l))+\frac{1}{12}(K_{1}(i,j,k,l)+K_{2}(i,j,k,l))\\
&+\frac{1}{8}(D_{1}^{\theta}(i,j,k,l)+D_{2}^{\theta}(i,j,k,l)+D_{3}^{\theta}(i,j,k,l))+\frac{1}{24}\sum_{\sigma \in \sym_{4}} e[i,j,k,l]^{\sigma}
\end{align*}
with each term respectively in  $V_{6}$, $V_{2}$, $V_{0}$, $V_{-2}$ and $V_{-6}$. This leads to
\begin{align*}
\esper[g_{ii}g_{jj}g_{kk}g_{ll}]&=\frac{1}{24}\,\E^{-\frac{2n-8}{n}t}+\frac{3}{8}\,\E^{-\frac{2n-4}{n}t}+\frac{1}{6}\,\E^{-\frac{2n-2}{n}t}+\frac{3}{8}\,\E^{-2t}+\frac{1}{24}\,\E^{-\frac{2n+4}{n}t};\\
\esper[g_{ij}g_{jk}g_{kl}g_{li}]&=-\frac{1}{24}\,\E^{-\frac{2n-8}{n}t}+\frac{1}{8}\,\E^{-\frac{2n-4}{n}t}-\frac{1}{8}\,\E^{-2t}+\frac{1}{24}\,\E^{-\frac{2n+4}{n}t};\\
\esper[g_{ii}g_{jj}g_{kl}g_{lk}]&=-\frac{1}{24}\,\E^{-\frac{2n-8}{n}t}-\frac{1}{8}\,\E^{-\frac{2n-4}{n}t} +\frac{1}{8}\,\E^{-2t}+\frac{1}{24}\,\E^{-\frac{2n+4}{n}t};\\
\esper[g_{ij}g_{ji}g_{kl}g_{lk}]&=\frac{1}{24}\,\E^{-\frac{2n-8}{n}t} -\frac{1}{8}\,\E^{-\frac{2n-4}{n}t}+\frac{1}{6}\,\E^{-\frac{2n-2}{n}t}-\frac{1}{8}\,\E^{-2t}+\frac{1}{24}\,\E^{-\frac{2n+4}{n}t};
\end{align*}
and we are done with the computations in the case of special orthogonal groups.\bigskip

\comment{\begin{remark}
It would be very nice to find the analogue of Equation \eqref{crazys4} for symmetric groups of larger order, in connection with the diagonalization of $M_{n,2k}$. An interesting feature of this identity is that it involves only generating functions of one-dimensional characters of subgroups $H \subset \sym_{4}$:
$$\varSigma(H,\chi)=\sum_{\sigma \in H \subset \sym_{4}}\chi(\sigma)\,\sigma.$$
Indeed, this is obvious for most of the terms, and notice that $K_{1}+K_{2}=\varSigma(\mathfrak{A}_{4},\chi)+\varSigma(\mathfrak{A}_{4},\chi^{2})$, where $\chi$ is the quotient map $\mathfrak{A}_{4} \to \mathfrak{A}_{4} / \mathfrak{K}_{4} \simeq \Z/3\Z = \{1,\E^{2\I \pi/3},\E^{-2\I \pi/3}\}$. This combinatorial property of the reduction of $M_{n,2k}$ seems profound and quite mysterious for the moment.
\end{remark}}

\begin{proposition}
For the real Grassmannian varieties $\Gra(n,q,\R)$ and the spaces $\SO(2n)/\unit(n)$, the coefficients of Lemma \ref{abstractexpansion} are:
\begin{align*}
&\Gra(n,q,\R): \quad \frac{2}{n^{2}+n-2}+\frac{2n^{2}}{3}\left(\frac{1}{(n-1)(n-2)}-\frac{1}{pq(n-2)}\right)\phi^{(2,2,0,\ldots,0)_{\lfloor \frac{n}{2} \rfloor}}\\
&\qquad\qquad\qquad +\frac{\frac{4n^{2}}{pq}-16}{(n-2)(n+4)}\,\phi^{(2,0,\ldots,0)_{\lfloor \frac{n}{2} \rfloor}}+\frac{n^{2}}{3}\left(\frac{1}{(n + 2)(n + 4)}+\frac{2}{pq(n + 4)}\right) \phi^{(4,0,\ldots,0)_{\lfloor \frac{n}{2} \rfloor}};\\
&\SO(2n)/\unit(n):\quad \frac{1}{2n^{2}-n}+\frac{n-1}{3n}\,\phi^{(1^{4},0,\ldots,0)_{n}}+\frac{4(n^{2}-1)}{(3n)(2n-1)}\,\phi^{(2,2,0,\ldots,0)_{n}}.
\end{align*}
\end{proposition}

\begin{proof}
One expands the square of the sum given by Proposition \ref{coeffsphericalfunction}, and one gathers the joint moments of the coefficients according to the possible identities between the indices. For real Grassmannians, $\left(\phi^{(2,0,\ldots,0)_{\lfloor \frac{n}{2} \rfloor}}\right)^{2}$ has for expansion:\label{expansquaregrass}
\begin{align*}&\left(\frac{1}{p}\sum_{i,j\leq p} (g_{ij})^{2}+\frac{1}{q}\sum_{i,j>p} (g_{ij})^{2}-1\right)^{2}=\left(\frac{1}{p}\sum_{i,j\leq p} (g_{ij})^{2}+\frac{1}{q}\sum_{i,j>p} (g_{ij})^{2}\right)^{2}-2\,\phi^{(2,0,\ldots,0)_{\lfloor \frac{n}{2} \rfloor}}-1\\
&=\left(\frac{n}{pq}\right)T[(g_{11})^{4}]+\left(4-\frac{n}{pq}\right)T[(g_{11}g_{22})^{2}]+\left(2-\frac{n}{pq}\right)\left(4\,T[(g_{11}g_{12})^{2}]+T[(g_{12})^{4}]+T[(g_{12}g_{21})^{2}]\right)\\
&\,\,\,+\left(4n-16+\frac{4n}{pq}\right)T[(g_{11}g_{23})^{2}]+\left(2n-12+\frac{4n}{pq}\right)\left(T[(g_{12}g_{13})^{2}]+T[(g_{12}g_{23})^{2}]\right)\\
&\,\,\,+\left(n^{2}-8n+24-\frac{6n}{pq}\right)T[(g_{12}g_{34})^{2}]-2\,\phi^{(2,0,\ldots,0)_{\lfloor \frac{n}{2} \rfloor}}-1.
\end{align*}
where by $T[(g_{11})^{4}]$ we mean a linear combination of products $(g_{ii})^{4}$, whose expectation is therefore $\esper[(g_{11})^{4}]$; by $T[(g_{11}g_{22})^{2}]$ we mean a linear combination of products $(g_{ii}\,g_{jj})^{2}$ whose expectation will be $\esper[(g_{11}g_{22})^{2}]$, \emph{etc.} Thus, the expectation of $\left(\phi^{(2,0,\ldots,0)_{\lfloor \frac{n}{2} \rfloor}}\right)^{2}$ is
\begin{align*}
&\frac{2}{n^{2}+n-2}+\frac{\frac{4n^{2}}{pq}-16}{(n-2)(n+4)}\,\E^{-t}+\frac{2n^{2}}{3}\left(\frac{1}{(n-1)(n-2)}-\frac{1}{pq(n-2)}\right)\,\E^{-\frac{2n-2}{n}t}\\
&+\frac{n^{2}}{3}\left(\frac{1}{(n + 2)(n + 4)}+\frac{2}{pq(n + 4)}\right)\,\E^{-\frac{2n+4}{n}t},
\end{align*}
and by identifying the Casimir coefficients of the spherical functions, one deduces from this the expansion of the square of the discriminating zonal function in zonal functions.
\bigskip

For the spaces $\SO(2n)/\unit(n)$, one computes again the square of the homogeneous polynomial of degree $2$ given by Proposition \ref{coeffsphericalfunction}. Thus, $\frac{1}{n^{2}}(\sum_{i,j=1}^{n}g_{2i,2j}\,g_{2i-1,2j-1}-g_{2i,2j-1}\,g_{2i-1,2j})^{2}$ equals
\begin{align*}
&\!\!\!\frac{1}{n}\left(T[(g_{11}g_{22})^{2}]+T[(g_{12}g_{21})^{2}]-2\,T[g_{11}g_{12}g_{22}g_{21}]\right)\\   
&+ \frac{n-1}{n} \left(2\,T[(g_{12}g_{34})^{2}]-2\,T[g_{13}g_{14}g_{23}g_{24}]\right)\\ 
&+ \frac{n-1}{n} \left(2\,T[g_{12}g_{21}g_{34}g_{43}]-2\,T[g_{12}g_{23}g_{34}g_{41}]\right)\\ 
&+ \frac{n-1}{n} \left(T[g_{11}g_{22}g_{33}g_{44}]+T[g_{12}g_{21}g_{34}g_{43}]-2\,T[g_{11}g_{22}g_{34}g_{43}]\right)\\
&+ \text{remainder},
\end{align*}
with the same notations as before, and where the remainder is a combination of products of coefficients whose expectation vanish under Brownian (and Haar) measures. More precisely, terms with a certain symmetry cancel with one another when taking the expectation: for instance,
\begin{equation} (g_{2i,2j}\,g_{2i-1,2j-1}-g_{2i,2j-1}\,g_{2i-1,2j})\times (g_{2k,2l}\,g_{2k-1,2l-1}-g_{2k,2l-1}\,g_{2k-1,2l}) \label{fourstrangers}\end{equation}
with $i \neq j \neq k \neq l$ is equal to $a + b - c - d$ , where $a,b,c,d$ are products of type $g_{ij}g_{kl}g_{mn}g_{op}$, and have therefore the same expectation. Consequently, every product of type \eqref{fourstrangers} will not contribute to the expectation of $(\phi^{(1,1,0,\ldots,0)_{\lfloor \frac{n}{2}\rfloor}})^{2}$. The following sets of indices have the same property of ``self-cancellation'':
\begin{align*}
&(i,i,i,j)\,\,;\,\,(i,i,j,i)\,\,;\,\,(i,j,i,i)\,\,;\,\,(j,i,i,i)\,\,;\,\,(i,i,j,k)\,\,;\,\,(j, k ,i ,i )\,\,;\\
&(i ,j, i, k) \,\,;\,\, (j ,i ,k ,i )\,\,;\,\, (i ,j ,k ,i) \,\,;\,\, (j ,i ,i ,k) \,\,;\,\, (i ,j, k, l)\,\,;\,\,(i,j,k,l)\,\,;\end{align*}
so it suffices to consider products with sets of indices $(i,i,i,i)$, $(i,j,i,j)$, $(i,j,j,i)$ or $(i,i,j,j)$ --- these are the four lines of the previous expansion.
Using the formulas given before for the joint moments of the entries (beware that one has to use them with the parameter $2n$), we obtain:
$$\esper[(\phi^{(1,1,0,\ldots,0)_{n}})^{2}]=\frac{1}{n(2n-1)}+\frac{n-1}{3n}\,\E^{-\frac{2n-4}{n}t}+\frac{4(n^{2}-1)}{(3n)(2n-1)}\,\E^{-\frac{2n-1}{n}t}$$
and it suffices then to identify the coefficients of the negative exponentials.
\end{proof}
\bigskip

\subsubsection{Quotients of unitary groups} For special unitary groups, set $\frac{1}{n^{2}}M_{n,k,l}=\sum_{1 \leq i<j \leq k+l} \widetilde\eta_{i,j}(C_{\mathfrak{su}(n)})$, to be considered as an element of $\mathrm{End}((\C^{n})^{\otimes k+l})$. If $k=l=1$, then 
$$\esper[|g_{ii}|^{2}]=\frac{1}{n}+\left(1-\frac{1}{n}\right)\E^{-t}\quad;\quad \esper[|g_{ij}|^{2}]=\frac{1}{n}\left(1-\E^{-t}\right)\quad;\quad\esper[g_{ii}\overline{g_{jj}}]=\E^{-t}$$\vspace{1mm}
since the eigenvalues and eigenvectors of $M_{n,1,1}=\I I_{n}\otimes \I I_{n}+n\,\sum_{i,j=1}^{n} E_{ij}\otimes E_{ij}$ are:\vspace{1mm}
$$
\begin{tabular}{|c|c|c|}
\hline\vspace{-2.5mm}&&\\
eigenvalue & multiplicity & eigenvectors \\
\vspace{-2.5mm}&&\\
\hline \vspace{-2.5mm}&& \\ 
$n^{2}-1$ & $1$ & $\sum_{i=1}^{n}e[i,i]$ \\
\vspace{-2.5mm}&&\\
\hline \vspace{-2.5mm}&& \\ 
$-1$ & $n^{2}-1$ &  $e[i,j],\,\,1\leq i\neq j\leq n$ \\
\vspace{-2mm}&&\\
& & $e[i,i]-e[i+1,i+1],\,\,i\leq n-1$ \\
&&\\
\hline
\end{tabular}
$$\vspace{1mm}

\noindent If $k=l=2$, then the eigenvalues and eigenvectors of $M_{n,2,2}$ are:\vspace{1mm}
$$
\begin{tabular}{|c|c|c|}
\hline\vspace{-2.5mm}&&\\
eigenvalue & multiplicity & eigenvectors (not exhaustive, some repetitions)\\
\vspace{-2.5mm}&&\\
\hline \vspace{-2.5mm}&& \\ 
$2n^{2}-2$ & $2$ & $\sum_{k=1}^{n}\sum_{l=1}^{n}e[k,l,k,l],\,\,\sum_{k=1}^{n}\sum_{l=1}^{n}e[k,l,l,k]$ \\
\vspace{-2.5mm}&& \\
\hline \vspace{-2.5mm}&& \\ 
$n^{2}-2$ & $4(n+1)(n-1)$ &$\sum_{k=1}^{n}e[i,k,i,k]-e[i+1,k,i+1,k], \,\, i \leq n-1$\\
\vspace{-2mm}&&\\
&& $\sum_{k=1}^{n}e[k,i,k,i]-e[k,i+1,k,i+1],\,\, i \leq n-1$ \\
\vspace{-2mm}&&\\
&&$\sum_{k=1}^{n}e[i,k,k,i]-e[i+1,k,k,i+1], \,\, i \leq n-1$\\
\vspace{-2mm}&&\\
&& $\sum_{k=1}^{n}e[k,i,i,k]-e[k,i+1,i+1,k],\,\, i \leq n-1$ \\
\vspace{-2.5mm}&&\\
\hline \vspace{-2.5mm}&& \\ 
$2n-2$ & $\frac{n^{2}(n+1)(n-3)}{4}$ & $\left(\substack{(e[i,j]-e[j,i])^{\otimes 2}-(e[j,k]-e[k,j])^{\otimes 2}\\+(e[k,l]-e[l,k])^{\otimes 2}-(e[l,i]-e[i,l])^{\otimes 2}}\right),\,\,i \neq j \neq k \neq l$\\
\vspace{-2.5mm}&&\\
\hline \vspace{-2.5mm}&& \\ 
$-2$ & $\frac{(n+2)(n+1)(n-1)(n-2)}{2}$ & $\left(\substack{e[i,j,i,j]+e[j,k,j,k]+e[k,i,k,i] \\ -e[j,i,j,i]-e[k,j,k,j]- e[i,k,i,k]}\right),\,\,i<j<k$\\
\vspace{-2mm}&&\\
&& $\left(\substack{e[i,j,j,i]+e[j,k,k,j]+e[k,i,i,k] \\ -e[j,i,i,j]-e[k,j,j,k]- e[i,k,k,i]}\right),\,\,i<j<k$\\
\vspace{-2.5mm}&&\\
\hline \vspace{-2.5mm}&& \\ 
$-2n-2$ & $\frac{n^{2}(n-1)(n+3)}{4}$ & $ e[i,i,i,i]+e[j,j,j,j]-(e[i,j]+e[j,i])^{\otimes 2},\,\,i<j$\\
&&\\
\hline
\end{tabular}
$$\vspace{1mm}

\noindent 
Again, we can use the previous table to decompose some elementary $4$-tensors in eigenvectors of $M_{n,2,2}$; we refer to Appendix \ref{expun}.  Thus,
\begin{align*}\esper[|g_{ii}|^{4}]&=\frac{2}{n(n+1)}+\frac{4(n-1)}{n(n+2)}\,\E^{-t}+\frac{n(n-1)}{(n+1)(n+2)}\,\E^{-\frac{2n+2}{n}t};\\
\esper[|g_{ij}|^{4}]&=\frac{2}{n(n+1)}-\frac{4}{n(n+2)}\,\E^{-t}+\frac{2}{(n+1)(n+2)}\,\E^{-\frac{2n+2}{n}t};\\
\esper[|g_{ii}|^{2}|g_{ij}|^{2}]&= \frac{1}{n(n+1)}+\frac{n-2}{n(n+2)}\,\E^{-t}-\frac{n}{(n+1)(n+2)}\,\E^{-\frac{2n+2}{n}t};\\
\esper[|g_{ij}|^{2}|g_{ik}|^{2}]&=\frac{1}{n(n+1)}-\frac{2}{n(n+2)}\,\E^{-t} +\frac{1}{(n+1)(n+2)}\,\E^{-\frac{2n+2}{n}t};\\
\esper[|g_{ii}|^{2}|g_{jj}|^{2}]&=\frac{1}{(n-1)(n+1)}+ \frac{2(n+1)}{n(n+2)}\,\E^{-t}+\frac{n-3}{4(n-1)}\,\E^{-\frac{2n-2}{n}t}\\
&\quad+\frac{n-2}{2n}\,\E^{-2t}+\frac{n^{2}+n+2}{4(n+1)(n+2)}\,\E^{-\frac{2n+2}{n}t};\\
\esper[|g_{ij}|^{2}|g_{ji}|^{2}]&=\frac{1}{(n-1)(n+1)}-\frac{2}{n(n+2)}\,\E^{-t}+\frac{n-3}{4(n-1)}\,\E^{-\frac{2n-2}{n}t}\\
&\quad-\frac{n-2}{2n}\,\E^{-2t}+\frac{n^{2}+n+2}{4(n+1)(n+2)}\,\E^{-\frac{2n+2}{n}t};\\
\esper[|g_{ii}|^{2}|g_{jk}|^{2}]&=\frac{1}{(n-1)(n+1)}+\frac{n^{2}-2n-2}{n(n-2)(n+2)}\,\E^{-t}-\frac{n-3}{4(n-1)(n-2)}\,\E^{-\frac{2n-2}{n}t}\\
&\quad-\frac{1}{2n}\,\E^{-2t}-\frac{n-1}{4(n+1)(n+2)}\,\E^{-\frac{2n+2}{n}t};\\
\esper[|g_{ij}|^{2}|g_{jk}|^{2}]&=\frac{1}{(n-1)(n+1)}-\frac{2(n-1)}{n(n-2)(n+2)}\,\E^{-t}-\frac{n-3}{4(n-1)(n-2)}\,\E^{-\frac{2n-2}{n}t}\\
&\quad+\frac{1}{2n}\,\E^{-2t}-\frac{n-1}{4(n+1)(n+2)}\,\E^{-\frac{2n+2}{n}t};\end{align*}
\begin{align*}
\esper[|g_{ij}|^{2}|g_{kl}|^{2}]&=\frac{1}{(n-1)(n+1)}-\frac{2}{(n-2)(n+2)}\,\E^{-t}+\frac{1}{2(n-1)(n-2)}\,\E^{-\frac{2n-2}{n}t}\\
&\quad+\frac{1}{2(n+1)(n+2)}\,\E^{-\frac{2n+2}{n}t}.
\end{align*}\medskip

\begin{proposition}
For the symmetric spaces with isometry group $\SU(n)$ or $\SU(2n)$, the coefficients of Lemma \ref{abstractexpansion} are:
\begin{align*}
&\Gra(n,q,\C): \quad \frac{1}{n^{2}-1}+ \frac{\frac{2n^{2}}{pq}-8}{n^{2}-4}\,\phi^{(2,1,\ldots,1)_{n-1}}+\frac{n^{2}}{2}\left( \frac{1}{(n-1)(n-2)}-\frac{1}{pq(n-2)}\right)\phi^{(2,2,1,\ldots,1,0)_{n-1}};\\
&\qquad\qquad\qquad +\frac{n^{2}}{2}\left(\frac{1}{(n+1)(n+2)}+\frac{1}{pq(n+2)}\right)\phi^{(4,2,\ldots,2)_{n-1}}; \\
&\SU(n)/\SO(n): \quad \frac{2}{n^{2}+n}+\frac{n^{2}+n-2}{n^{2}+n}\,\phi^{(4,2,\ldots,2)_{n-1}};\\
&\SU(2n)/\unit\SP(n):\quad \frac{1}{2n^{2}-n}+ \frac{2n^{2}-n-1}{2n^{2}-n} \,\phi^{(2,2,1,\ldots,1,0)_{2n-1}}.
\end{align*}
\end{proposition}
\medskip

\begin{proof}
For $\SU(n)/\SO(n)$ and $\SU(2n)/\unit\SP(n)$, the only missing coefficient has already been computed. For complex Grassmannians, $(\phi^{(2,1,\ldots,1)_{n-1}})^{2}$ has exactly the same expansion as in the real case, but with square modules. From the computation of the joint moments $\esper[|g_{ij}g_{kl}|^{2}]$ performed previously, one deduces that the expectation of the square of the discriminating zonal function is
\begin{align*}
&\frac{1}{n^{2}-1}+\frac{\frac{2n^{2}}{pq}-8}{n^{2}-4}\,\E^{-t}+\frac{n^{2}}{2}\left( \frac{1}{(n-1)(n-2)}-\frac{1}{pq(n-2)}\right)\E^{-\frac{2n-2}{n}t}\\
&+\frac{n^{2}}{2}\left(\frac{1}{(n+1)(n+2)}+\frac{1}{pq(n+2)}\right)\E^{-\frac{2n+2}{n}t}
\end{align*}
whence the expansion in zonal spherical functions by identifying the coefficients.
\end{proof}
\bigskip

\subsubsection{Quotients of symplectic groups} Finally, set $\frac{1}{n}M_{n,k}=\sum_{1 \leq i<j \leq k+l} \widetilde\eta_{i,j}(C_{\mathfrak{usp}(n)})$, which is considered as an element of $\mathrm{End}((\C^{2n})^{\otimes k})$. Recall that the diagonalization of these matrices will yield the joint moments of the entries of $\widetilde{g}$, the matrix obtained from $g$ by the map \eqref{doublequaternion}. Again, as a warm-up, let us compute the joint moments of order $2$. If $k=2$, then 
\begin{align*}
\esper\!\left[|g_{ii}|^{2}\right]&=\frac{1}{n}+\frac{n-1}{n}\,\E^{-t}\qquad;\qquad\esper\!\left[|g_{ij}|^{2}\right]=\frac{1}{n}\left(1-\E^{-t}\right) \quad \forall i,j \in \lle 1,n\rre;\\
\esper\!\left[(\widetilde{g}_{ii})^{2}\right]&=\E^{-\frac{n+1}{n}t}\qquad;\qquad\esper\!\left[(\widetilde{g}_{ij})^{2}\right]=0\quad\forall i,j \in \lle 1,2n\rre;
\end{align*}
since the eigenvectors and eigenvalues of $M_{n,2}$ are:\vspace{1mm}
$$\!\!\!\!\!\!\!
\begin{tabular}{|c|c|c|}
\hline\vspace{-2.5mm}&&\\
eigenvalue & multiplicity & eigenvectors \\
\vspace{-2.5mm}&&\\
\hline \vspace{-2.5mm}&& \\ 
$\frac{2n+1}{2}$ & $1$ &  $ \sum_{i=1}^{n} e[2i-1,2i]-e[2i,2i-1]$ \\
\vspace{-2.5mm}&&\\
\hline \vspace{-2.5mm}&& \\ 
$\frac{1}{2}$ & $(n-1)(2n+1) $ & $ (e[2i-1,2i]-e[2i,2i-1])-(e[2i+1,2i+2]-e[2i+2,2i+1]),\,\,i\leq n-1$ \\
\vspace{-2mm}&&\\
&&$e[2i-1,2j-1]-e[2j-1,2i-1],\,\,e[2i,2j]-e[2j,2i],\,\, 1\leq i<j \leq n$\\
\vspace{-2mm}&&\\
&&$e[2i-1,2j]-e[2j,2i-1],\,\, 1\leq i \neq j\leq n$\\
\vspace{-2.5mm}&&\\
\hline \vspace{-2.5mm}&& \\ 
$-\frac{1}{2}$ & $n(2n+1)$ &  $e_{k}\otimes e_{l}+e_{l} \otimes e_{k}, 1 \leq k \leq l \leq 2n$\\
&&\\
\hline
\end{tabular}
$$\vspace{1mm}

For $k=4$, we refer to Appendix \ref{expspn} for the expansion in eigenvectors of simple tensors. One obtains:
\begin{align*}\esper\!\left[(\widetilde{g}_{ii})^{4}\right]&=\E^{-\frac{2n+4}{n}t};\\
\esper\!\left[(\widetilde{g}_{ij})^{4}\right]&=\esper\!\left[(\widetilde{g}_{ii}\,\widetilde{g}_{ij})^{2}\right]=\esper\!\left[(\widetilde{g}_{ij}\,\widetilde{g}_{ik})^{2}\right]=0;\\
\esper[(\widetilde{g}_{2i-1,2i-1}\,\widetilde{g}_{2i,2i})^{2}]&=\frac{1}{n(2n+1)}+\frac{n-1}{n(n+1)}\,\E^{-t}+\frac{1}{n+1}\,\E^{-\frac{n+1}{n}t}+\frac{(2n-1)(2n-2)}{3(2n+1)(2n+2)}\,\E^{-\frac{2n+1}{n}t}\\
&\quad+\frac{n-1}{2(n+1)}\,\E^{-\frac{2n+2}{n}t}+\frac{1}{6}\,\E^{-\frac{2n+4}{n}t};\\
\esper[(\widetilde{g}_{2i-1,2i}\,\widetilde{g}_{2i,2i-1})^{2}]&=\frac{1}{n(2n+1)}+\frac{n-1}{n(n+1)}\,\E^{-t}-\frac{1}{n+1}\,\E^{-\frac{n+1}{n}t}+\frac{(2n-1)(2n-2)}{3(2n+1)(2n+2)}\,\E^{-\frac{2n+1}{n}t}\\
&\quad-\frac{n-1}{2(n+1)}\,\E^{-\frac{2n+2}{n}t}+\frac{1}{6}\,\E^{-\frac{2n+4}{n}t};\\
\esper[(\widetilde{g}_{2i-1,2j-1}\,\widetilde{g}_{2i,2j})^{2}]&=\esper[(\widetilde{g}_{2i-1,2j}\,\widetilde{g}_{2i,2j-1})^{2}]=\frac{1}{n(2n+1)}-\frac{1}{n(n+1)}\,\E^{-t}+\frac{1}{(2n+1)(n+1)}\,\E^{-\frac{2n+1}{n}t};\end{align*}
and the other moments of type $\esper[(\widetilde{g}_{2i-1,a}\widetilde{g}_{2i,b})^{2}]$ vanish. On the other hand, assuming that $\{a,b\}$ is not a pair $\{2i-1,2i\}$ in $\lle 1,2n\rre$, one has also
\begin{align*}\esper[(\widetilde{g}_{aa}\,\widetilde{g}_{bb})^{2}]&=\frac{1}{3}\,\E^{-\frac{2n+1}{n}t}+\frac{1}{2}\,\E^{-\frac{2n+2}{n}t}+\frac{1}{6}\,\E^{-\frac{2n+4}{n}t};\\
\esper[(\widetilde{g}_{ab}\,\widetilde{g}_{ba})^{2}]&=\frac{1}{3}\,\E^{-\frac{2n+1}{n}t}-\frac{1}{2}\,\E^{-\frac{2n+2}{n}t}+\frac{1}{6}\,\E^{-\frac{2n+4}{n}t};
\end{align*}
and the other moments of type $\esper[(\widetilde{g}_{ab}\,\widetilde{g}_{cd})^{2}]$ with $\{c,d\}\neq \{a,b\}$ vanish.\bigskip

The same expansions enable one to compute many moments of type $\esper[|g_{ij}\,g_{kl}|^{2}]$, namely, all those that write as $\esper[|g_{ij}\,g_{ik}|^{2}]$. For instance, since
$|g_{ii}|^{4}=(\widetilde{g}_{2i-1,2i-1}\,\widetilde{g}_{2i,2i}-\widetilde{g}_{2i-1,2i}\,\widetilde{g}_{2i,2i-1})^{2},$
its expectation is a combination of those of $(\widetilde{g}_{2i-1,2i-1}\,\widetilde{g}_{2i,2i})^{2}$, $(\widetilde{g}_{2i-1,2i}\,\widetilde{g}_{2i,2i-1})^{2}$ and $\widetilde{g}_{2i-1,2i-1}\,\widetilde{g}_{2i-1,2i}\,\widetilde{g}_{2i,2i}\,\widetilde{g}_{2i,2i-1}$. This last expectation is
$$-\frac{1}{2n(2n+1)}-\frac{n-1}{2n(n+1)}\,\E^{-t}-\frac{(2n-1)(2n-2)}{6(2n+1)(2n+2)}\,\E^{-\frac{2n+1}{n}t}+\frac{1}{6}\,\E^{-\frac{2n+4}{n}t}.$$
Thus, with a few more computations, one gets
\begin{align*}\esper[|g_{ii}|^{4}]&=\frac{3}{n(2n+1)}+\frac{3(n-1)}{n(n+1)}\,\E^{-t}+\frac{(2n-1)(2n-2)}{(2n+1)(2n+2)}\,\E^{-\frac{2n+1}{n}t};\\
\esper[|g_{ij}|^{4}]&=\frac{3}{n(2n+1)}-\frac{3}{n(n+1)}\,\E^{-t}+\frac{3}{(2n+1)(n+1)}\,\E^{-\frac{2n+1}{n}t};\\
\esper[|g_{ii}\,g_{ij}|^{2}]&=\frac{2}{n(2n+1)}+\frac{(n-2)}{n(n+1)}\,\E^{-t} - \frac{2(2n-1)}{(2n+1)(2n+2)}\,\E^{-\frac{2n+1}{n}t};\\
\esper[|g_{ij}\,g_{ik}|^{2}]&=\frac{2}{n(2n+1)}-\frac{2}{n(n+1)}\,\E^{-t}+\frac{2}{(2n+1)(n+1)}\,\E^{-\frac{2n+1}{n}t};\\
\esper[|g_{ii}\,g_{jj}|^{2}]&=\frac{2n-1}{n(n-1)(2n+1)}+\frac{2}{n+1}\,\E^{-t}+ \frac{n-3}{6(n-1)}\,\E^{-\frac{2n-2}{n}t}\\
&\quad+\frac{n-2}{2n}\,\E^{-2t}+\frac{2n^{2}-n+3}{3(n+1)(2n+1)}\,\E^{-\frac{2n+1}{n}t};
\end{align*}
\begin{align*}
\esper[|g_{ij}\,g_{ji}|^{2}]&=\frac{2n-1}{n(n-1)(2n+1)}-\frac{2}{n(n+1)}\,\E^{-t}+\frac{n-3}{6(n-1)}\,\E^{-\frac{2n-2}{n}t}\\
&\quad-\frac{n-2}{2n}\,\E^{-2t}+\frac{2n^{2}-n+3}{3(n+1)(2n+1)}\,\E^{-\frac{2n+1}{n}t};\\
\esper[|g_{ii}\,g_{jk}|^{2}]&=\frac{2n-1}{n(n-1)(2n+1)}+\frac{n^2-3n+1}{n(n+1)(n-2)}\,\E^{-t}-\frac{n-3}{6(n-1)(n-2)}\,\E^{-\frac{2n-2}{n}t}\\
&\quad-\frac{1}{2n}\,\E^{-2t}-\frac{2n-3}{3(n+1)(2n+1)}\,\E^{-\frac{2n+1}{n}t};\\
\esper[|g_{ij}\,g_{jk}|^{2}]&=\frac{2n-1}{n(n-1)(2n+1)}-\frac{2n-3}{n(n+1)(n-2)}\,\E^{-t}-\frac{n-3}{6(n-1)(n-2)}\,\E^{-\frac{2n-2}{n}t} \\
&\quad +\frac{1}{2n}\,\E^{-2t}-\frac{2n-3}{3(n+1)(2n+1)}\,\E^{-\frac{2n+1}{n}t};\\
\esper[|g_{ij}\,g_{kl}|^{2}]&=\frac{2n-1}{n(n-1)(2n+1)}-\frac{2n-2}{n(n+1)(n-2)}\,\E^{-t}+\frac{1}{3(n-1)(n-2)}\,\E^{-\frac{2n-2}{n}t} \\
&\quad+ \frac{4}{3(n+1)(2n+1)}\,\E^{-\frac{2n+1}{n}t}.
\end{align*}\medskip

\begin{proposition}
For the quaternionic Grassmannian varieties $\Gra(n,q,\Hq)$ and the spaces $\unit\SP(n)/\unit(n)$, the coefficients of Lemma \ref{abstractexpansion} are:
\begin{align*}
&\Gra(n,q,\Hq):\quad\frac{1}{2n^{2}-n-1}+\frac{n^{2}}{3}\left(\frac{1}{(n-1)(n-2)}-\frac{1}{pq(n-2)}\right)\phi^{(1^4,0,\ldots,0)_{n}}\\
&\qquad\qquad\qquad+\frac{\frac{n^{2}}{pq}-4}{(n-2)(n+1)}\,\phi^{(1^2,0,\ldots,0)_{n-1}}+\frac{n^{2}}{3}\left(\frac{4}{(n+1)(2n+1)}+\frac{1}{pq(n+1)}\right)\phi^{(2,2,0,\ldots,0)_{n}};\\
&\unit\SP(n)/\unit(n):\quad \frac{1}{2n^{2}+n}+\frac{4(n-1)(n+1)}{3n(2n+1)}\,\phi^{(2,2,0,\ldots,0)_{n}}+\frac{n+1}{3n}\,\phi^{(4,0,\ldots,0)_{n}}.
\end{align*}
\end{proposition}
\begin{proof}
The case of quaternionic Grassmannians is again done by using the expansion on  page \pageref{expansquaregrass}, with square modules instead of squares. One obtains the following formula for the expectation of $(\phi^{(1,1,0,\ldots,0)_{n}})^{2}$:
\begin{align*}&\frac{1}{2n^{2}-n-1}+\frac{\frac{n^{2}}{pq}-4}{(n-2)(n+1)}\,\E^{-t}+\frac{n^{2}}{3}\left(\frac{1}{(n-1)(n-2)}-\frac{1}{pq(n-2)}\right)\E^{-\frac{2n-2}{n}t}\\
&+\frac{n^{2}}{3}\left(\frac{4}{(n+1)(2n+1)}+\frac{1}{pq(n+1)}\right)\E^{-\frac{2n+1}{n}t}, \end{align*}
hence the expansion in zonal functions by identification of the coefficients. Finally, for the structure spaces $\unit\SP(n)/\unit(n)$, $(\phi^{(2,0,\ldots,0)_{n}})^{2}$ is equal to
$$\frac{1}{2n}\left(T[(\widetilde{g}_{11})^{4}]+T[(\widetilde{g}_{11}\widetilde{g}_{22})^{2}]+T[(\widetilde{g}_{12}\widetilde{g}_{21})^{2}]\right)+\frac{n-1}{n}\left(T[(\widetilde{g}_{13}\widetilde{g}_{24})^{2}]+T[(\widetilde{g}_{11}\widetilde{g}_{33})^{2}]+T[(\widetilde{g}_{13}\widetilde{g}_{31})^{2}]\right)$$
plus some remainder whose expectation under Brownian measures will be zero. Hence, 
$$\esper[(\phi^{(2,0,\ldots,0)_{n}})^{2}]=\frac{1}{n(2n+1)}+\frac{4(n-1)(n+1)}{3n(2n+1)}\,\E^{-\frac{2n+1}{n}t}+\frac{n+1}{3n}\,\E^{-\frac{2n+4}{n}t},$$
and $\frac{2n+1}{n}$ is the exponent corresponding to the spherical representation of label $(2,2,0,\ldots,0)_{n}$, whereas $\frac{2n+4}{n}$ is the exponent corresponding to the spherical representation of label $(4,0,\ldots,0)_{n}$.
\end{proof}

\subsection{Proof of the lower bound on the total variation distance}\label{bienayme}
The proof of the lower bound is now a simple application of Bienaym\'e-Chebyshev inequality. First, under the Haar measure, we have:
\begin{proposition}\label{upperboundhaartrace}
If $E_{a}$ is the event $\{|\Omega|\geq a\}$, then the Haar measure of $E_{a}$ satisfies the inequality
$$
\eta_{X}(E_{a}) \leq \frac{1}{a^{2}}
$$
for every classical simple compact Lie group $X=K$ and every classical simple compact symmetric space $X=G/K$.
\end{proposition}
\begin{proof}
The previous computations ensure that $\esper_{\infty}[|\Omega|^{2}]=1$ in every case, so
$$\eta_{X}[|\Omega|\geq a]=\eta_{X}[|\Omega|^{2}\geq a^{2}]\leq \frac{\esper_{\infty}[|\Omega|^{2}]}{a^{2}}=\frac{1}{a^{2}}.\vspace{-11mm}$$
\end{proof}\vspace{2mm}

Next, let us estimate $\esper_{t}[\Omega]$ and $\Var_{t}[\Omega]$ for $t=\alpha\,(1-\eps)\log n$. The exact values are listed in the table on the following page. We assume $\eps < \frac{1}{4}$; indeed, Lemma \ref{nonincreasingdistance} ensures that it is sufficient to control the total variation distance around the cut-off time. We shall use a lot the inequality of convexity
$$\exp(x)\leq 1+\frac{\E^{y}-1}{y}\,x\quad \forall x \in \left(0,y\right).$$
\begin{lemma}\label{controlvariance}
Under the usual assumptions on $n$, for groups and spaces of structures (but not for Grassmannian varieties), $\Var_{t}[\Omega]$ is uniformly bounded for every $t =\alpha\,(1-\eps)\,\log n$ with $\eps \in (0,1/4)$. Possible upper bounds are listed below:
\begin{align*} 
&\SU(n),\,\,\SU(n)/\SO(n),\,\, \SU(2n)/\unit\SP(n): 1 \quad;\\
&\SO(2n)/\unit(n),\,\,\unit\SP(n),\,\,\unit\SP(n)/\unit(n): 3\quad;\\
&\SO(n): 8.
\end{align*}
\end{lemma}
\begin{proof}
We proceed case by case, and denote $\Delta_{t}(\lambda,\mu)=\E^{-\lambda t}-\E^{-\mu t}$. Notice that $\Delta_{t}(\lambda,\mu)\leq 0 $ if $\lambda \geq \mu$. On the other hand, $\Delta_{t}(\lambda,\mu)$ is always smaller than $1$ for $\lambda, \mu \geq 0$.\vspace{2mm}
\begin{itemize}
\item $\SO(n)$:
\begin{align*}
\Var_{t}[\Omega] &= \Delta_{t}\!\left(0,\frac{n-1}{n}\right)+\frac{n(n-1)}{2}\,\Delta_{t}\!\left(\frac{n-4}{n},\frac{n-1}{n}\right)+\left(\frac{n(n+1)}{2}-1\right)\Delta_{t}\!\left(1,\frac{n-1}{n}\right)\\
&\leq 1 + \frac{n(n-1)}{2}\,\Delta_{t}\!\left(\frac{n-4}{n},\frac{n-1}{n}\right)=1+\frac{n(n-1)}{2}\,\E^{-\frac{n-1}{n}t}\,\left(\E^{\frac{3t}{n}}-1\right) \\
&\leq 1+\frac{13}{2}\,n\log n\,\E^{-\frac{n-1}{n}t}
\end{align*}
since $\frac{6 \log n}{n}\leq 1.382$ when $n \geq 10$, and $\frac{\E^{1.382}-1}{1.382} \leq \frac{13}{6}$. Then,
$$\E^{-\frac{n-1}{n}t} \leq \E^{-\frac{3(n-1)\log n}{2n}} = n^{-1}\,\E^{-\frac{(n-3)\,\log n}{2n}}\leq \frac{14}{13}\,(n\log n)^{-1} $$
for $n \geq 10$, so $\Var_{t}[\Omega]\leq 1+7=8$.

$$\!\!\!\!\!\!\!\!\!\begin{tabular}{|c|c|l|}
\hline\vspace{-2.5mm}&&\\
$K$ or $G/K$ & $\esper_{t}[\Omega]$& $\qquad\qquad\qquad\qquad\qquad\quad\Var_{t}[\Omega]$  \\
\vspace{-2.5mm}&&\\
\hline \vspace{-2.5mm}&& \\ 
$\SO(n)$ & $n\,\E^{-\frac{n-1}{2n}t}$& $1+\frac{n(n-1)}{2}\,\E^{-\frac{n-4}{n}t} +\left(\frac{n(n+1)}{2}-1\right)\E^{-t}-n^{2}\,\E^{-\frac{n-1}{n}t}$ \\
\vspace{-2.5mm}&&\\
\hline \vspace{-2.5mm}& & \\ 
$\SU(n)$ &$n\,\E^{-\frac{n^{2}-1}{2n^{2}}t} $& $1+(n^{2}-1)\,\E^{-t}-n^{2}\,\E^{-\frac{n^{2}-1}{n^{2}}t} $  \\
\vspace{-2.5mm}&&\\
\hline \vspace{-2.5mm}&& \\ 
$\unit\SP(n)$&$2n\,\E^{-\frac{2n+1}{4n}t}$ & $1+(2n+1)(n-1)\,\E^{-t}+(2n+1)\,n\,\E^{-\frac{n+1}{n}t}-4n^{2}\,\E^{-\frac{2n+1}{2n}t} $  \\
\vspace{-2.5mm}&&\\
\hline \vspace{-2.5mm}&& \\ 
$\Gra(n,q,\R)$ & $\sqrt{\frac{(n+2)(n-1)}{2}}\,\E^{-t}$& $1+\left(\frac{2n^{2}}{pq}-8\right)\frac{(n-1)(n+2)}{(n-2)(n+4)}\,\E^{-t}+\frac{n^{2}}{3}\left(\frac{n+2}{n-2}-\frac{(n+2)(n-1)}{pq(n-2)}\right)\,\E^{-\frac{2n-2}{n}t}$\\
\vspace{-3mm}& &\\
&&$+\frac{n^{2}}{6}\left(\frac{n-1}{n + 4}+\frac{2(n+2)(n-1)}{pq(n + 4)}\right)\,\E^{-\frac{2n+4}{n}t}-\frac{(n+2)(n-1)}{2}\,\E^{-2t}$  \\ 
\vspace{-2.5mm}&&\\
\hline \vspace{-2.5mm}&& \\ 
$\Gra(n,q,\C)$ &$\sqrt{n^{2}-1}\,\E^{-t}$ & $ 1+\left(\frac{2n^{2}}{pq}-8\right)\frac{n^{2}-1}{n^{2}-4}\,\E^{-t}+\frac{n^{2}}{2}\left( \frac{n+1}{n-2}-\frac{n^{2}-1}{pq(n-2)}\right)\E^{-\frac{2n-2}{n}t}$\\
\vspace{-3mm}&&\\
&&$+\frac{n^{2}}{2}\left(\frac{n-1}{n+2}+\frac{n^{2}-1}{pq(n+2)}\right)\E^{-\frac{2n+2}{n}t}-(n^{2}-1)\,\E^{-2t}$  \\
\vspace{-2.5mm}&&\\
\hline \vspace{-2.5mm}&& \\ 
$\Gra(n,q,\Hq)$ & $\sqrt{(2n+1)(n-1)}\,\E^{-t}$& $ 1+\left(\frac{n^{2}}{pq}-4\right)\frac{(n-1)(2n+1)}{(n-2)(n+1)}\,\E^{-t}+\frac{n^{2}}{3}\left(\frac{2n+1}{n-2}-\frac{(2n+1)(n-1)}{pq(n-2)}\right)\E^{-\frac{2n-2}{n}t}$\\
\vspace{-3mm}&&\\
&&$+\frac{n^{2}}{3}\left(\frac{4(n-1)}{(n+1)}+\frac{(2n+1)(n-1)}{pq(n+1)}\right)\E^{-\frac{2n+1}{n}t} - (2n+1)(n-1)\,\E^{-2t}$ \\
\vspace{-2.5mm}&&\\
\hline \vspace{-2.5mm}&& \\ 
$\SO(2n)/\unit(n)$ &$\sqrt{n(2n-1)}\,\E^{-\frac{n-1}{n}t}$& $1+\frac{(n-1)(2n-1)}{3}\,\E^{-\frac{2n-4}{n}t}+\frac{4(n^{2}-1)}{3}\,\E^{-\frac{2n-1}{n}t} -n(2n-1)\,\E^{-\frac{2n-2}{n}t}$  \\
\vspace{-2.5mm}&&\\
\hline \vspace{-2.5mm}&& \\ 
$\SU(n)/\SO(n)$ &$\sqrt{\frac{n(n+1)}{2}}\,\E^{-\frac{(n-1)(n+2)}{n^{2}}t}$& $ 1+\frac{(n+2)(n-1)}{2}\,\E^{-\frac{2n+2}{n}t}-\frac{n(n+1)}{2}\,\E^{-\frac{(n-1)(2n+4)}{n^{2}}t}$ \\
\vspace{-2.5mm}&&\\
\hline \vspace{-2.5mm}&& \\ 
$\SU(2n)/\unit\SP(n)$ & $\sqrt{2n^{2}-n}\,\E^{-\frac{(n-1)(2n+1)}{2n^{2}}t}$ & $ 1+ (2n^{2}-n-1) \,\E^{-\frac{2n-1}{n}t} - (2n^{2}-n)\,\E^{-\frac{(n-1)(2n+1)}{n^{2}}t}$  \\
\vspace{-2.5mm}&&\\
\hline \vspace{-2.5mm}&& \\ 
$\unit\SP(n)/\unit(n)$ & $\sqrt{n(2n+1)}\,\E^{-\frac{n+1}{n}t}$ & $1+\frac{4(n-1)(n+1)}{3}\,\E^{-\frac{2n+1}{n}t}+\frac{(2n+1)(n+1)}{3}\,\E^{-\frac{2n+4}{n}t} -n(2n+1)\,\E^{-\frac{2n+2}{n}t}$  \\
&&\\
\hline
\end{tabular}\vspace{2mm}$$

\item $\SU(n)$:
$$\Var_{t}[\Omega]= \Delta_{t}\!\left(0,\frac{n^{2}-1}{n^{2}}\right)+(n^{2}-1)\,\Delta_{t}\!\left(1,\frac{n^{2}-1}{n^{2}}\right) \leq 1.$$ \vspace{2mm}
\item $\unit\SP(n)$:
\begin{align*}
\Var_{t}[\Omega] &= \Delta_{t}\!\left(0,\frac{2n+1}{2n}\right)+(2n+1)(n-1)\,\Delta_{t}\!\left(1,\frac{2n+1}{2n}\right)+(2n+1)\,n\,\Delta_{t}\!\left(\frac{2n+2}{2n},\frac{2n+1}{2n}\right)\\
&\leq 1+(2n+1)(n-1)\,\Delta_{t}\!\left(1,\frac{2n+1}{2n}\right)\leq  1+2n^{2}\,\E^{-\frac{2n+1}{2n}t}\,\left(\E^{\frac{t}{2n}}-1\right)\\
&\leq 1+\frac{5}{2}\,n\log n\,\E^{-\frac{2n+1}{2n}t}
\end{align*} 
since $\frac{\log n}{n} \leq 0.367$ when $n \geq 3$, and $\frac{\E^{0.367}-1}{0.367} \leq \frac{5}{4}$. Then,
$$\E^{-\frac{2n+1}{2n}t} \leq \E^{-\frac{3\log n}{2}} = n^{-\frac{3}{2}} \leq \frac{4}{5}\,(n\log n)^{-1}$$
for $n \geq 3$, so $\Var_{t}[\Omega]\leq 1+2=3$.
\vspace{2mm}
\item $\SO(2n)/\unit(n)$: 
\begin{align*}
\Var_{t}[\Omega] &=\Delta_{t}\!\left(0,\frac{2n-2}{n}\right)+\frac{(n-1)(2n-1)}{3}\,\Delta_{t}\!\left(\frac{2n-4}{n},\frac{2n-2}{n}\right)+\frac{4(n^{2}-1)}{3}\,\Delta_{t}\!\left(\frac{2n-1}{n},\frac{2n-2}{n}\right) \\
&\leq 1+\frac{(n-1)(2n-1)}{3}\,\Delta_{t}\!\left(\frac{2n-4}{n},\frac{2n-2}{n}\right) \leq 1+\frac{2n^{2}}{3}\,\E^{-\frac{2n-2}{n}t}\,\left(\E^{\frac{2t}{n}}-1\right) \\
&\leq 1+\frac{20}{9} n \log 2n\,\E^{-\frac{2n-2}{n}t}
\end{align*}
since $\frac{2\log 2n}{n} \leq 0.922$ when $2n \geq 10$, and $\frac{\E^{0.922}-1}{0.922} \leq \frac{5}{3}$. Since
$$\E^{-\frac{2n-2}{n}t} \leq \E^{-\frac{3(n-1)\log 2n}{2n}} = \frac{1}{2}\,n^{-1}\,\E^{-\frac{(n-3)\,\log 2n}{2n}}\leq \frac{3}{4}(n \log 2n)^{-1}$$
for $2n \geq 10$, one concludes that $\Var_{t}[\Omega]\leq 1+\frac{5}{3}\leq 3$.
\vspace{2mm}
\item $\SU(n)/\SO(n)$: 
$$\Var_{t}[\Omega] =\Delta_{t}\!\left(0,\frac{2(n-1)(n+2)}{n^{2}}\right)+\frac{(n+2)(n-1)}{2}\,\Delta_{t}\!\left(\frac{2(n+1)}{n},\frac{2(n-1)(n+2)}{n^{2}}\right) \leq 1.$$
\vspace{2mm}
\item $\SU(2n)/\unit\SP(n)$:
$$\Var_{t}[\Omega] =\Delta_{t}\!\left(0,\frac{(n-1)(2n+1)}{n^{2}}\right) + (2n^{2}-n-1) \,\Delta_{t}\!\left(\frac{2n-1}{n},\frac{(n-1)(2n+1)}{n^{2}}\right) \leq 1.$$
 \vspace{2mm}
\item $\unit\SP(n)/\unit(n)$: 
\begin{align*}
\Var_{t}[\Omega] &=\Delta_{t}\!\left(0,\frac{2n+2}{n}\right)+\frac{4(n^{2}-1)}{3}\,\Delta_{t}\!\left(\frac{2n+1}{n},\frac{2n+2}{n}\right)+\frac{2n^{2}+3n+1}{3}\,\Delta_{t}\!\left(\frac{2n+4}{n},\frac{2n+2}{n}\right) \\
&\leq 1+\frac{4(n^{2}-1)}{3}\,\Delta_{t}\!\left(\frac{2n+1}{n},\frac{2n+2}{n}\right) \leq 1+\frac{4n^{2}}{3}\,\E^{-\frac{2n+2}{n}t}\,\left(\E^{\frac{t}{n}}-1\right) \\
&\leq 1+\frac{5}{3}\,n \log n\,\E^{-\frac{2n+2}{n}t}
\end{align*}
by using the same estimate on $\frac{\log n}{n}$ as in the case of $\unit\SP(n)$. Since
$$\E^{-\frac{2n+2}{n}t} \leq \E^{-\frac{3\log n}{2}} = n^{-\frac{3}{2}} \leq \frac{4}{5}\,(n\log n)^{-1}$$
for $n \geq 3$, one obtains $\Var_{t}[\Omega]\leq 1+\frac{4}{3}\leq 3$.\vspace{2mm}
\end{itemize}
It is not possible to prove such uniform bounds for Grassmannians, because of the term $\E^{-t}$ that appears in the variance. We shall address this problem in Lemma \ref{uniformboundvariancegrass}.
\end{proof}
\begin{proposition}\label{trickychebyshev}
Denote $K_{X}$ the bound computed in the previous Lemma for the variance of the discriminating zonal function $\Omega$ associated to a space $X$.  Then, 
$$\dtv(\mu_{t},\mathrm{Haar})\geq 1-\frac{4(K_{X}+1)}{(\esper_{t}[\Omega])^{2}}.$$
\end{proposition}
\begin{proof}
Assuming $a$ smaller than $m=\esper_{t}[\Omega]$, if $|\Omega-m|\leq a$, then $|\Omega| \geq m-a$. Consequently, 
$$\mu_{t}[|\Omega| \geq m-a] \geq 1 - \proba[|\Omega-m| > a] \geq 1-\frac{\Var_{t}[\Omega]}{a^{2}}=1-\frac{K_{X}}{a^{2}}.$$
Next, take $a=\frac{m}{2}$. The combination of Lemma \ref{upperboundhaartrace} and of the previous inequality yields
$$\dtv(\mu_{t},\mathrm{Haar}) \geq \mu_{t}(E_{a})-\eta_{X}(E_{a}) \geq 1-\frac{K_{X}+1}{a^{2}} = 1 - \frac{4(K_{X}+1)}{m^{2}}.$$
Since $m^{2}$ behaves as $n^{2\eps}$, this essentially ends the proof of the lower bounds in the case of compact Lie groups and compact spaces of structures. More precisely:\vspace{2mm}
\begin{itemize}
\item$\SO(n)$: $m^{2} \geq n^{2\eps}$ so the constant $c$ in our main Theorem \ref{main} is $4(8+1)=36$.\vspace{2mm}
\item$\SU(n)$: again, $m^{2} \geq n^{2\eps}$, so the constant is $4(1+1)=8$. \vspace{2mm}
\item$\unit\SP(n)$: here, $m^{2} \geq 4\,n^{2\eps}\,\E^{-\frac{\log n}{2n}} \geq \frac{16}{5}\,n^{2\eps}$ for $n \geq 3$, so the constant is $\frac{5}{16}\,4(3+1)=5$. \vspace{2mm}
\item$\SO(2n)/\unit(n)$: $m^{2} \geq \frac{2n-1}{4n}\,(2n)^{2\eps} \geq \frac{9}{20}\,(2n)^{2\eps}$ for $2n \geq 10$, whence a constant $\frac{9}{20}\,4(3+1)=\frac{36}{5}\leq 8$. \vspace{2mm}
\item$\SU(n)/\SO(n)$: $m^{2} \geq \frac{n^{2\eps}}{2}\,\E^{-\frac{2(n-2)\log n}{n^{2}}} \geq \frac{n^{2\eps}}{3}$ for $n \geq 2$, so a possible constant is $3\times 4(1+1)=24$.\vspace{2mm}
\item$\SU(2n)/\unit\SP(n)$: $m^{2} \geq \frac{2n-1}{4n}\,(2n)^{2\eps} \geq \frac{3}{8}\,(2n)^{2\eps}$, and a possible constant is $\frac{8}{3}\,4(1+1)=\frac{64}{3}\leq 22$.\vspace{2mm}
\item$\unit\SP(n)/\unit(n)$: $m^{2} \geq 2 n^{2\eps}\, \E^{-\frac{2 \log n}{n}}\geq \frac{16}{17}\,n^{2\eps}$ for $n \geq 3$, whence a constant $\frac{17}{16}\,4(3+1)= 17$.
\end{itemize}\vspace{-6mm}
\end{proof}
\bigskip

Unfortunately, for Grassmannian varieties, the variance of $\Omega$ at time $t=(1-\eps)\log n$ can only be bounded by a constant times $n^{\eps}$. However, since the mean of $\Omega$ is also of order $n^{\eps}$, this will still ensure that the discriminating zonal spherical function has not at all the same behavior under Haar measure and under Brownian measures before cut-off time. The only downside is the loss of a factor $n^{\eps}$ in the estimate of the total variation distance.
\begin{lemma}\label{uniformboundvariancegrass}
Under the usual assumptions on $n$, for Grassmannian varieties, 
$$\frac{\Var_{t}[\Omega]}{n^{\eps}} \leq \begin{cases}
3 &\text{if }\Bbbk=\R,\\
5 &\text{if }\Bbbk=\C\text{ or }\Hq,
\end{cases}$$  
for every $t =\alpha\,(1-\eps)\,\log n$ with $\eps \in (0,1/4)$. 
\end{lemma}
\begin{proof}
The quantity $\frac{1}{pq}$ is bounded by 
$$\frac{4}{n^{2}} \leq \frac{1}{pq} \leq \frac{1}{n-1},$$
the extremal values corresponding to $p=q=\frac{n}{2}$ and to $p=n-1$ or $q=n-1$. In particular, in the expansions hereafter, all the coefficients that precede differences of exponentials $\Delta_{t}(\lambda,\mu)$ are positive. Now, we proceed case by case:\vspace{2mm}
\begin{itemize}
\item $\Gra(n,q,\R)$:
\begin{align*}
\Var_{t}[\Omega]&=\Delta_{t}\!\left(0,2\right)+\left(\frac{2n^{2}}{pq}-8\right)\frac{(n-1)(n+2)}{(n-2)(n+4)}\,\Delta_{t}\!\left(1,2\right)+\frac{n^{2}}{3}\!\left(\frac{n+2}{n-2}-\frac{(n+2)(n-1)}{pq(n-2)}\right)\Delta_{t}\!\left(\frac{2n-2}{n},2\right)\\
&\quad +\frac{n^{2}}{6}\left(\frac{n-1}{n + 4}+\frac{2(n+2)(n-1)}{pq(n + 4)}\right)\Delta_{t}\!\left(\frac{2n+4}{n},2\right)\\
&\leq 1+2n\,\Delta_{t}(1,2)+\frac{n^{2}}{3}\,\Delta_{t}\!\left(\frac{2n-2}{n},2\right).
\end{align*}
For the difference $\Delta_{t}(1,2)$, one cannot obtain a better bound than $\E^{-t}=n^{\eps-1}$. The second difference $\Delta_{t}\!\left(\frac{2n-2}{n},2\right)$ is bounded from above by
$$\E^{-2t}\left(\E^{\frac{2t}{n}}-1\right) \leq n^{-\frac{3}{2}}\,\frac{8\log n}{3n} \leq 2n^{-2}$$
by using similar arguments as in the proof of Lemma \ref{controlvariance}, and the inequality $n\geq 10$. So,
$$\Var_{t}[\Omega] \leq 1+\frac{2}{3}+2n^{\eps} \leq 3n^{\eps}.$$
\vspace{0mm}
\item$\Gra(n,q,\C)$:
\begin{align*}
\Var_{t}[\Omega]&= \Delta_{t}\!\left(0,2\right)+\left(\frac{2n^{2}}{pq}-8\right)\frac{n^{2}-1}{n^{2}-4}\,\Delta_{t}\!\left(1,2\right)+\frac{n^{2}}{2}\left( \frac{n+1}{n-2}-\frac{n^{2}-1}{pq(n-2)}\right)\Delta_{t}\!\left(\frac{2n-2}{n},2\right)\\
&\quad+\frac{n^{2}}{2}\left(\frac{n-1}{n+2}+\frac{n^{2}-1}{pq(n+2)}\right)\,\Delta_{t}\!\left(\frac{2n+2}{2},2\right)\\
&\leq 1+2n\,\Delta_{t}(1,2)+\frac{n^{2}}{2}\,\Delta_{t}\!\left(\frac{2n-2}{n},2\right).
\end{align*}
The second difference is controlled exactly as in the case of real Grassmannians, but under the constraint $n \geq 2$:
$$\Delta_{t}\!\left(\frac{2n-2}{n},2\right)\leq \E^{-2t}\left(\E^{\frac{2t}{n}}-1\right) \leq n^{-\frac{3}{2}}\,\frac{3\log n}{n} \leq \frac{9}{4}\,n^{-2}.$$
Hence, $\Var_{t}[\Omega] \leq 1+\frac{9}{8}+2n^{\eps} \leq 5n^{\eps}.$
\vspace{2mm}
\item$\Gra(n,q,\Hq)$:
\begin{align*}
\Var_{t}[\Omega]&= \Delta_{t}\!\left(0,2\right)+\frac{n^{2}}{3}\left(\frac{2n+1}{n-2}-\frac{(2n+1)(n-1)}{pq(n-2)}\right)\Delta_{t}\!\left(\frac{2n-2}{n},2\right)\\
&\quad+\left(\frac{n^{2}}{pq}-4\right)\frac{(n-1)(2n+1)}{(n-2)(n+1)}\,\Delta_{t}(1,2)+\frac{n^{2}}{3}\left(\frac{4(n-1)}{(n+1)}+\frac{(2n+1)(n-1)}{pq(n+1)}\right)\Delta_{t}\!\left(\frac{2n+1}{n},2\right)\\
&\leq 1+2n\,\Delta_{t}(1,2)+\frac{2n^{2}}{3}\,\Delta_{t}\!\left(\frac{2n-2}{n},2\right)\leq 1+2n^{\eps}+\frac{3}{2} \leq 5n^{\eps}.
\end{align*}\vspace{-11mm}
\end{itemize}
\end{proof}\bigskip

Now, Proposition \ref{trickychebyshev} still holds, but with $K_{X}$ varying with $n$ and equal to $3n^{\eps}$ or $5n^{\eps}$ according to the field $\Bbbk=\R$, $\C$ or $\Hq$. Thus:
\begin{proposition}
For Grassmannian varieties $\Gra(n,q,\Bbbk)$, if $t=(1-\eps)\log n$ with $\eps \in (0,1/4)$, then
$$\dtv(\mu_{t},\mathrm{Haar}) \geq  1 - \frac{L n^{\eps}}{m^{2}}\quad \text{with }L=\begin{cases}
16 &\text{if }\Bbbk=\R,\\
24 &\text{if }\Bbbk=\C\text{ or }\Hq.
\end{cases}$$  
\end{proposition}
\noindent Finally, the deduction of the constants in Theorem \ref{main} goes as follows:\vspace{2mm}
\begin{itemize}
\item $\Gra(n,q,\R)$: $m \geq \frac{n^{2\eps}}{2}$, so the constant can be taken equal to $2\times 16=32$.\vspace{2mm}
\item $\Gra(n,q,\C)$: $m \geq \frac{n^{2}-1}{n^{2}}\,n^{2\eps} \geq \frac{3}{4}\,n^{2\eps}$, so a possible constant is again $\frac{4}{3}\,24=32$.\vspace{2mm}
\item $\Gra(n,q,\Hq)$: $m \geq \frac{2n^{2}-n-1}{2}\,n^{2\eps} \geq \frac{3}{2}\,n^{2\eps}$ for $n \geq 3$, whence a constant $\frac{2}{3}\,24=16$. \vspace{2mm}
\end{itemize}
These computations end the proof of the cut-off phenomenon.\bigskip\vspace{2mm}

\section{Appendices (tedious computations)}
\subsection{Proof of the upper bound for odd special orthogonal groups}\label{annexsoupper}
With the same scheme of growth of partitions as for compact symplectic groups, one has the following bounds:\vspace{2mm}
\begin{itemize}
\item $\eta_{1,n}$: it is given by the exact formula $\binom{2n+1}{n+1}\,\E^{-\frac{n(n+1)}{2n+1}\log (2n+1)}$, which is indeed smaller than $1$ for $n \geq 5$. \vspace{2mm}
\item $\eta_{k\geq 2,n}$: the comparison techniques between sums and integrals give
\begin{align*}
\log \eta_{k,n} &\leq -\frac{n(2k-1+n)}{2n+1}\log (2n+1) + \frac{2}{2k-1}+\frac{1}{k}+\log(k+n-2)-\log k\\
&\quad+(2k+2n-2)\log(2k+2n-2)+(2k-2)\log(2k-2)-2(2k+n-2)\log(2k+n-2).
\end{align*}
This bound is decreasing in $k$, whence smaller than its value when $k=2$, which is negative for every value of $n \geq 5$.
 \vspace{2mm}
\item $\eta_{k,l \in \lle 3,n-1\rre}$: as before, $\rho_{k,l}$ splits into $\rho_{k,l,(1)}$ and $\rho_{k,l,(2)}$:
$$\rho_{k,l}=\!\prod_{j=l+1}^{n} \frac{k+j-1+\lambda_{l+1}-\lambda_{j}}{k+j-l-1+\lambda_{l+1}-\lambda_{j}}\,\frac{k+\lambda_{l+1}+\lambda_{j}+2n-j}{k+\lambda_{l+1}+\lambda_{j}+2n-j-l} \prod_{1\leq i \leq j \leq l}\!\frac{2k+2\lambda_{l+1}+2n+1-i-j}{2k+2\lambda_{l+1}+2n-1-i-j}.$$
The bound on $\log \widetilde{\eta}_{k,l}$, the sum of $\log \rho_{k,l,(2)}$ and of the variation of $-\frac{t_{n,0}}{2}B_{n}(\lambda)$, is
\begin{align*}
\log \widetilde{\eta}_{k,l} &\leq -\frac{l(2k'-1+2n-l)}{2n+1}\log (2n+1) + \frac{1}{k'+n-l-1}+\log(k'+n-2)-\log(k'+n-l-1)\\
&\quad+(2k'+2n-2)\log(2k'+2n-2)+(2k'+2n-2l-2)\log(2k'+2n-2l-2)\\
&\quad-2(2k'+2n-l-2)\log(2k'+2n-l-2)\\
&\leq-\frac{l(2v+2n+1-l)}{2n+1}\log (2n+1) + \frac{1}{v+n-l}+\log(v+n-1)-\log(v+n-l)\\
&\quad+(2v+2n)\log(2v+2n)+(2v+2n-2l)\log(2v+2n-2l)\\
&\quad-2(2v+2n-l)\log(2v+2n-l)
\end{align*}
with $k'=k+\lambda_{l+1}=k+v$. On the other hand, in the product $\rho_{k,l,(1)}$, each term of index $j$ writes as
\begin{align*} \frac{(k'+n-1/2)^{2}-(\lambda_{j}+n+1/2-j)^{2}}{(k'+n-1/2-l)^{2}-(\lambda_{j}+n+1/2-j)^{2}} &\leq \frac{(k'+n-1/2)^{2}-(\lambda_{l+1}+n+1/2-j)^{2}}{(k'+n-1/2-l)^{2}-(\lambda_{l+1}+n+1/2-j)^{2}} \\
&\leq\frac{k+j-1}{k+j-l-1}\,\frac{k''+2n-j}{k''+2n-j-l},
 \end{align*}
so the quantity $\rho_{k,l,(1)}$ is bounded by 
\begin{align*}&\frac{(k+n-1)!}{(k+l-1)!}\,\frac{(k-1)!}{(k+n-l-1)!}\,\frac{(k''+2n-l-1)!}{(k''+n-1)!}\,\frac{(k''+n-l-1)!}{(k''+2n-2l-1)!}\\
&\leq\frac{n!\,(2v+2n-l)!\,(2v+n-l)!}{l!\,(n-l)!\,(2v+n)!\,(2v+2n-2l)!}.
\end{align*}
Again, Stirling approximation leads to
\begin{align*}\log \rho_{k,l,(1)} &\leq (2v+2n-l)\log(2v+2n-l)+(2v+n-l)\log(2v+n-l)-(2v+n)\log(2v+n)\\
&\quad-(2v+2n-2l)\log(2v+2n-2l)+n\log n-l \log l -(n-l)\log (n-l)+\frac{1}{2n-2},
\end{align*}
and therefore 
\begin{align*}\log\eta_{k,l} &\leq -\frac{l(2v+2n+1-l)}{2n+1}\log (2n+1) + \frac{1}{v+n-l}+\log(v+n-1)-\log(v+n-l)\\
&\quad+(2v+2n)\log(2v+2n)+(2v+n-l)\log(2v+n-l)\\
&\quad-(2v+n)\log(2v+n)-(2v+2n-l)\log(2v+2n-l)\\
&\quad+n\log n-l \log l -(n-l)\log (n-l)+\frac{1}{2n-2}\\
&\leq -\frac{l(2n+1-l)}{2n+1}\log (2n+1) + \frac{1}{n-l}+\log(n-1)-\log(n-l)\\
&\quad+n\log n-l \log l -(n-l)\log (n-l)+\frac{1}{2n-2}.
\end{align*}
The last bound is decreasing in $l$, so it suffices to look at the case $l=3$; then the bound is decreasing in $n$, so it suffices to check that the bound is negative when $n=5$, which is just a computation. We conclude that $\log \eta_{k,l}\leq 0$ for any $k$ and any $l \in \lle 3,n-1\rre$.
 \vspace{2mm}
\item $\eta_{k,1}$: a bound on $\rho_{k,1}$ is $\frac{k+n-2}{k}\,\frac{2k+2n-1}{2k+2n-3}$, so
\begin{align*}\eta_{k,1} &\leq \frac{k+n-2}{k}\,\frac{2k+2n-1}{2k+2n-3}\,\E^{-\frac{2k+2n-2}{2n+1}\log(2n+1)}\leq \frac{(n-1)(2n+1)}{2n-1}\,\E^{-\frac{2n}{2n+1}\log(2n+1)}\\
&\leq \frac{n-1}{2n-1}\,\E^{\frac{\log(2n+1)}{2n+1}}\leq \frac{1}{2}\,\E^{\frac{\log 11}{11}}\leq 1.
\end{align*}\vspace{2mm}
\item $\eta_{k,2}$: a bound on $\rho_{k,2}$ is $\frac{k+2n-4}{k}\,\frac{k+2n-3}{k+1}\,\frac{2k+2n-1}{2k+2n-5}\,\frac{k+n-1}{k+n-2}$, so
$$\eta_{k,2}\leq \frac{k+2n-4}{k}\,\frac{k+2n-3}{k+1}\,\frac{2k+2n-1}{2k+2n-5}\,\frac{k+n-1}{k+n-2}\,\E^{-\frac{4k+4n-6}{2n+1}\log(2n+1)}\leq \frac{n}{2n+1}\,\E^{\frac{4\log(2n+1)}{2n+1}}.$$
The last bound is bigger than $1$ only when $n=5$ or $6$. The maximal value is obtained for $n=5$, and is smaller than  $1.09 \leq \frac{11}{10}$. Moreover, if $k \geq 2$, then one has a much better bound, that is smaller than $1$ even when $n=5$ or $6$.
\vspace{2mm}
\end{itemize}
Putting all together, one sees that at most one quotient $\eta_{k,l}$ may be bigger than $1$ (and actually only when $n=5$ or $6$). Thus, we have proved Proposition \ref{tobeprovedinannexsoupper}.
\bigskip

\subsection{Proof of the upper bound for even special orthogonal groups}\label{annexsouppereven}
We analyze as before the various quotients $\rho_{k,l}$ and $\eta_{k,l}$ corresponding to the growth of partition described by Equation \eqref{evolution}:\vspace{2mm}
\begin{itemize}
\item $\eta_{k,n}$: the general formula is 
$$\eta_{k,n}=\left(\prod_{i=1}^{n-1}\frac{2k+2n-2i-1}{2k+n-i-1}\,\frac{2k+2n-2i-2}{2k+n-i-2}\right)\E^{-\frac{2k+n-2}{2}\log(2n)},$$
which is decreasing in $k$ and reduces to $\binom{2n-1}{n}\,\E^{-\frac{n\log(2n)}{2}}$ when $k=1$. This latter bound is always smaller than $1$. \vspace{2mm}
\item $\eta_{k,l\in \lle 2,n-1\rre}$: the quotient of dimensions $\rho_{k,l}=\rho_{k,l,(1)}\,\rho_{k,l,(2)}$ is equal to
$$\prod_{j=l+1}^{n}\frac{k+j-1+\lambda_{l+1}-\lambda_{j}}{k+j-l-1+\lambda_{l+1}-\lambda_{j}}\,\frac{k+\lambda_{l+1}+\lambda_{j}+2n-1-j}{k+\lambda_{l+1}+\lambda_{j}+2n-1-j-l}\,\prod_{1\leq i<j\leq l}\!\frac{2k+2\lambda_{l+1}+2n-i-j}{2k+2\lambda_{l+1}+2n-2-i-j}.$$
The main difference with the previous computations is that $\rho_{k,l,(2)}$ is a product over distinct indices $i<j$, so we will not have to worry about diagonal terms in the corresponding sum (see the argument at the beginning of \S\ref{sympupper}). Hence, with the same notations as before,
\begin{align*}
\log \widetilde{\eta}_{k,l}&\leq -\frac{l(2k'-2+2n-l)}{2n}\log(2n)+(2k'+2n-3)\log(2k'+2n-3)\\
&\quad+(2k'+2n-2l-3)\log(2k'+2n-2l-3)-2(2k'+2n-l-3)\log(2k'+2n-l-3)\\
&\leq -\frac{l(2v+2n-l)}{2n}\log(2n)+(2v+2n-1)\log(2v+2n-1)\\
&\quad+(2v+2n-2l-1)\log(2n-2l-1)-2(2v+2n-l-1)\log(2v+2n-l-1);\\
\log \rho_{k,l,(1)}&\leq (2v+2n-l-1)\log(2v+2n-l-1)+(2v+n-l-1)\log(2v+n-l-1)\\
&\quad-(2v+n-1)\log(2v+n-1)-(2v+2n-2l-1)\log(2v+2n-2l-1)\\
&\quad+n\log n-l \log l -(n-l)\log (n-l)+\frac{1}{2n-2}.
\end{align*}
Adding together these bounds, using the concavity of $x \log x$ and then the decreasing character with respect to $v$ gives
$$\log \eta_{k,l} = \log \widetilde{\eta}_{k,l} + \log \rho_{k,l,(1)}\leq -\frac{l(2n-l)}{2n}\log(2n)+n \log n - l \log l -(n-l)\log(n-l)+\frac{1}{2n-2},$$
which is decreasing in $l \geq 2$. Then,
$$-\frac{2n-2}{n}\log(2n)+n\log(n)-2\log 2 -(n-2)\log(n-2)+ \frac{1}{2n-2}$$
is decreasing in $n$, and one can check that it is negative when $n=5$. So, $\eta_{k,l}\leq 1$ for any $k$ and any $l \in \lle 2,n-1\rre$.
\vspace{2mm}
\item $\eta_{k,1}$: one has $\rho_{k,1}\leq \frac{k+2n-3}{k}\,\frac{k+n-1}{k+n-2}$, and therefore
$$\eta_{k,1}\leq \frac{k+2n-3}{k}\,\frac{k+n-1}{k+n-2}\,\E^{-\frac{2n+2k-3}{2n}\log(2n)}.$$
Suppose $k\geq 2$; then the right-hand side is smaller than $\frac{2n-1}{2n}\,\frac{n+1}{2n}$, so $\eta_{k,1}\leq 1$. On the other hand, for $k=1$, which happens only once, 
$$\eta_{1,1}\leq \E^{\frac{\log (2n)}{2n}} \leq \E^{\frac{\log 10}{10}}\leq \frac{4}{3}.$$
\end{itemize}
This proves Proposition \ref{tobeprovedinannexsouppereven}.
\bigskip

\subsection{Proof of the upper bound for complex Grassmannians}\label{annexsuupper}
For a partition of size $p=\lfloor \frac{n}{2}\rfloor$, one has $B_{n}(\lambda)= \frac{2}{n}\sum_{i=1}^{p} \lambda_{i}^{2}+(n+1-2i)\lambda_{i}$ and either
$$ A_{n}(\lambda) = \left(\prod_{i=1}^{p}\prod_{j=1}^{p}\frac{\lambda_{i}+\lambda_{j}+n+1-i-j}{n+1-i-j}\right)\left(\prod_{1\leq i<j\leq p} \frac{\lambda_{i}-\lambda_{j}+j-i}{j-i}\right)^{2} $$
if $n=2p$, or
$$ A_{n}(\lambda) = \left(\prod_{i=1}^{p+1}\prod_{j=1}^{p+1}\frac{\lambda_{i}+\lambda_{j}+n+1-i-j}{n+1-i-j}\right)\left(\prod_{1\leq i<j\leq p} \frac{\lambda_{i}-\lambda_{j}+j-i}{j-i}\right)^{2} $$
when $n=2p+1$. Let us give the details when $n=2p$. Again, one looks at $\rho_{k,l}=A_{n}(\lambda)/A_{n}(\mu)$ and $\eta_{k,l}=\rho_{k,l}\,\E^{-\log n (B_{n}(\lambda)-B_{n}(\mu))}$, with $\mu$ and $\lambda$ equal to
$$(\lambda_{l+1}+k-1,\ldots,\lambda_{l+1}+k-1,\lambda_{l+1},\ldots,\lambda_{p})_{p}\quad \text{and}\quad (\lambda_{l+1}+k,\ldots,\lambda_{l+1}+k,\lambda_{l+1},\ldots,\lambda_{p})_{p}.$$
The quotient of dimensions is
$$\rho_{k,l}=\left(\prod_{j=1}^{l}\frac{(2k'+n-j)(2k'+n-j-1)}{(2k'+n-j-l)(2k'+n-j-l-1)}\right)\left(\prod_{j=l+1}^{p} \frac{(k'-\lambda_{j}+j-1)(k'+\lambda_{j}+n-j)}{(k'-\lambda_{j}+j-l-1)(k'+\lambda_{j}+n-j-l)}\right)^{\!2},$$
and a lower bound is then obtained by the usual replacement $\lambda_{l+1}=\lambda_{j}=0$ and then $k=1$:
$$\rho_{k,l} \leq \frac{n-2l+1}{n+1}\,\binom{n+1}{l}^{\!2}.$$
This leads to the inequality
$$\eta_{k,l} \leq \frac{n-2l+1}{n+1}\,\binom{n+1}{l}^{\!2}\,\E^{-\frac{2l(n+1-l)}{n}\log n}$$
The last quantity is decreasing in $l$, as the quotient of two consecutive terms of parameters $n,l$ and $n,l+1$ is smaller than 
$$\left(\frac{n+1-l}{l+1}\,\E^{-\frac{n-2l}{n}\log n}\right)^{2}\leq 1.$$
So, 
$$\eta_{k,l}\leq \frac{n-1}{n+1}\,(n+1)^{2}\,\E^{-2\log n}=\frac{n^{2}-1}{n^{2}}\leq 1$$
and $A_{n}(\lambda)\,\E^{-\log n\,B_{n}(\lambda)}$ is smaller than $1$ for any partition (we leave to the reader the verification of the other case $n=2p+1$, which is very similar).
\bigskip

\subsection{Expansion of elementary $4$-tensors for unitary groups}\label{expun}
For the eigenvectors associated to the value $2n-2$, we shall write
$$S(i,j,k,l)=(e[i,j]-e[j,i])^{\otimes 2}-(e[j,k]-e[k,j])^{\otimes 2}+(e[k,l]-e[l,k])^{\otimes 2}-(e[l,i]-e[i,l])^{\otimes 2}.$$
The elementary tensor $e_{i}^{\otimes 4}$ is equal to
\begin{align*}
&\frac{1}{n(n+1)}\sum_{k,l=1}^{n} e[k,l,k,l]+e[k,l,l,k]+\frac{1}{n(n+2)}\sum_{k \neq i}\sum_{l=1}^{n}\left(\substack{(e[i,l,i,l]-e[k,l,k,l] )+ (e[l,i,l,i]-e[l,k,l,k])\\
+(e[i,l,l,i]-e[k,l,l,k])+ (e[l,i,i,l]-e[l,k,k,l])}\right)\\
&+\frac{1}{n+2}\sum_{k \neq i} e_{i}^{\otimes 4}+e_{k}^{\otimes 4}-(e[i,k]+e[k,i])^{\otimes 2} - \frac{1}{(n+1)(n+2)}\sum_{k<l } e_{k}^{\otimes 4}+e_{l}^{\otimes 4}-(e[k,l]+e[l,k])^{\otimes 2}
\end{align*}
with the two first terms respectively in $V_{2n^{2}-2}$ and $V_{n^{2}-2}$, and the second line in $V_{-2n-2}$.
Similarly, $e_{i}\otimes e_{j} \otimes e_{i} \otimes e_{j}$ is equal to
\begin{align*}
&\frac{1}{(n-1)(n+1)}\sum_{k=1}^{n}\sum_{l=1}^{n}e[k,l,k,l] - \frac{1}{n(n-1)(n+1)}\sum_{k=1}^{n}\sum_{l=1}^{n}e[k,l,l,k]\\
&\text{- - - - - - - - - - - - - - - - - - - - - - - - - - - - - - - - - - - - - - - - - - - - - - - - - - - - - - - - - - - - -}\\
&+\frac{1}{n(n+2)}\left(\sum_{l=1}^{n}\left(\substack{e[i,l,i,l]-e[j,l,j,l]\\+e[l,j,l,j]-e[l,i,l,i]}\right)\right)+\frac{1}{(n-2)(n+2)}\sum_{k \neq i,j}\left(\sum_{l=1}^{n}\left(\substack{e[i,l,i,l]-e[k,l,k,l]\\+e_[l,j,l,j]-e[l,k,l,k]}\right)\right)\\
&-\frac{1}{n(n-2)(n+2)}\sum_{k \neq i,j} \left(\sum_{l=1}^{n}\left(\substack{e[i,l,l,i]+e[j,l,l,j]-2e[k,l,l,k] \\ + e[l,i,i,l]+e[l,j,j,l]-2e[l,k,k,l] }\right)\right)\\
&\text{- - - - - - - - - - - - - - - - - - - - - - - - - - - - - - - - - - - - - - - - - - - - - - - - - - - - - - - - - - - - -}\\
&+\frac{1}{4(n-1)(n-2)}\sum_{(k<l )\neq i,j}\!\!2S(i,j,k,l)-S(i,k,j,l)\\
&\text{- - - - - - - - - - - - - - - - - - - - - - - - - - - - - - - - - - - - - - - - - - - - - - - - - - - - - - - - - - - - -}\\
&+\frac{1}{2n}\sum_{k\neq i,j}\left(\substack{e[i,j,i,j]+e[j,k,j,k]+e[k,i,k,i] \\ -e[j,i,j,i]-e[k,j,k,j]- e[i,k,i,k]}\right)\\
&\text{- - - - - - - - - - - - - - - - - - - - - - - - - - - - - - - - - - - - - - - - - - - - - - - - - - - - - - - - - - - - -}\\
&+\frac{1}{4(n+2)}\left(\sum_{k \neq i}e_{i}^{\otimes 4}+e_{k}^{\otimes 4}-(e[i,k]+e[k,i])^{\otimes 2}+\sum_{k\neq j}e_{j}^{\otimes 4}+e_{k}^{\otimes 4}-(e[j,k]+e[k,j])^{\otimes 2}\right)\\
&-\frac{1}{4}\left(e_{i}^{\otimes 4}+e_{j}^{\otimes 4}-(e[i,j]+e[j,i])^{\otimes 2}\right)-\frac{1}{2(n+1)(n+2)}\left(\sum_{k<l} e_{k}^{\otimes 4}+e_{l}^{\otimes 4}-(e[k,l]+e[l,k])^{\otimes 2}\right)
\end{align*}
with the parts of this expansion respectively in $V_{2n^{2}-2}$, $V_{n^{2}-2}$, $V_{2n-2}$, $V_{-2}$ and $V_{-2n-2}$.
\bigskip

\subsection{Expansion of elementary $4$-tensors for compact symplectic groups}\label{expspn}
It is a little more tedious than before to find a complete list of ``simple'' eigenvectors of $M_{n,4}$ (or at least a sufficient list to expand simple tensors). The list of possible eigenvalues of $M_{n,4}$ is $\{2n+1,n+1,n,3,1,0,-1,-3\},$
and on the other hand, one can easily identify a basis of $V_{2n+1}$: it consists in the three vectors
\begin{align*}v_{2n+1,1}&=\sum_{i,j=1}^{n}\left(\substack{e[2i-1,2i,2j-1,2j]+e[2i,2i-1,2j,2j-1]\\-e[2i,2i-1,2j-1,2j]-e[2i-1,2i,2j,2j-1]}\right);\\
v_{2n+1,2}&=\sum_{i,j=1}^{n}\left(\substack{e[2i-1,2j-1,2i,2j]+e[2i,2j,2i-1,2j-1]\\-e[2i,2j-1,2i-1,2j]-e[2i-1,2j,2i,2j-1]}\right);\\
v_{2n+1,3}&=\sum_{i,j=1}^{n}\left(\substack{e[2i-1,2j-1,2j,2i]+e[2i,2j,2j-1,2i-1]\\-e[2i,2j-1,2j,2i-1]-e[2i-1,2j,2j-1,2i]}\right).
\end{align*}
But then, it becomes really difficult to describe the other eigenspaces. However, one can still find the eigenvector expansion of simple tensors such as $e_{i}^{\otimes 4}$, $e_{i}^{\otimes 2}e_{j}^{\otimes 2}$, or $e[i,j,k,l]$; hence, in the following, we just give these expansions (again it is easy to check that each part of an expansion is indeed an eigenvector). The tensor $e[i,i,i,i]$ is an eigenvector in $V_{-3}$, and 
on the other hand, $e[2i-1,2i-1,2i,2i]$ decomposes into the eigenvectors
\begin{align*}
&\frac{1}{2n(2n+1)}\left(v_{2n+1,2}+v_{2n+1,3}\right)\\
&\text{- - - - - - - - - - - - - - - - - - - - - - - - - - - - - - - - - - - - - - - - - - - - - - - - - - -}\\
&+\frac{n-2}{4n(n+1)}\sum_{\sigma \in S}\sum_{j \neq i}\left(\substack{e[2i-1,2j-1,2i,2j]+e[2i,2j,2i-1,2j-1]\\-e[2i-1,2j,2i,2j-1]-e[2i,2j-1,2i-1,2j]}\right)^{\sigma}\\
&+\frac{1}{4n(n+1)}\sum_{\sigma \in S}\sum_{j,k\neq i} \left(\substack{e[2j-1,2k,2j,2k-1]+e[2j,2k-1,2j-1,2k]\\-e[2j-1,2k-1,2j,2k]-e[2j,2k,2j-1,2k-1]}\right)^{\sigma}\\
&+\frac{n-1}{2n(n+1)}\left(\substack{2e[2i-1,2i-1,2i,2i]+2e[2i,2i,2i-1,2i-1]-e[2i-1,2i,2i-1,2i]\\-e[2i,2i-1,2i,2i-1]-e[2i,2i-1,2i-1,2i]-e[2i-1,2i,2i,2i-1]}\right)\\
&\text{- - - - - - - - - - - - - - - - - - - - - - - - - - - - - - - - - - - - - - - - - - - - - - - - - - -}\\
&+\frac{1}{4(n+1)}\sum_{\sigma \in S}\sum_{j=1}^{n}\left(\substack{e[2i-1,2j-1,2i,2j]+e[2i,2j-1,2i-1,2j]\\-e[2i-1,2j,2i,2j-1]-e[2i,2j,2i-1,2j-1]}\right)^{\sigma} \\
&\text{- - - - - - - - - - - - - - - - - - - - - - - - - - - - - - - - - - - - - - - - - - - - - - - - - - -}\end{align*}
\begin{align*}
&+\frac{2n-1}{2(2n+1)(2n+2)}\sum_{\sigma \in S}\sum_{j \neq i} \left(\substack{e[2i-1,2j,2j-1,2i]+e[2i,2j-1,2j,2i-1]\\-e[2i-1,2j-1,2j,2i]-e[2i,2j,2j-1,2i-1]}\right)^{\sigma}\\
&+\frac{1}{2(2n+1)(2n+2)}\sum_{\sigma \in S}\sum_{j, k \neq i} \left(\substack{e[2j-1,2k-1,2j,2k]+e[2j,2k,2j-1,2k-1]\\-e[2j-1,2k,2j,2k-1]-e[2j,2k-1,2j-1,2k]}\right)^{\sigma}\\
&+\frac{(2n-1)(2n-2)}{6(2n+1)(2n+2)}\left(\substack{2e[2i-1,2i-1,2i,2i]+2e[2i,2i,2i-1,2i-1]-e[2i-1,2i,2i-1,2i]\\-e[2i,2i-1,2i,2i-1]-e[2i,2i-1,2i-1,2i]-e[2i-1,2i,2i,2i-1]}\right)\\
&\text{- - - - - - - - - - - - - - - - - - - - - - - - - - - - - - - - - - - - - - - - - - - - - - - - - - -}\\
&+\frac{n-1}{2(n+1)}\left(e[2i-1,2i-1,2i,2i]-e[2i,2i,2i-1,2i-1]\right) \\
&+\frac{1}{4(n+1)}\sum_{\sigma \in S}\sum_{j \neq i} \left(\substack{e[2i-1,2j,2j-1,2i]+e[2i,2j,2j-1,2i-1]\\-e[2i-1,2j-1,2j,2i]-e[2i,2j-1,2j,2i-1]}\right)^{\sigma}\\
&\text{- - - - - - - - - - - - - - - - - - - - - - - - - - - - - - - - - - - - - - - - - - - - - - - - - - -}\\
&+\frac{1}{6}\sum_{\sigma \in \sym_{4}}'e[2i-1,2i-1,2i,2i]^{\sigma}
\end{align*}
with the parts of this expansion respectively in $V_{2n+1}$, $V_{n+1}$, $V_{n}$, $V_{0}$, $V_{-1}$, and $V_{-3}$. In these expansions, $S=(\Z/2\Z)^{2}$ denotes the group of permutations $\{\id,(1,2),(3,4),(1,2)(3,4)\}$.
\bigskip

The expansion in eigenvectors of $e[2i,2i,2j,2j]$ is
\begin{align*}
&\frac{1}{6}\left(\substack{2e[2i,2i,2j,2j]+2e[2j,2j,2i,2i]-e[2i,2j,2i,2j] \\ -e[2j,2i,2j,2i]-e[2i,2j,2j,2i]-e[2j,2i,2i,2j]}\right)+\frac{1}{2}\left(\substack{e[2i,2i,2j,2j]\\-e[2j,2j,2i,2i]}\right) +\frac{1}{6}\sum_{\sigma \in \sym_{4}}'e[2i,2i,2j,2j]^{\sigma}
\end{align*}
with each part respectively in $V_{0}$, $V_{-1}$ and $V_{-3}$; and similarly for the expansions of $e[2i-1,2i-1,2j,2j]$ or $e[2i-1,2i-1,2j-1,2j-1]$. 
Finally, we skip the expansion in eigenvectors of $e[2i-1,2i,2j-1,2j]$ as it is two pages long.
\bigskip\bigskip

\bibliographystyle{alpha}
\bibliography{cutoff}

\begin{thebibliography}{CSST08}

\bibitem[AD86]{AD86}
D.~Aldous and P.~Diaconis.
\newblock Shuffling cards and stopping times.
\newblock {\em Amer. Math. Monthly}, 93(5):333--348, 1986.

\bibitem[App11]{Apple11}
D.~Applebaum.
\newblock Infinitely divisible central probability measures on compact {L}ie
  groups --- {R}egularity, semigroups and transition kernels.
\newblock {\em Ann. Probab.}, 39(6):2474--2496, 2011.

\bibitem[BD85]{BD85}
T.~Br\"{o}cker and T.~Dieck.
\newblock {\em Representations of Compact {L}ie Groups}, volume~98 of {\em
  Graduate Texts in Mathematics}.
\newblock Springer-Verlag, 1985.

\bibitem[BD92]{BD92}
D.~Bayer and P.~Diaconis.
\newblock Trailing the dovetail shuffle to its lair.
\newblock {\em Ann. Appl. Prob.}, 2(2):294--313, 1992.

\bibitem[CSC08]{CSC08}
G.-Y. Chen and L.~Saloff-Coste.
\newblock The cutoff phenomenon for ergodic {M}arkov processes.
\newblock {\em Electronic J. Probab.}, 13(3):26--78, 2008.

\bibitem[CSST08]{CSST}
T.~Ceccherini-Silberstein, F.~Scarabotti, and F.~Tolli.
\newblock {\em Harmonic Analysis on Finite Groups}, volume 108 of {\em
  Cambridge studies in advanced mathematics}.
\newblock Cambridge University Press, 2008.

\bibitem[Dia96]{Dia96}
P.~Diaconis.
\newblock The cutoff phenomenon in finite {M}arkov chains.
\newblock {\em Proc. Nat. Acad. Sci. of U.S.A.}, 93(4):1659--1664, 1996.

\bibitem[DS94]{DS94}
P.~Diaconis and M.~Shahshahani.
\newblock On the eigenvalues of random matrices.
\newblock {\em Journal of Applied Probability}, 31:49--62, 1994.

\bibitem[DSC96]{DSC96}
P.~Diaconis and L.~Saloff-Coste.
\newblock Random walks on finite groups: A survey of analytic techniques.
\newblock In H.~Heyer, editor, {\em Prob. Meas. on Groups XI}, pages 44--75.
  World Scientific Singapore, 1996.

\bibitem[Far08]{Far08}
J.~Faraut.
\newblock {\em Analysis on Lie Groups: An Introduction}, volume 110 of {\em
  Cambridge Studies in Advanced Mathematics}.
\newblock Cambridge University Press, 2008.

\bibitem[FH91]{FH91}
W.~Fulton and J.~Harris.
\newblock {\em Representation theory}, volume 129 of {\em Graduate Texts in
  Mathematics}.
\newblock Springer-Verlag, 1991.

\bibitem[Gas70]{Gas70}
G.~Gasper.
\newblock Linearization of the product of {J}acobi polynomials. {I} and {II}.
\newblock {\em Can. J. Math.}, 22:171--175{;\,\,}582--593, 1970.

\bibitem[GW09]{GW09}
R.~Goodman and N.~R. Wallach.
\newblock {\em Symmetry, Representations, and Invariants}, volume 255 of {\em
  Graduate Texts in Mathematics}.
\newblock Springer-Verlag, 2009.

\bibitem[Hel78]{Hel78}
S.~Helgason.
\newblock {\em Differential Geometry, Lie Groups, and Symmetric Spaces}.
\newblock Academic Press, 1978.

\bibitem[Hel84]{Hel84}
S.~Helgason.
\newblock {\em Groups and Geometric Analysis. Integral Geometry, Invariant
  Differential Operators, and Spherical Functions}.
\newblock Academic Press, 1984.

\bibitem[HS94]{HS94}
G.~Heckman and H.~Schlichtkrull.
\newblock {\em Harmonic Analysis and Special Functions on Symmetric Spaces},
  volume~16 of {\em Perspectives in Mathematics}.
\newblock Academic Press, 1994.

\bibitem[Hun56]{Hun56}
G.~A. Hunt.
\newblock Semigroups of measures on {L}ie groups.
\newblock {\em Trans. Am. Math}, 81:264--293, 1956.

\bibitem[L{\'e}v11]{Lev11}
T.~L{\'e}vy.
\newblock Asymptotics of {B}rownian motions on classical {L}ie groups, the
  master field on the plane, and the {M}akeenko-{M}igdal equations.
\newblock \texttt{arXiv:1112.2452v1 [math-ph]}, 2011.

\bibitem[Lia04a]{Liao04paper}
M.~Liao.
\newblock L\'evy processes and {F}ourier analysis on compact {L}ie groups.
\newblock {\em Ann. Probab.}, 32(2):1553--1573, 2004.

\bibitem[Lia04b]{Liao04book}
M.~Liao.
\newblock {\em L\'evy processes in {L}ie groups}, volume 162 of {\em Cambridge
  Tracts in Mathematics}.
\newblock Cambridge University Press, 2004.

\bibitem[Por96a]{Por96a}
U.~Porod.
\newblock The cut-off phenomenon for random reflections.
\newblock {\em Ann. Probab.}, 24(1):74--96, 1996.

\bibitem[Por96b]{Por96b}
U.~Porod.
\newblock The cut-off phenomenon for random reflections {II:} complex and
  quaternionic cases.
\newblock {\em Probab. Th. Rel. Fields}, 104(2):181--209, 1996.

\bibitem[Ros94]{Ros94}
J.~S. Rosenthal.
\newblock Random rotations: characters and random walks on {$\mathrm{SO}(N)$}.
\newblock {\em Ann. Probab.}, 22:398--423, 1994.

\bibitem[SC94]{SC94}
L.~Saloff-Coste.
\newblock Precise estimates on the rate at which certain diffusions tend to
  equilibrium.
\newblock {\em Math. Zeitschrift}, 217:641--677, 1994.

\bibitem[SC04]{SC04}
L.~Saloff-Coste.
\newblock On the convergence to equilibrium of {B}rownian motion on compact
  simple {L}ie groups.
\newblock {\em J. Geometric Analysis}, 14(4):715--733, 2004.

\bibitem[Var89]{Var89}
V.~S. Varadarajan.
\newblock {\em An Introduction to Harmonic Analysis on Semisimple Lie Groups},
  volume~16 of {\em Cambridge Studies in Advanced Mathematics}.
\newblock Cambridge University Press, 1989.

\end{thebibliography}

\end{document}